\documentclass[abstract=true]{scrartcl}
\usepackage{amsthm,amsmath,amssymb}
\usepackage{mathtools}
\usepackage{authblk}
\usepackage{hyperref}

\usepackage[utf8]{inputenc}
\usepackage{enumitem}
\usepackage{scalerel}

\usepackage{subcaption}
\usepackage{tikz}
\usepackage{pgfplots}
\pgfplotsset{compat=1.15}
\usepackage{pgfmath}
\usetikzlibrary{arrows.meta}
\usetikzlibrary{decorations.markings}
\usetikzlibrary{shapes.geometric}
\usetikzlibrary{calc}
\usetikzlibrary{fillbetween}

\definecolor{gen0}{rgb}{1,0,0}
\definecolor{gen1}{rgb}{0,0,1}
\definecolor{gen2}{rgb}{0,0.7,0}

\tikzstyle{virtual}=[postaction={draw=#1,dash pattern= on 3pt off 7pt,dash phase=5pt,thick},dash pattern= on 3pt off 7pt]

\tikzstyle{walkarrow}=[very thick,decoration={markings, mark=at position 0.5 with {\arrow[xshift=5pt]{latex}}}] 

\newtheorem{thm}{Theorem}[section]
\newtheorem{pro}[thm]{Proposition}
\newtheorem{lem}[thm]{Lemma}
\newtheorem{cor}[thm]{Corollary}

\theoremstyle{definition}
\newtheorem{exa}[thm]{Example}
\newtheorem{rmk}[thm]{Remark}

\newcommand{\abs}[1]{\left\lvert#1\right\rvert}
\newcommand{\norm}[1]{\left\lVert#1\right\rVert}

\newcommand\Gra{\mathbf{G}}
\newcommand\Si{\mathbf{\Sigma}}
\newcommand\Non{\mathbf{N}}
\newcommand\Pro{\mathbf{P}}
\newcommand\Ccal{\mathcal{C}}
\newcommand\Tcal{\mathcal{T}}
\newcommand\Gcal{\mathcal{G}}
\newcommand\Vcal{\mathcal{V}}
\newcommand\Ecal{\mathcal{E}}
\newcommand\Pcal{\mathcal{P}}
\newcommand\N{\mathbb{N}}

\newcommand\Z{\mathbb{Z}}
\newcommand\da{{\scalerel*{\downarrow}{X}}}
\newcommand\ua{{\scalerel*{\uparrow}{X}}}
\newcommand\AUT{\mathrm{AUT}}
\newcommand\Cay{\mathrm{Cay}}
\newcommand\SAW{\mathrm{SAW}}

\title{Self-avoiding walks and multiple context-free languages}
\author{Florian Lehner\thanks{F.\ Lehner was supported by FWF (Austrian Science Fund) projects P31889-N35 and J3850-N32}\; and Christian Lindorfer\thanks{C.\ Lindorfer was partially supported by FWF (Austrian Science Fund) projects P31889-N35 and DK~W1230.}}

\affil{\normalsize Institut f\"ur Diskrete Mathematik, 
Technische Universit\"at Graz,
Austria}

\date{\today} 

\begin{document}
\maketitle

\begin{abstract}
Let $G$ be a quasi-transitive, locally finite, connected graph rooted at a vertex $o$, and let $c_n(o)$ be the number of self-avoiding walks of length $n$ on $G$ starting at $o$. We show that if $G$ has only thin ends, then the generating function $F_{\SAW,o}(z)=\sum_{n \geq 1} c_n(o) z^n$ is an algebraic function. In particular, the connective constant of such a graph is an algebraic number.

If $G$ is deterministically edge-labelled, that is, every (directed) edge carries a label such that no two edges starting at the same vertex have the same label, then the set of all words which can be read along the edges of self-avoiding walks starting at $o$ forms a language denoted by $L_{\SAW,o}$. Assume that the group of label-preserving graph automorphisms acts quasi-transitively. We show that $L_{\SAW,o}$ is a $k$-multiple context-free language if and only if the size of all ends of $G$ is at most $2k$. Applied to Cayley graphs of finitely generated groups this says that $L_{\SAW,o}$ is multiple context-free if and only if the group is virtually free.
\end{abstract}

\section{Introduction}

A \emph{walk} in a graph is an alternating sequence $(v_0,e_1,v_1, \dots, e_n,v_n)$ of vertices $v_i$ and edges $e_i$ such that $e_i$ starts at $v_{i-1}$ and ends at $v_i$ for every $i$. It is called \emph{self-avoiding} (or SAW), if its vertices are pairwise different. This notion was introduced in 1953 by the chemist Flory~\cite{CR1000000} as a model for long-chain polymer molecules and has since attracted considerable interest. A lot of research has focused on lattices, see for instance the monograph by Madras and Slade~\cite{MR2986656} and also the lecture notes by Bauerschmidt et al.~\cite{MR3025395}, but other graphs have also been investigated, see for example the survey of Grimmett and Li \cite{zbMATH07217533}. 

One of the most fundamental problems related to this model is determining (exactly or asymptotically) the number of SAWs of a given length. Denote by $c_n(o)$ the number of SAWs with $n$ edges starting at a fixed root vertex $o$. Hammersley~\cite{MR0091568} showed that the limit
\[
\mu(G)=\lim_{n \rightarrow \infty} c_n(o)^{1/n}
\]
exists for quasi-transitive graphs, that is, graphs which allow a group action by graph automorphisms with finitely many orbits on the vertex set.  Moreover, the value of $\mu(G)$ is independent of the choice of $o$. This number $\mu(G)$ is called the \emph{connective constant} of the graph $G$. Note that by the Cauchy--Hadamard theorem, $\mu(G)$ is the reciprocal of the radius of convergence of the \emph{SAW-generating function}
\[
    F_{\SAW,o}(z)= \sum_{n=1}^\infty c_n(o)z^n.
\]

Explicit computation of $\mu(G)$ can be a very challenging task, even in seemingly harmless instances such as two-dimensional lattices. For example, despite very precise numerical estimates, the precise value of $\mu(\Z^2)$ remains elusive; in fact, it is not even known whether $\mu(\Z^2)$ is an algebraic number. In light of this, it is not surprising that the celebrated paper \cite{MR2912714} by Duminil-Copin and Smirnov containing the first rigorous calculation of the connective constant of the hexagonal lattice is considered a milestone in the theory.

The special case of SAWs on one-dimensional lattices turns out to be a lot more manageable. The reason for this is that the large scale structure of one-dimensional lattices resembles a line: they can be decomposed into infinitely many pairwise isomorphic finite parts such that each part only intersects with two others (its predecessor and its successor). By considering restrictions of SAWs to single parts and analysing how these restrictions fit together, Alm and Janson~\cite{MR1047969} showed that SAW-generating functions on these lattices are always rational. While this approach fails for higher-dimensional lattices, analogous techniques have been successfully applied to other graph classes, in particular for graphs exhibiting some kind of large scale tree structure, see for instance \cite{MR3584819} and \cite{Lindorfer2020}.

In the present paper, we will follow a similar approach to study SAWs on graphs all of whose ends are thin. Recall that an \emph{end} of a graph $G$ is an equivalence class of rays (one-way infinite paths) in $G$, where two rays are equivalent if and only if $G$ contains a third ray intersecting both of them infinitely often. The \emph{size} of an end is the maximum number of pairwise disjoint rays it contains; ends are \emph{thin} if they have finite size, otherwise they are \emph{thick}. 

Our first main result concerns the SAW-generating function of graphs all of whose ends are thin.

\begin{thm} \label{thm:main-1}
Let $G$ be a simple, locally finite, connected, quasi-transitive graph having only thin ends and let $o \in V(G)$. Then $F_{\SAW,o}$ is algebraic over $\mathbb{Q}$. In particular the connective constant $\mu(G)$ is an algebraic number.
\end{thm}

Our second main result connects SAWs to formal languages. For this purpose consider the following setup. Let $G$ be deterministically edge-labelled, that is, every (directed) edge $e$ of $G$ is assigned a label $\ell(e)$ from some given alphabet $\Si$ such that different edges with the same initial vertex have different labels. As before, edge-labelled graphs are assumed to be quasi-transitive, that is, the group $\AUT(G,\ell)$ of automorphisms of $G$ which preserve $\ell$ acts with finitely many orbits on $G$. Important examples of transitive deterministically edge-labelled graphs are the Cayley graphs of finitely generated groups.

The edge-labelling is extended to walks $p=(v_0,e_1,v_1, \dots, e_n,v_n)$ by setting
\[
\ell(p)= \ell(e_1) \dots \ell(e_n).
\]
In this way, any set $\Pcal$ of walks gives rise to a language $L(\Pcal) = \{\ell(p) \colon p \in \Pcal\}$ thus allowing us to study properties of $\Pcal$ via properties of the corresponding language. 

This is particularly fruitful when $L(\Pcal)$ belongs to a well understood family of languages. Besides the well-known classes of regular and context-free languages, the class of multiple context free languages (MCFLs) plays an important role. These were introduced by Seki et al.~\cite{MR1131066} as a generalisation of context free languages capable of modelling cross-serial dependencies which occur in some natural languages such as Swiss German. A concise definition of MCFLs will be given in Section~\ref{sec:basics}; for now we only mention that they share many useful traits with context-free languages, including polynomial time parsability, semi-linearity and the closure properties of being a full AFL. 

While MCFLs may seem artificial at first, they appear in some natural problems. One of them is the well known word problem, which is equivalent to recognising closed walks in a given Cayley graph. Anisimov~\cite{MR301981} showed that the arising language is regular if and only if the underlying group is finite, and Muller and Schupp~\cite{MR710250} showed that it is context-free if and only if the group is virtually free. The word problem on $\Z^d$ thus is not context free, but it was shown to be multiple context-free in seminal work by Salvati~\cite{MR3354791} for $d=2$, and this result has since been extended by Ho~\cite{MR3794915} to all positive integers $d$.

We are interested in the \emph{language of self-avoiding walks} defined by 
\[
L_{\SAW,o}=L(\Pcal_{\SAW,o})=\{\ell(p) \colon p \in \Pcal_{\SAW,o}\},
\]
where $\Pcal_{\SAW,o}$ is the set of all SAWs of length at least 1 on $G$ starting at $o$.
In his computation of $F_{\SAW,o}$ for the infinite ladder graph Zeilberger \cite{MR1406737} implicitly used that $L_{\SAW,o}$ for this graph is context free. More generally, Lindorfer and Woess~\cite{Lindorfer2020} showed that $L_{\SAW,o}$ on a locally finite, connected, quasi-transitive deterministically edge-labelled graph $G$ is regular if and only if all ends of $G$ have size 1, and that it is context-free if and only if all ends have size at most 2. In both of these cases $F_{\SAW,o}$ can be computed using an appropriate grammar generating $L_{\SAW,o}$. Our second main result generalises these results.

\begin{thm}\label{thm:main-2x}
Let $G$ be a simple, locally finite, connected, quasi-transitive deterministically edge-labelled graph and let $o \in V(G)$. Then $L_{\SAW,o}$ is an MCFL if and only if all ends of $G$ are thin.
\end{thm}

In fact, what we prove is slightly stronger. Every MCFL can be assigned a rank (see Section~\ref{sec:languages} for details); an MCFL is called $k$-multiple context free if its rank is at most $k$. It is worth noting that the families of $k$-MCFLs form a strictly increasing hierarchy, and that $1$-MCFLs are exactly the context free languages. We show that the maximal size of an end of $G$ tells us exactly where $L_{\SAW,o}$ lies in this hierarchy.

\begin{thm}\label{thm:main-2}
Let $G$ be a simple, locally finite, connected, deterministically edge-labelled quasi-transitive graph and let $o \in V(G)$. Then $L_{\SAW,o}$ is $k$-multiple context-free if and only if every end of $G$ has size at most $2k$.
\end{thm}

Applied to Cayley graphs of groups, Theorem~\ref{thm:main-2x} states that the language of self-avoiding walks on a Cayley graph of a group is multiple context-free if and only if the group is virtually free. In particular, the property of having a multiple context-free language of self-avoiding walks is a group invariant.

As mentioned above, we are following a similar approach as Alm and Janson in~\cite{MR1047969}. There are two key ingredients to this approach: firstly, decomposing the graph into finite parts, and secondly, analysing the restrictions of self-avoiding walks to these parts.

The decomposition into finite parts is formalised by the notion of tree decompositions which will be the subject of Section~\ref{sec:treedecomp}. Roughly speaking these are decompositions of a graph into parts which intersect in a tree-like manner. This notion was introduced by Halin~\cite{MR444522} in 1976; later it was rediscovered by Robertson and Seymour~\cite{MR742386} and plays a central role in the proof of the celebrated Graph Minor Theorem. For our applications it is crucial that the tree decompositions are invariant under some quasi-transitive group of automorphisms. Such tree decompositions have been constructed by Dunwoody and Krön~\cite{MR3385727}, inspired by a similar construction based on edge cuts introduced by Dunwoody in~\cite{MR671142}. 

The restriction of self avoiding walks to the parts of a tree decomposition is captured by the notion of configurations introduced and studied in Section~\ref{sec:configs}. Among other things we show that there is a bijection between SAWs and a specific class of configurations called \emph{bounded consistent configurations}, in other words, any SAW can be obtained by piecing together appropriate configurations. 

This bijection is central to the rest of the paper because it allows us to work with configurations rather than self avoiding walks; this turns out to be beneficial since configurations (unlike SAWs) carry a recursive structure.
In Section~\ref{sec:configgrammar} we use this recursive structure to show that the set of bounded consistent configurations is in bijection with a context free language; Theorem~\ref{thm:main-1} follows from this fact. In Section~\ref{sec:mcflofsaws}, again using the recursive structure, we show that $L_{\SAW,o}$ is an MCFL, thus proving the first half of Theorem~\ref{thm:main-2}. Finally, we combine techniques from \cite{Lindorfer2020} with a result from \cite{lehner2020comparing} to complete the proof of Theorem~\ref{thm:main-2}.

\section{Basic background}\label{sec:basics}

Throughout this paper, we denote by $\N$ the set of natural numbers starting at $1$ and by $[n]$ the set $\{k \in \N : k \leq n\}$.

\subsection{Graph theory}

A \emph{graph} $G$ consists of a set $V(G)$ of vertices and a set $E(G)$ of edges. Every edge $e \in E(G)$ starts at its \emph{initial vertex} $e^- \in V(G)$ and ends at its \emph{terminal vertex} $e^+ \in V(G)$. We do not allow loops, so the two endpoints $e^-$ and $e^+$ of every edge $e \in E(G)$ are different. Furthermore all graphs considered are \emph{undirected}, so all edges appear in pairs $e, \bar{e}$ having the same endpoints but different direction. In other words, for $e \in E(G)$ the edge $\bar{e} \in E$ satisfies $\bar{e}^-=e^+$ and $\bar{e}^+=e^-$ and $\bar{\bar{e}}=e$. When drawing a graph, every pair $(e,\bar{e})$ will usually be represented by and thought of as one undirected edge. A graph is called \emph{simple}, if it contains no \emph{multiple edges}, or in other words, if every edge $e$ is uniquely defined by the pair $(e^-,e^+)$ of its initial and terminal vertex. We sometimes abuse notation and write $e=e^-e^+$; if $G$ is not simple we still use this notation, but will include further information needed to identify $e$ among the edges with the same initial and terminal vertices if necessary. The \emph{degree} $\deg(v)$ of a vertex $v$ is the number of outgoing edges of $v$. The graph $G$ is called \emph{locally finite}, if all vertices have finite degrees.

A \emph{walk} in a graph is an alternating sequence $p=(v_0,e_1,v_1, \dots, e_n,v_n)$ of vertices $v_i \in V(G)$ and edges $e_i \in E(G)$ such that $e_{i}^-=v_{i-1}$ and $e_i^+=v_i$ for every $i \in [n]$. Its \emph{length} is the number $n$ of edges and its \emph{initial} and \emph{terminal vertices} are $p^-=v_0$ and $p^+=v_n$, respectively. This comprises the \emph{trivial walk} $(v)$ of length 0, starting (and ending) at a vertex $v$ and also the \emph{empty walk} $\emptyset$ consisting of no vertices and no edges. A walk $p$ is called \emph{self-avoiding} or a \emph{SAW}, if the vertices in $p$ are pairwise different.
For two vertices $v$ and $w$ of $p$ we write $v p w$ for the maximal sub-walk of $p$ starting at $v$ and ending at $w$. If $v=v_0$ or $w=v_n$ we omit the corresponding vertex and denote the sub-walk by $p w$ or $v p$, respectively. We extend this notation even further and denote for walks $p_1, \dots, p_n$ and vertices $v_0, \dots, v_n$ in the respective walks the concatenation $(v_0 p_1 v_1)(v_1 p_2 v_2) \dots (v_{n-1} p_n v_n)$ of the sub-walks $v_{i-1} p_i v_i$ by $v_0 p_1 v_1 p_2 \dots p_n v_n$. If the terminal vertex $v$ of $p_1$ coincides with the initial vertex of $p_2$, we write $p_1p_2$ instead of $p_1vp_2$, and similarly for concatenations of multiple walks. If $e$ is an edge connecting the terminal vertex $v_1$ of $p_1$ to the initial vertex $v_2$ of $p_2$, then we write $p_1ep_2$ instead of $p_1v_1(v_1,e,v_2)v_2p_2$, and similarly for concatenations with more parts. A walk is \emph{closed} if its initial and terminal vertex coincide. A closed walk $p$ of length at least 3 such that $v_1 p$ is self-avoiding is called \emph{cycle}.

A \emph{multi-walk} $p$ is a sequence of vertices and edges obtained by stringing together the sequences of vertices and edges corresponding to walks $p_1, \dots, p_k$; the $p_i$ are called the \emph{walk components} of $p$. In other words, a multi-walk is a sequence of vertices and edges, such that every edge in the sequence is preceded by its initial vertex and succeeded by its terminal vertex. Note that each of the walks $p_i$ is a sequence starting and ending with a vertex, so that the final vertex of $v_i$ and the initial vertex of $v_{i+1}$ will appear next to each other in the sequence $p$. In fact every appearance of two consecutive vertices in a multi-walk marks the start of a new walk component.

For a (multi-)walk $p$ on $G$ and $A \subset V(G) \cup E(G)$ we denote by $p \cap A$ the sequence obtained from $p$ by deleting all elements not in $A$ and by $p-A$ the sequence obtained by deleting all elements of $A$. For a sub-graph $H$ of $G$ we write $p \cap H$ for the sequence $p \cap (V(H) \cup E(H))$. In general the sequences $p \cap A$ and $p - A$ need not be multi-walks, but we note that $p-E$ is a multi-walk for $E \subseteq E(G)$, as is $p \cap H$ for any sub-graph $H$ of $G$. 

A graph $G$ is \emph{connected} if any two vertices $v,w$ are the endpoints of some walk $p$ in $G$. \emph{Components} of $G$ are maximal connected sub-graphs. 

A \emph{tree} is a connected and cycle-free graph. A \emph{rooted tree} is a tree in which one vertex has been designated the \emph{root}. For vertices $u$ and $v$ of a rooted tree we say that $u$ is an \emph{ancestor} of $v$ and $v$ is a \emph{descendant} of $u$ if any walk from $r$ to $v$ contains $u$. The unique ancestor of $v$ which is a neighbour of $v$ is called its \emph{parent} and denoted by $v^\ua$, descendants of $v$ in the neighbourhood of $v$ are called its \emph{children}. The \emph{forefather} of a set $A \subseteq V(T)$ is the unique common ancestor of all vertices in $A$ none of whose children are ancestors of all vertices in $A$. The \emph{cone} at a vertex $v$ in a rooted tree, denoted by $K_v$, is the subtree induced by $v$ and its descendants. An \emph{ordered tree} is a rooted tree in which an ordering is specified for the children of each vertex; in this case we will denote the $i$-th child of a vertex $v$ with respect to this order by $v^{\da}_i$. 

A tree consisting only of vertices of degree at most 2 is called \emph{path}. We point out that unlike walks, paths are graphs and have no direction; a finite path can be seen as the support of a SAW. Given two disjoint subsets $A$ and $B$ of vertices of a graph $G$, an \emph{$A$--$B$-path} on $G$ is a sub-graph of $G$ which is a finite path intersecting $A$ and $B$ only in its two endpoints. A \emph{ray} is a one-way infinite path and a \emph{double ray} is a two-way infinite path. 

For any set $K \subseteq V(G)$ we denote by $G-K$ the sub-graph obtained from $G$ by removing $K$ and all edges incident to $K$. If removing $K$ disconnects $G$, then $K$ is called a \emph{separating set}. Furthermore, we denote by $G[K]$ the sub-graph of $G$ \emph{induced} by $K$, that is the graph $G-(V(G)\setminus K)$. 

Two rays in a graph $G$ are called \emph{equivalent}, if for every finite set $K \subseteq V(G)$ they end up in the same component of $G-K$, that is, all but finitely many of their vertices are contained in that component. An \emph{end} of $G$ is an equivalence class of rays with respect to this equivalence relation. We say that two ends $\omega_1$ and $\omega_2$ of a graph $G$ are \emph{separated} by $K$ if any two rays $R_1 \in \omega_1$ and $R_2 \in \omega_2$ end up in different components of $G-K$. Halin~\cite{MR190031} showed that an end containing arbitrarily many disjoint rays must contain an infinite family of disjoint rays, hence the maximum number of disjoint rays contained in an end $\omega$ is well defined and lies in $\N \cup \{\infty\}$. This number is called the \emph{size} of the end $\omega$. An end of finite size is called \emph{thin}, an end of infinite size is called \emph{thick}.

An \emph{automorphism} $\gamma$ of a graph $G$ is a permutation of $V(G)$ preserving the neighbourhood relation in $G$. The set of all automorphisms of $G$ forms a group which is called the \emph{automorphism group} of $G$ and denoted by $\AUT(G)$. For a subgroup $\Gamma \subseteq \AUT(G)$ we can define an equivalence relation on $V(G)$ by $u \sim v \iff \exists \gamma \in \Gamma\colon u = \gamma v$. The equivalence classes with respect to this relation are called \emph{orbits} and denoted by $\Gamma v$. We say that $\Gamma$ acts \emph{transitively}, if there is exactly one orbit, and that it acts \emph{quasi-transitively}, if there are only finitely many orbits. In this case the graph $G$ is also called \emph{(quasi-)transitive}. A sub-graph $H$ of $G$ is called $\gamma$-invariant for $\gamma \in \AUT(G)$ if $\gamma(H)=H$.

It is well known, that any infinite, locally finite, connected graph which is quasi-transitive has either one, two, or infinitely many ends. If it has one end, this end is thick. If it has two ends, both are thin and must have the same size. Finally, if it has infinitely many ends, then it must have thin ends. These and many more results were given by Halin in~\cite{MR335368}.

The action of the automorphism group $\AUT(G)$ of a locally finite, connected graph extends in the obvious way to its ends. An automorphism $\gamma \in \AUT(G)$ is called
\begin{enumerate}[label=(\roman*)]
    \item \emph{elliptic}, if its fixes a finite subset of $V(G)$,
    \item \emph{parabolic}, if it is not elliptic and fixes a unique end of $G$, and
    \item \emph{hyperbolic}, if it is not elliptic and fixes each of a unique pair of ends of $G$.
\end{enumerate}

Halin~\cite{MR335368} showed that any graph automorphism is either elliptic, or parabolic, or hyperbolic, and additionally, that these different types of automorphisms have the following properties. Firstly, $\gamma$ is elliptic if and only if for some ($\iff$ every) vertex $v$ of $G$ the sequence $v, \gamma v, \gamma^2 v, \dots$ contains only finitely many different vertices (and is periodic). Secondly, if $\gamma$ is hyperbolic then the two ends fixed by $\gamma$ have the same finite size $k$ and $G$ contains $k$ disjoint double rays, which are invariant under $\gamma^n$ for some $n \in \N$. Finally, if $\gamma$ is parabolic then the unique end fixed by $\gamma$ is thick and $G$ contains infinitely many double rays invariant under $\gamma^n$ for some $n \in \N$.

Recall that a \emph{quasi-isometry} between two metric spaces $(X, d_X)$ and $(Y, d_Y)$ is a mapping $\varphi: X \to Y$ such that there are constants $A > 0$, and $B, B' \ge 0$ such that for all
$x_1\,,x_2 \in X$ and $y \in Y$,
\[
A^{-1}d_X(x_1,x_2) - B \le d_Y(\varphi x_1,\varphi x_2) \le  A d_X(x_1,x_2) + B \quad \text{and} \quad  d_Y(y, \varphi X) \le B'.
\]
Two connected graphs $G$ and $H$ are called \emph{quasi-isometric}, if the corresponding metric spaces $(G,d_G)$ and $(H, d_H)$ are quasi-isometric, where $d_G$ and $d_H$ denote the standard graph distance in $G$ and $H$, respectively. Every quasi-isometry $\varphi$ has a quasi-inverse $\psi: Y \rightarrow X$, which is a quasi-isometry such that $\psi \varphi$ and $\varphi \psi$ are at bounded distance from the respective identity mappings. In particular being quasi-isometric is an equivalence relation. It is well known that any quasi-isometry between graphs $G$ and $H$ can be extended to the ends of $G$ and that this extension maps thick ends to thick ends and thin ends to thin ends, see for example \cite[Lemma 21.4]{MR1743100}.

An \emph{edge-labelled} graph is a graph $G$ together with a \emph{label function} $\ell$ assigning to every edge $e \in E(G)$ an element of some finite set $\Si$, called \emph{label alphabet}. The labelling is called \emph{deterministic}, if any two edges $e$ and $f$ starting at the same vertex $e^-=f^-$ have different labels. For quasi-transitive graphs, we would like the edge-labelling to be compatible with the action of a quasi-transitive subgroup of $\AUT(G)$. To this end, we denote by $\AUT(G,\ell)$ the group of \emph{label-preserving} graph automorphisms; when speaking of a quasi-transitive edge-labelled graph it will be implicitly assumed that $\AUT(G,\ell)$ acts quasi-transitively. Note that in the case of a deterministic labelling $\ell$, the group $\AUT(G,\ell)$ acts freely on $G$, that is, the identity in $\AUT(G,\ell)$ is the only element fixing a vertex of $G$.

One well-known class of simple, connected, locally finite, transitive graphs that come with a natural deterministic edge-labelling are Cayley graphs of finitely generated groups. Starting with a symmetric generating set $S$ of a group $\Gamma$, the \emph{Cayley graph} $G=\Cay(\Gamma,S)$ has vertex set $V(G)=\Gamma$. We choose $\Si=S$, and for each $\gamma \in \Gamma$ and $s \in S$ there is a directed edge from $\gamma$ to $\gamma s$ with label $s$. The left regular action of $\Gamma$ on itself extends to an action on $G$ by label preserving automorphisms; in fact, it is not hard to see that $\Gamma = \AUT(G,\ell)$.

We finish this section by giving a simple example of a Cayley graph which will be used as a running example to demonstrate various constructions throughout this paper. Let $C_n$ denote the cyclic group of order $n$ and consider the group $\Gamma=(C_2 * C_2 * C_2) \times C_3$, that is, the direct product of $C_3$ and a free product of three copies of $C_2$. Let $a$, $b$ and $c$ be the generators of the copies of $C_2$ and let $r$ be the generator of $C_3$. Then $\Gamma$ can be represented as $\langle a,b,c,r \mid a^2=b^2=c^2=r^3=arar^{-1} = brbr^{-1} = crcr^{-1} = 1_{\Gamma} \rangle$. Figure~\ref{fig:caleygraph} shows the Cayley graph $G$ of $\Gamma$ with respect to the symmetric generating set $S=\{a,b,c,r,r^{-1}\}$.

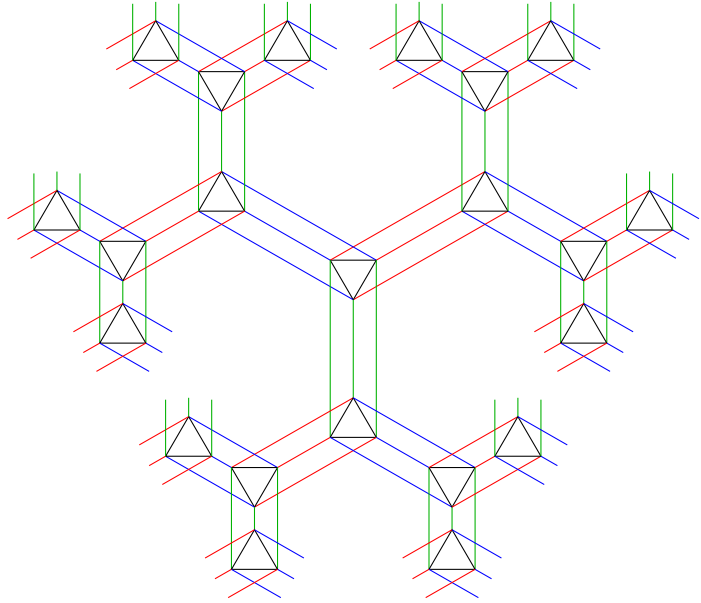
\begin{figure}
	\centering
	\begin{tikzpicture}
	\begin{scope}[scale=0.5]
	\def\dista{4}
	\def\distb{3}
	\def\distc{2}
	\def\diamr{0.7}
	\def\diama{0.7}
	\def\diamb{0.7}
	\def\diamc{0.7}
	\foreach \i in {0,1,2} {
		\draw (120*\i+30:\diamr)--(120*\i+150:\diamr);
		\coordinate (A\i) at (120*\i+30:\dista);
		\draw[gen\i] (A\i)+(120*\i+210:\diama) --(120*\i+30:\diamr);
		\draw[gen\i] (A\i)+(120*\i+90:\diama) --(120*\i+150:\diamr);
		\draw[gen\i] (A\i)+(120*\i+330:\diama) --(120*\i+270:\diamr);
		\draw (A\i)+(120*\i+90:\diama) -- +(120*\i+330:\diama);
		\foreach \j in {0,1} {
		    \draw (A\i)+(120*\i+120*\j+210:\diama) -- +(120*\i+120*\j+90:\diama);
		    \coordinate(B\i\j) at ($(A\i)+(120*\i+120*\j+330:\distb)$);
		    \pgfmathparse{int(mod(\i+\j+1,3))}
		    \draw[gen\pgfmathresult] (B\i\j)+(120*\i+120*\j+150:\diamb) -- ($(A\i)+(120*\i+120*\j+330:\diama)$);
		    \draw[gen\pgfmathresult] (B\i\j)+(120*\i+120*\j+30:\diamb) -- ($(A\i)+(120*\i+120*\j+90:\diama)$);
		    \draw[gen\pgfmathresult] (B\i\j)+(120*\i+120*\j+270:\diamb) -- ($(A\i)+(120*\i+120*\j+210:\diama)$);
		    \draw (B\i\j)+(120*\i+120*\j+30:\diamb) -- +(120*\i+120*\j+270:\diamb);
		    \foreach \k in {0,1} {
    		    \draw (B\i\j)+(120*\i+120*\j+120*\k+150:\diamb) -- +(120*\i+120*\j+120*\k+30:\diamb);
    		    \coordinate(C\i\j\k) at ($(B\i\j)+(120*\i+120*\j+120*\k+270:\distc)$);
    		    \pgfmathparse{int(mod(\i+\j+\k+2,3))}
    		    \draw[gen\pgfmathresult] (C\i\j\k)+(120*\i+120*\j+120*\k+210:\diamc) -- ($(B\i\j)+(120*\i+120*\j+120*\k+150:\diamb)$);
    		    \draw[gen\pgfmathresult] (C\i\j\k)+(120*\i+120*\j+120*\k+330:\diamc) -- ($(B\i\j)+(120*\i+120*\j+120*\k+30:\diamb)$);
    		    \draw[gen\pgfmathresult] (C\i\j\k)+(120*\i+120*\j+120*\k+90:\diamc) -- ($(B\i\j)+(120*\i+120*\j+120*\k+270:\diamb)$);
    		    \draw (C\i\j\k)+(120*\i+120*\j+120*\k+210:\diamc) -- +(120*\i+120*\j+120*\k+330:\diamc) -- +(120*\i+120*\j+120*\k+90:\diamc) -- cycle;
    		    \foreach \l in {0,1} {
    		        \pgfmathparse{int(mod(\pgfmathresult-\l+2,3))}
        		    \draw[gen\pgfmathresult] (C\i\j\k)++(120*\i+120*\j+120*\k-120*\l+330:\diamc) -- +(120*\i+120*\j+120*\k-120*\l+330:0.5);
        		    \draw[gen\pgfmathresult] (C\i\j\k)++(120*\i+120*\j+120*\k-120*\l+210:\diamc) -- +(120*\i+120*\j+120*\k-120*\l+330:1.5);
        		    \draw[gen\pgfmathresult] (C\i\j\k)++(120*\i+120*\j+120*\k-120*\l+90:\diamc) -- +(120*\i+120*\j+120*\k-120*\l+330:1.5);
        		}
    		}
		}
	}	
	\end{scope}
  	\end{tikzpicture}
  	\caption{Cayley graph of the group $(C_2*C_2*C_2) \times C_3$. Different edge colours correspond to different generators.}
  	\label{fig:caleygraph}
\end{figure}

\subsection{Formal languages}\label{sec:languages}

For an \emph{alphabet} (finite set of letters) $\Si$ we denote by 
\[
\Si^*=\{w=a_1 a_2 \dots a_n \mid n \geq 0, a_i \in \Si\}
\]
the set of all words over $\Si$. The number $n$ of letters in a word $w$ is called the \emph{length} of $w$ and denoted by $\abs{w}$; we write $\epsilon$ for the unique word of length 0. A \emph{language} over $\Si$ is a subset of $\Si^*$.

Any set of walks on an edge-labelled graph $G$ defines a language in the following way. Extend the label function $\ell$ to walks $p=(v_0,e_1,\dots, e_n, v_n)$ by setting
\[
\ell(p)=\ell(e_1) \ell(e_2) \dots \ell(e_n) \in \Si^*.
\]
Then for any set of walks $\Pcal$ on $G$, the associated language is 
\[
L(\Pcal)=\{\ell(p) : p \in \Pcal).
\]
For a given vertex $o$ of $G$ we denote by $\Pcal_{\SAW,o}$ the set of SAWs of length at least 1 on $G$ starting at $o$ and by $L_{\SAW,o}=L(\Pcal_{\SAW,o})$ the associated \emph{language of self-avoiding walks}. Note that if the edge-labelling is deterministic, then $\ell$ is a bijection between $\Pcal_{\SAW,o}$ and $L_{\SAW,o}$.

A \emph{context-free grammar} is a quadruple $\Gra=(\Non,\Si,\Pro,S)$, where $\Non$ is a finite set of \emph{non-terminals} with $\Non \cap \Si = \emptyset$, $S \in \Non$ is the \emph{start symbol} and $\Pro \subseteq \Non \times (\Non \cup \Si)^*$ is a finite set of \emph{production rules}. We write $A \vdash \alpha$ for $(A,\alpha) \in \Pro$. 

For two strings $\alpha,\beta \in (\Non \cup \Si)^*$, we say $\beta$ is obtained from $\alpha$ in a \emph{single step of leftmost derivation}, and write $\alpha \Rightarrow \beta$, if $\alpha=\alpha_1 A \alpha_2$ and $\beta= \alpha_1 \gamma \alpha_2$ for some $\alpha_1 \in \Si^*$, $\alpha_2 \in (\Non \cup \Si)^*$ and additionally $A \vdash \gamma \in \Pro$. Thus, $\beta$ is a result of using the rule $A \vdash \gamma$ to replace the leftmost non-terminal in $\alpha$. A \emph{leftmost derivation} of $\beta$ from $\alpha$ is a sequence \[(\alpha=\alpha_0,\alpha_1,\dots,\alpha_k=\beta)\] 
such that $\alpha_{i-1} \Rightarrow \alpha_i$ for every $i$. We say that $\beta$ is derived from $\alpha$ and write $\alpha \xRightarrow{*} \beta$. Each non-terminal $A \in \Non$ generates a language $L_A=\{\alpha \mid A \xRightarrow{*} \alpha\}$ and the \emph{language generated by} the grammar $\Gra$ is $L(\Gra)=L_S$. The grammar is called unambiguous if for every $\alpha \in L(\Gra)$ there is a unique leftmost derivation generating $\alpha$.

For a given context-free grammar $\Gra$ a \emph{derivation tree} is an ordered tree $D$ together with a labelling $\lambda:V(D) \rightarrow \Non \cup \Si^*$ such that
\begin{enumerate}[label=(\roman*)]
    \item internal vertices have labels in $\Non$,
    \item leaves have labels in $\Si^*$ and
    \item whenever $v_1, \dots, v_k$ are the ordered children of $u$ in $D$, 
    \[\lambda(u) \vdash \lambda(v_1) \dots \lambda(v_k) \in \Pro.\]
\end{enumerate}

Any ordered tree induces a total order $u_1, \dots, u_k$ on its leaves and we call $D$ a derivation tree of $w \in \Si^*$ if $w=\lambda(u_1) \dots \lambda(u_k)$. It is a standard result in formal language theory that there is a bijection between leftmost derivations of $w \in \Si^*$ from $A \in \Non$ and derivation trees of $w$ whose roots are labelled $A$.

The \emph{(commutative) language generating function} of a given language $L$ over the alphabet $\Si=\{a_1, \dots, a_{m}\}$ is a formal power series in the commuting variables $a_1, \dots, a_{m}$ over $\mathbb{Z}$ 
\[
F_L(a_1, \dots, a_{m})= \sum_{w \in L} c(w), 
\]
where $c(w)=\prod_{i=1}^{m} a_i^{l_i}$ for any word $w$ containing exactly $l_i$ copies of the letter $a_i$.

A famous result of Chomsky and Sch\"utzenberger~\cite{MR0152391} states that the commutative generating function of the language $L(\Gra)$ generated by an unambiguous context-free grammar $\Gra$ is algebraic over $\mathbb{Q}$, meaning that there is an irreducible polynomial $P$ in $m+1$ variables with coefficients in $\mathbb{Q}$ such that
\[
P\left(F_L(a_1, \dots, a_m),a_1, \dots, a_{m}\right)=0.
\]
A proof of this statement can for example be found in the book~\cite{MR817983} of Kuich and Salomaa.

\medskip

To motivate the upcoming introduction of multiple context-free grammars as a generalisation of context-free grammars, let us briefly discuss a different notation for context-free grammars $\Gra=(\Non,\Si,\Pro,S)$. When producing words, one usually starts with the start symbol $S$ and iteratively replaces non-terminals according to the rules given by $\Pro$. In terms of derivation trees, we build the trees starting from the top (root). 

There is also an alternative way to build words using production rules. A production rule $A \vdash x_0 A_1 x_1 \dots A_n x_n$ tells us that we can obtain an element of the language $L_A$ generated by $A$ by sticking together the strings $x_0, \dots, x_n \in \Si^*$ with strings $y_i \in L_{A_i}$, $i \in [n]$, according to the rule. This way of thinking is closely related to predicate logic. We might say that a word $w \in \Si^*$ has property $A \in \Non$ and write $\vdash_{\Gra} A(w)$ if $w \in L_A$. Then the rule $A \vdash x_0 A_1 x_1 \dots A_n x_n$ is equivalent to the statement 
\begin{center}
    ``If $\vdash_{\Gra} A_1(z_1), \dots, \vdash_{\Gra} A_n(z_n)$, then also $\vdash_{\Gra}A(x_0 z_1 x_1 \dots z_n x_n)$''.
\end{center} 
In this statement, the $z_i$ play the role of variables and it is natural to write the production as
\[
    A(x_0 z_1 x_1 \dots z_n x_n) \leftarrow A_1(z_1), \dots, A_n(z_n).
\]
With this in mind, the natural way to generate words is by starting with terminal rules and constructing the derivation from bottom to top, starting at its leaves.

Keeping this in mind, we introduce multiple context-free languages. Just like in the definition of context-free languages, the first two ingredients are an alphabet $\Si$ and a set of non-terminals $\Non$. As noted above we can intuitively view every non-terminal in a context-free language as a property applying to all strings it generates. Similarly, a non-terminal of a multiple context-free language should be viewed as a property applying to tuples of strings. To this end, every non-terminal is assigned an integer $r \geq 1$ counting the size of the tuples, called \emph{rank}. In other words, $\Non$ is a finite disjoint union $\Non=\bigcup_{r \in \mathbb{N}} \Non^{(r)}$ of finite sets $\Non^{(r)}$, whose elements are called \emph{non-terminals of rank $r$}. \emph{Production rules} $\rho$ of a multiple context-free grammar with non-terminals $\Non$ and alphabet $\Si$ are expressions of the form 
\[
\rho = A(\alpha_1, \dots, \alpha_r) \leftarrow A_1(z_{1,1}, \dots, z_{1,r_1}), \dots, A_n(z_{n,1}, \dots, z_{n,r_n}),
\]
where
\begin{enumerate}[label=(\roman*)]
    \item $n \geq 0$,
    \item $A \in \Non^{(r)}$ and $A_i \in N^{(r_i)}$ for all $i \in [n]$,
    \item $z_{i,j}$ are variables,
    \item $\alpha_1, \dots, \alpha_r$ are strings over $\Si \cup \{z_{i,j} \mid i \in [n], j \in [r_i]\}$, such that each $z_{i,j}$ occurs at most once in $\alpha_1 \dots \alpha_r$.
\end{enumerate}
Production rules with $n=0$ are called \emph{terminating rules}.
For convenience we sometimes use the shortened notation
\[
A(\alpha_1, \dots, \alpha_r) \leftarrow \left( A_i(z_{i,1}, \dots, z_{i,r_i}) \right)_{i \in [n]}
\]
for the production rule $\rho$.
For a non-terminal $A \in \Non^{(r)}$ and words $w_1, \dots, w_r \in \Si^*$ an expression of the form $A(w_1, \dots, w_r)$ is called a \emph{term}. The \emph{application} of the production rule $\rho$ to a sequence $(A_i(w_{i,1}, \dots, w_{i,r_i}))_{i \in [n]}$ of $n$ terms yields the term $A(w_1, \dots, w_r)$, where $w_l$ is obtained from $\alpha_l$ by replacing every variable $z_{i,j}$ by the word $w_{i,j}$.
The non-terminal $A$ is called the \emph{head} of the production and the sequence of non-terminals $A_1, \dots, A_n$ are its \emph{tail}. 

A \emph{multiple context-free grammar} is a quadruple $\Gra=(\Non,\Si,\Pro,S)$, where $\Non$ is a finite ranked set of non-terminals, $\Si$ is a finite alphabet, $\Pro$ is a finite set of production rules over $(\Non,\Si)$ and $S \in \Non^{(1)}$ is the \emph{start symbol}. We call $\Gra$ \emph{$k$-multiple context-free} or a \emph{$k$-MCFG}, if the rank of all non-terminals is at most $k$.

A term $\tau$ is called \emph{derivable} in $\Gra$ and we write $\vdash_\Gra \tau$ if there is a sequence $\mathcal{A}$ of derivable terms such that the application of a rule $\rho \in \Pro$ to $\mathcal{A}$ yields $\tau$. It is implicit in this recursive definition that if $A(w_1, \dots, w_r) \leftarrow$ is a terminal rule, then the term $A(w_1, \dots, w_r)$ is derivable by letting $\mathcal{A}$ be the empty sequence.
The language \emph{generated} by $\Gra$ is the set $L(\Gra)=\{w \in \Si^* \mid \; \vdash S(w)\}$. We call a language \emph{$k$-multiple context-free} or an $k$-MCFL if it is generated by an $k$-MCFG.

The following simple example of a 2-MCFG should be beneficial for a better understanding of the concepts above.

\begin{exa}
    Consider the MCFG $\Gra=(\Non,\Si,\Pro,S)$, where $\Non=\{S, A\}$, $\Si=\{a,b,c\}$ and the set $\Pro$ consists of the rules $\rho_1, \dots, \rho_5$ given as follows:
    \begin{alignat*}{2}
        &\rho_1: \quad & S(z_1 z_2) &\leftarrow A(z_1,z_2), \\
        &\rho_2: \quad & A(az_1b,z_2c) &\leftarrow A(z_1,z_2),\\
        &\rho_3: \quad & A(az_1b,z_2) &\leftarrow A(z_1,z_2),\\
        &\rho_4: \quad & A(az_1,z_2) &\leftarrow A(z_1,z_2),\\
        &\rho_5: \quad & A(\epsilon,\epsilon) &\leftarrow.
    \end{alignat*}

    From the production rules it is immediately clear that the rank of $A$ is 2 as $A$ works with pairs of strings. The rank of the start symbol $S$ is 1 by definition, so that $\Gra$ is 2-multiple context-free. We use the recursive definition above to find all terms derivable in $\Gra$.

    By the terminal rule $\rho_5$, the term $A(\epsilon,\epsilon)$ is derivable, we write $\vdash_G A(\epsilon,\epsilon)$. Applying $A(az_1,z_2) \leftarrow A(z_1,z_2)$ to the term $A(\epsilon,\epsilon)$, we replace $z_1$ and $z_2$ on the left side of the rule by the empty word $\epsilon$ and obtain the term $A(a,\epsilon)$. In a similar way, iterative application of $\rho_4$ yields that all terms of the form $A(a^k,\epsilon), k \geq 0$ are derivable. Making use of rule $\rho_3$, we obtain $\vdash_\Gra A(a^{k+l}b^l,\epsilon)$ for every $k,l \geq 0$. Analogously, the rule $\rho_2$ provides $\vdash_\Gra A(a^{k+l+m}b^{l+m},c^m)$ for $k,l,m \geq 0$. In a final step, we use the rule $S(z_1 z_2) \leftarrow A(z_1,z_2)$ containing the start symbol $S$ to create words of the language $L(\Gra)$: this rule $\rho_1$ is used to concatenate the pairs $(w_1,w_2)$ of strings appearing in derivable terms $A(w_1,w_2)$ and yields $\vdash_{\Gra} S(w_1 w_2)$. As a conclusion, all terms of the form $S(a^{k+l+m}b^{l+m}c^m)$ are derivable. 
    
    For the converse direction, note that any derivable term $S(w)$ must arise from an application of $\rho_1$, so that $w=w_1 w_2$ for some derivable term $A(w_1,w_2)$. It is not hard to see that in any such term, $w_1=a^k b^l$ and $w_2=c^m$ holds for some $k \geq l \geq m \geq 0$: as the only term arising from a terminal rule, $A(\epsilon,\epsilon)$ satisfies this condition and the rules $\rho_2$, $\rho_3$ and $\rho_4$ preserve it.
    We conclude that the language generated by $\Gra$ is
    \[
        L(\Gra)=\{a^{k+l+m}b^{l+m}c^m \mid k,l,m \geq 0\} = \{a^k b^l c^m \mid k \geq l \geq m \geq 0\}. 
    \]
\end{exa}

\begin{rmk} \label{rmk:cflanguages}
As one might already guess from the discussion right before the introduction of MCFGs, the class of context-free languages coincides with the class of 1-multiple context-free languages. A given context-free grammar $\Gra=(\Non,\Si,\Pro,S)$ can be easily translated into a 1-MCFG by replacing every production rule 
\[
A \vdash w_0 A_1 w_1 A_2 \dots A_n w_n \in \Pro,\] 
where $A, A_1, \dots, A_n \in \Non$ and $w_1, \dots, w_n \in \Si^*$ with the multiple-context-free production rule 
\[A(w_0 z_1 w_1 z_2 \dots z_n w_n) \leftarrow A_1(z_1), \dots, A_n(z_n)\] 
over $(\Non,\Si)$. The resulting 1-MCFG $\Gra'=(\Non,\Si,\Pro',S)$ then generates the same language $L(\Gra)$.
\end{rmk}

\begin{rmk}
Sometimes it will be convenient to work with a slightly different definition of multiple context-free grammars allowing non-terminals $A$ to have rank $r=0$. For such a non-terminal $A$ of rank $0$, the only valid term is $A(\emptyset)$, where $\emptyset$ denotes the empty tuple. We point out that $\emptyset$ is different from the 1-tuple $(\epsilon)$ containing the empty string. Note that the generative ability of $k$-multiple context-free languages does not change and that all properties discussed here remain valid under this variation. 
\end{rmk}

In the Chomsky-hierarchy of formal languages, multiple context-free languages lie strictly between context-free languages and the bigger class of context-sensitive languages. MCFLs share some important properties with context-free languages. They are closed under under homomorphisms, inverse homomorphisms, union, intersection with regular languages and Kleene closure. Furthermore they are parsable in polynomial time and semilinear.

Derivation trees for multiple context-free languages were first defined by Seki et al~\cite{MR1131066}; we use a slight variation of their definition. Let $\Gra=(\Non,\Si,\Pro,S)$ be an MCFG. A \emph{derivation tree} of a term $\tau$ with respect to the grammar $\Gra$ is an ordered tree $D$ whose vertices are labelled with elements of $\Pro$ satisfying the following conditions:
\begin{enumerate}[label=(\roman*)]
    \item The root of $D$ has $n \geq 0$ children and is labelled with a rule $\rho \in \Pro$. 
    \item For $i \in [n]$ the subtree $D_i$ rooted at the $i$-th child of the root of $D$ is a derivation tree of a term $\tau_i$.
    \item The rule $\rho$ applied to the sequence $(\tau_i)_{i \in [n]}$ yields $\tau$. 
\end{enumerate}
It is not hard to see that $\vdash A(w_1, \dots, w_r)$ if and only if there is a derivation tree $D$ of $A(w_1, \dots, w_r)$. However, in general such a derivation tree need not be unique. 
An MCFG $\Gra$ is called \emph{unambiguous}, if for every term $S(w)$ there is at most one derivation tree of $S(w)$ with respect to $\Gra$. An MCFL is called \emph{unambiguous} if it is generated by an unambiguous MCFG.
We denote by $w(D)$ the tuple of strings $w_1, \dots, w_r$ generated by $D$.

In some sense derivations trees of MCFGs are more natural than derivation trees of CFGs. The tree corresponding to the derivation process of a term $\tau$ in an MCFG consists of a single vertex labelled $\rho$ for every rule $\rho$ occurring in the process.

The pumping lemma for $k$-MCFLs, similarly to the well known pumping lemma for CFLs, provides a convenient way to show that certain languages are not $k$-multiple context free.
\begin{lem}[{\cite[Lemma 3.2]{MR1131066}}] \label{lem:pumplem}
For every infinite $k$-MCFL $L$ there is some $w \in L$, which can be written in the form $w=x_1 y_1 x_2 y_2 \dots x_{2k} y_{2k} x_{2k+1}$ for some $x_i, y_i, \in \Si^*$ such that 
\begin{itemize}
    \item $y_1 y_2 \dots y_{2k} \neq \epsilon$ and
    \item $x_1 y_1^n x_2 y_2^n \dots x_{2k} y_{2k}^n x_{2k+1} \in L$ for every $n \in \N_0$.
\end{itemize}
\end{lem}
Note that this lemma is weaker than the pumping lemma for CFLs: it only provides the existence of ``pumpable'' strings whereas the pumping lemma for CFLs states that all words exceeding a certain length are pumpable. In particular it is not strong enough to imply the second part of the following result from \cite{lehner2020comparing}, which is a main pillar of Theorem~\ref{thm:main-2}.
\begin{thm}[{\cite[Theorem 1.1]{lehner2020comparing}}] \label{thm:mcfgpump}
The language $\{a_1^{n_1} a_2^{n_2} \dots a_{k}^{n_k} \mid n_1 \geq n_2 \geq \dots \geq n_k \geq 0\}$ is a $\lceil k/2\rceil$-MCFL, but not a $(\lceil k/2\rceil-1)$-MCFL.
\end{thm}

\section{Tree decompositions}\label{sec:treedecomp}

A \emph{tree decomposition} of a graph $G$ is a pair $\Tcal=(T,\Vcal)$, consisting of a tree $T$ and a function $\Vcal: V(T) \rightarrow 2^{V(G)}$ assigning a subset of $V(G)$ to every vertex of $T$, such that the following three conditions are satisfied:
\begin{enumerate}[label=(T\arabic*)]
\item \label{itm:td-coververtices}
$V(G)= \bigcup_{t \in V(T)} \Vcal(t)$.
\item \label{itm:td-coveredges}
For every $e \in E(G)$ there is a $t \in V(T)$ such that $\Vcal(t)$ contains both vertices that are incident with $e$.
\item \label{itm:td-nocrosssedge}
$\Vcal(s) \cap \Vcal(t) \subseteq \Vcal(r)$ for every vertex $r$ on the unique $s$--$t$-path in $T$.
\end{enumerate}
The set $\Vcal(t)$ is called the \emph{part} of $t$. For an edge $e=st$ of $T$, the intersection $\Vcal(e)=\Vcal(s,t)= \Vcal(s) \cap \Vcal(t)$ $(=\Vcal(t,s))$ is called the \emph{adhesion set} of $e$. 
A tree decomposition $(T,\Vcal)$ of $G$ is called \emph{$\Gamma$-invariant} for a group $\Gamma \leq \AUT(G)$, if every $\gamma \in \Gamma$ maps parts onto parts and thereby induces an automorphism of $T$. More precisely there is an action of $\Gamma$ on $T$ by automorphisms such that for every $\gamma \in \Gamma$ and $t \in V(T)$, it holds that $\gamma(\Vcal(t))=\Vcal(\gamma t)$.

The tree decomposition $\Tcal$ is said to \emph{distinguish} two given ends $\omega_1$ and $\omega_2$ of $G$ if there is some edge $e$ of $T$ such that the adhesion set $\Vcal(e)$ separates $\omega_1$ and $\omega_2$. Moreover $\Tcal$ distinguishes the two ends \emph{efficiently}, if one of its adhesion sets has the smallest size of all subsets of vertices of $G$ separating $\omega_1$ and $\omega_2$. We call $\Tcal$ \emph{reduced} if every adhesion set efficiently distinguishes some pair of ends of $G$ and no two parts corresponding to adjacent vertices of $T$ coincide.

A graph $G$ is called \emph{accessible} if there is a natural number $k$ such that any two ends can be separated by a set of vertices of size at most $k$. Originally the notion of accessibility comes from group theory. Stalling's theorem about ends of groups states that some (every) Cayley graph a finitely generated group $\Gamma$ has more than one end if and only if $\Gamma$ admits a nontrivial decomposition as an amalgamated free product or an HNN-extension over a finite subgroup. $\Gamma$ is called \emph{accessible} if the process of iterated nontrivial splitting of $\Gamma$ always terminates in a finite number of steps. Thomassen and Woess~\cite{MR1223698} showed that accessibility of a group is equivalent to accessibility of some (and thus all) of its Cayley graphs. The following theorem is closely related to the work of Dunwoody~\cite{MR671142}  and is a direct consequence of \cite[Theorem 6.4]{hamann2018stallings}.

\begin{thm}\label{thm:graphdecomp}
Let $G$ be a simple, locally finite, connected, accessible graph and let $\Gamma$ be a group acting quasi-transitively on $G$. Then there is a $\Gamma$-invariant tree decomposition $(T,\Vcal)$ of $G$ efficiently distinguishing all ends of $G$ and an action of $\Gamma$ on $T$ witnessing the $\Gamma$-invariance of $(T, \Vcal)$ with only finitely many $\Gamma$-orbits on $E(T)$. 
\end{thm}

For our purpose, it is necessary for tree decompositions to additionally be reduced. However, this is not really a restriction, as the following construction shows.

Let $\Tcal=(T,\Vcal)$ be a tree decomposition of $G$ and $F$ be a subset of edges of $T$. The contraction  of $F$ in $\Tcal$ is the pair $\Tcal/F=(T/F,\Vcal/F)$ defined in the following way. The tree $T/F$ consists of a single vertex for every component of the graph $(V(T),F)$ obtained from $T$ by only keeping the edges in $F$; for a vertex $t$ of $T/F$ let $[t]_F$ denote the vertex set of the corresponding component. Two different vertices $[s]_F$ and $[t]_F$ of $T/F$ are connected by an edge if and only if there are $s' \in [s]_F$ and $t' \in [t]_F$, which are adjacent in $T$. The part corresponding to $[t]_F \in V(T/F)$ is $(\Vcal/F)([t]_F) = \bigcup_{s \in [t]_F} \Vcal(s)$. It is not hard to see that $\Tcal/F$ is a tree decomposition of $G$.

Starting from a tree decomposition $\Tcal=(T,\Vcal)$ provided by the previous theorem, we can construct a reduced tree decomposition as follows. Let the set $F$ consist of all edges $e$ of $T$ such that the adhesion set $\Vcal(e)$ does not minimally separate any pair of ends of $G$.
It is easy to check that the contraction $\Tcal/F$ is a tree decomposition retaining all properties mentioned in Theorem~\ref{thm:graphdecomp} and additionally every adhesion set minimally separates two ends of $G$. In a second step we contract all edges of $T/F$ connecting two vertices whose parts coincide to obtain a reduced tree decomposition as in the following corollary.

\begin{cor}\label{cor:graphdecomp}
Let $G$ be a simple, locally finite, connected, accessible graph and let $\Gamma$ be a group acting quasi-transitively on $G$. Then there is a reduced $\Gamma$-invariant tree decomposition $(T,\Vcal)$ of $G$ efficiently distinguishing all ends of $G$ such that there are only finitely many $\Gamma$-orbits on $E(T)$.
\end{cor}

Let us again look at the Cayley graph $G$ of the group $\Gamma=(C_2*C_2*C_2)\times C_3$ given in Figure~\ref{fig:caleygraph}. We already mentioned that $\Gamma$ acts freely on $G$ by left multiplication. A reduced $\Gamma$-invariant tree decomposition $(T,\Vcal)$ of $G$ as provided by the previous corollary is shown in Figure~\ref{fig:treedecomp}. 

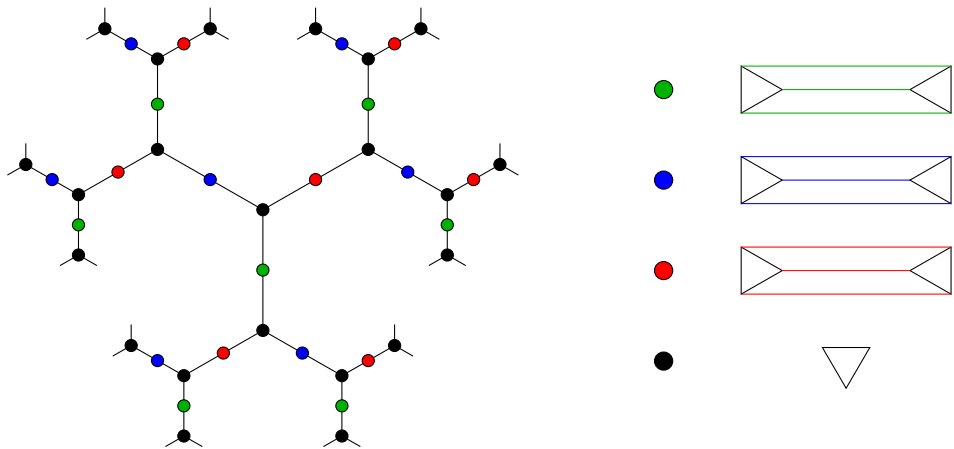
\begin{figure}
	\centering
	\begin{subfigure}{0.6\textwidth}
	\centering
	\begin{tikzpicture}
	\begin{scope}[scale=0.4]
	\def\dista{4}
	\def\distb{3}
	\def\distc{2}
	\def\diamr{0.2}
	\def\diama{0.2}
	\def\diamb{0.2}
	\def\diamc{0.2}
    \draw[fill=black] (0,0) circle (\diamr);
	\foreach \i in {0,1,2} {
		\coordinate (A\i) at (120*\i+30:\dista);
		\draw[fill=black] (A\i) circle (\diama);
		\draw (A\i)+(120*\i+210:\diama) --coordinate[midway](m) (120*\i+30:\diamr);
	    \draw[fill=gen\i] (m) circle (\diama);
		\foreach \j in {0,1} {
		    \coordinate(B\i\j) at ($(A\i)+(120*\i+120*\j+330:\distb)$);
		    \draw[fill=black] (B\i\j) circle (\diamb);
		    \pgfmathparse{int(mod(\i+\j+1,3))}
		    \draw (B\i\j)+(120*\i+120*\j+150:\diamb) -- coordinate[midway](m) ($(A\i)+(120*\i+120*\j+330:\diama)$);
		    \draw[fill=gen\pgfmathresult] (m) circle (\diamb);
		    \foreach \k in {0,1} {
    		    \coordinate(C\i\j\k) at ($(B\i\j)+(120*\i+120*\j+120*\k+270:\distc)$);
                \draw[fill=black] (C\i\j\k) circle (\diamc);
    		    \pgfmathparse{int(mod(\i+\j+\k+2,3))}
    		    \draw (C\i\j\k)+(120*\i+120*\j+120*\k+90:\diamc) -- coordinate[midway](m) ($(B\i\j)+(120*\i+120*\j+120*\k+270:\diamb)$);
    		    \draw[fill=gen\pgfmathresult] (m) circle (\diamc);
    		    \foreach \l in {0,1} {
        		    \draw (C\i\j\k)++(120*\i+120*\j+120*\k-120*\l+330:\diamc) -- +(120*\i+120*\j+120*\k-120*\l+330:0.5);
        		}
    		}
		}
	}	
	\end{scope}
  	\end{tikzpicture}
  	\end{subfigure}
  	\hfill
  	\begin{subfigure}{0.35\textwidth}
  	\begin{tikzpicture}
  	\def\diamr{0.2}
  	\def\diamt{0.6}
  	\begin{scope}[scale=0.6]
      	\draw[fill=black] (0,0) circle (\diamr);
      	\draw[fill=gen0] (0,2) circle (\diamr);
      	\draw[fill=gen1] (0,4) circle (\diamr);
      	\draw[fill=gen2] (0,6) circle (\diamr);
      	\foreach \i in {0,1,2} {
      	    \draw (2,2*\i+2)+(0:\diamt)--+(120:\diamt)--+(240:\diamt)--cycle;
      	    \draw (6,2*\i+2)+(180:\diamt)--+(300:\diamt)--+(60:\diamt)--cycle;
      	    \draw[color=gen\i] (2,2*\i+2)+(0:\diamt) -- ($(6,2*\i+2)+(180:\diamt)$);
      	    \draw[color=gen\i] (2,2*\i+2)+(120:\diamt) -- ($(6,2*\i+2)+(60:\diamt)$);
      	    \draw[color=gen\i] (2,2*\i+2)+(240:\diamt) -- ($(6,2*\i+2)+(300:\diamt)$);
      	}
      	\draw (4,0)+(-90:\diamt)--+(-210:\diamt)--+(-330:\diamt)--cycle;
  	\end{scope}
  	\end{tikzpicture}
  	\end{subfigure}
  	\caption{Decomposition tree $T$ of the Cayley graph $G$. The vertex-colouring indicates the four orbits of the group action on the vertices of $T$. The subgraphs of $G$ induced by the different parts are shown at the right side.}
  	\label{fig:treedecomp}
\end{figure}

We require the following four important properties of tree decompositions $\Tcal=(T,\Vcal)$ obtained from this corollary.
\begin{enumerate}[label=(P\arabic*)]
    \item \label{itm:boundedadhesion}
    The size of all adhesion sets is bounded from above by some $k \in \N$.
    \item \label{itm:separationfinitelyused}
    For every $K\subseteq V(G)$ there are only finitely many edges $e$ of $T$ with $K=\Vcal(e)$. 
    \item  \label{itm:finiteparts-thinends}
    All parts are finite if and only if all ends of $G$ have finite size.
    \item \label{itm:thinends-locfintree}
    If all ends of $G$ have finite size, then $T$ is locally finite.
\end{enumerate}

Firstly, $\Tcal$ is reduced, so every adhesion set minimally separates some pair of ends of $G$. The graph $G$ is accessible, so \ref{itm:boundedadhesion} holds.

For the proof of the other properties, we need some notation. A \emph{separation} of $G$ is a pair $(A,B)$ of vertex sets such that $G[A] \cup G[B]=G$, which means that there are no edges between $A \setminus B$ and $B \setminus A$. The intersection $A \cap B$ is called the \emph{separator} of $(A,B)$. Note that every edge $e$ of the tree decomposition $\Tcal$ corresponds to a separation of $G$ with separator $\Vcal(e)$. Removal of $e$ splits $T$ into two components $T_1$ and $T_2$ and $\left(\bigcup_{t \in V(T_1)} \Vcal(t),\bigcup_{t \in V(T_2)} \Vcal(t)\right)$ is the separation of $G$ induced by $e$.

Clearly any given finite set $K$ of vertices of a locally finite graph $G$ can occur only finitely many times as the separator of a separation $(A,B)$ of $G$. Indeed, $G-K$ has only finitely many components and each of these components has to be fully contained in either $A$ or $B$. This observation together with the following lemma yields property~\ref{itm:separationfinitelyused}.

\begin{lem}
Let $\Tcal=(T,\Vcal)$ be a tree decomposition as in Corollary~\ref{cor:graphdecomp}. Then every separation $(A,B)$ of $G$ corresponds to at most two edges of $T$.
\end{lem}
\begin{proof}
Let $e_1,e_2$ be two edges of $T$ inducing the separation $(A,B)$. We first show that $e_1$ and  $e_2$ share a vertex. For $i=1,2$ let $T_i^A$ and $T_i^B$ be the components of $T-e_i$ corresponding to $A$ and $B$, respectively. The separation $(A,B)$ is induced by an edge of $T$ and every adhesion set separates two ends of $G$, so $A$ and $B$ are both infinite sets and $\Vcal(e_1)=\Vcal(e_2)=A \cap B$ is finite. We may assume without loss of generality that $T_1^A \subseteq T_2^A$ and $T_2^B \subseteq T_1^B$. Let $s$ be a vertex on the unique shortest path $P$ connecting $e_1$ and $e_2$ in $T$. Property~\ref{itm:td-nocrosssedge} of tree decompositions yields
\[A\cap B=\Vcal(e_1) \cap \Vcal(e_2) \subseteq \Vcal(s).\]
Moreover $s$ is a vertex of $T_1^B$ and $T_2^A$, so in particular
\[\Vcal(s) \subseteq \bigcup_{t \in V(T_1^B)} \Vcal(t) \cap \bigcup_{t \in V(T_2^A)} \Vcal(t)=A \cap B,\]
implying that the part $\Vcal(s)$ is equal to $A \cap B$ for every vertex $s$ of $P$. Reducedness of $\Tcal$ implies that $P$ consists of a single vertex and thus both $e_1$ and $e_2$ contain $s$.\par
Let $e_3$ be any edge of $T$ inducing the separation $(A,B)$ and let $T_3^A$ and $T_3^B$ be the components of $T-e_3$ corresponding to $A$ and $B$, respectively. Then $e_3$ intersects $e_1$ and $e_2$ and thus contains their common vertex $r$. Finally, if $e_3$ is different from $e_1$ and $e_2$, then either $T_2^B \subseteq T_3^A$ or $T_1^A \subseteq T_3^B$, leading to a contradiction.
\end{proof}

The basis for the proof of \ref{itm:finiteparts-thinends} is the following lemma, which is closely related to \cite[Proposition 4.5]{hamann2018stallings}.

\begin{lem}\label{lem:partsqtrans}
Let $G$, $\Gamma$ and $(T,\Vcal)$ be as in Theorem~\ref{thm:graphdecomp}. Then for every vertex $t$ of $T$, the induced sub-graph $G[\Vcal(t)]$ is a quasi-transitive graph.
\end{lem}
\begin{proof}
We show that the set-wise stabiliser $\Gamma_{\Vcal(t)}$ of $\Vcal(t)$ in $\Gamma$ acts quasi-transitively on $G[\Vcal(t)]$.
If $u \in \Vcal(t)$ does not lie in any adhesion set, then neither does any image of $u$ under a graph automorphism. In particular, any $\gamma \in \Gamma$ mapping $u$ to some vertex $v \in \Vcal(t)$ fixes $\Vcal(t)$ and thus, under the action of the stabiliser of $\Vcal(t)$ there are only finitely many orbits of vertices in $\Vcal(t)$ not contained in any adhesion set. 

Let $m$ be the finite number of $\Gamma$-orbits on $E(T)$. Whenever $\gamma \in \Gamma$ fixes $t$ and maps a neighbour $s$ of $t$ onto some other neighbour $s'$, $\gamma$ lies in  $\Gamma_{\Vcal(t)}$ and maps the adhesion set $\Vcal(s,t)$ onto $\Vcal(s',t)$. Therefore the number of orbits of adhesion sets under the action of $\Gamma_{\Vcal(t)}$ is at most $2m$. As every adhesion set contains at most $k$ elements, $\Gamma_{\Vcal(t)}$ acts with at most $2mk$ orbits on vertices of $\Vcal(t)$, which lie in adhesion sets of the tree decomposition. We conclude that $\Gamma_{\Vcal(t)}$ acts with finitely many orbits on $\Vcal(t)$.
\end{proof}

Our goal is to apply the following proposition due to Thomassen~\cite{MR1194738} to a part $\Vcal(t)$, but in general $G[\Vcal(t)]$ need not be connected, so some additional work is necessary.

\begin{pro}[{\cite[Proposition 5.6]{MR1194738}}] \label{pro:qtransendthick}
If $G$ is a locally finite, connected, quasi-transitive graph with only one end, then this end is thick.
\end{pro}

In order to prove the first implication of \ref{itm:finiteparts-thinends}, assume that there is some vertex $t$ of $T$ such that the part $\Vcal(t)$ is infinite. Let $H$ be the sub-graph of $G$ obtained from the induced sub-graph $G[\Vcal(t)]$ in the following way. For every edge $e$ of $T$ incident to $t$ add all shortest paths between any pair of vertices in the adhesion set $\Vcal(e)$. Note that since the stabiliser of $\Vcal(t)$ acts quasi-transitively, the length of these paths is bounded by some constant $m \in \N$. 

Any walk on $G$ connecting two vertices of $\Vcal(t)$ consists of sub-walks on $G[\Vcal(t)]$ and detours leaving $\Vcal(t)$ via some adhesion set $\Vcal(e)$ and re-entering via the same set. These detours can be replaced by a shortest detour, which is by definition a walk on $H$, so $H$ is connected. Furthermore, $\Gamma_{\Vcal(t)}$ acts quasi-transitively on $H$ because it acts with finitely many orbits on the edges of $T$ and thus on the adhesion sets contained in $\Vcal(t)$.

Assume for a contradiction that $H$ has more than one end. Then there must be a separation $(A,B)$ of $H$ with finite separator $A \cap B$ such that both $A$ and $B$ are infinite. Let $K$ be the union of $A \cap B$ and all adhesion sets $\Vcal(s,t)$ containing both a vertex $a$ of $A \setminus B$ and a vertex $b$ of $B \setminus A$. Due to construction of $H$ it contains an $a$--$b$-path of length at most $m$ and any such path must intersect the separator $A \cap B$. Thus any vertex of $K$ lies at distance at most $m$ from $A \cap B$, implying that $K$ is finite. 

It is not hard to see that $G-K$ contains no $(A \setminus B)$--$(B \setminus A)$-path. Indeed, assume that $P$ is a $(A \setminus B)$--$(B \setminus A)$-path in $G$ which does not intersect $A \cap B$. Then $P$ is a detour leaving and re-entering $G[\Vcal(t)]$ via some adhesion set $\Vcal(s,t)$ intersecting $A \setminus B$ and $B \setminus A$ and thus contains at least one vertex in $K$.

Finally, let $R_A$ and $R_B$ be rays in $H[A]$ and $H[B]$, respectively. Then the ends $\omega_A$ and $\omega_B$ of $G$ containing $R_A$ and $R_B$ are different, as $R_A$ and $R_B$ are separated by $K$. On the other hand, each of $R_A$ and $R_B$ contains infinitely many vertices of $\Vcal(t)$, so they are not separated by any adhesion set of the tree decomposition $\Tcal$. This contradicts the fact that $\Tcal$ distinguishes all ends of $G$. We conclude that the infinite connected graph $H$ has precisely one end.
By Proposition~\ref{pro:qtransendthick} the end of the one-ended quasi-transitive graph $H$ is thick. The graph $G$ contains $H$ as a sub-graph and thus inherits the thick end of $H$.

On the other hand, it is not hard to see that all ends of $G$ are thin, if all parts of $\Tcal$ are finite. For any set of disjoint rays in the same end of $G$ there must be some adhesion set intersecting each of the rays. The size of adhesion sets is at most $k$, so every end of $G$ has size at most $k$.

Finally \ref{itm:thinends-locfintree} is a consequence of \ref{itm:separationfinitelyused} and \ref{itm:finiteparts-thinends}. Every edge $e$ incident to a vertex $s$ of $T$ corresponds to some adhesion set $\Vcal(e)$ which is a subset of the part $\Vcal(t)$. But the finite part $\Vcal(t)$ has only finitely many different subsets and each of them occurs only finitely often as an adhesion set in $(T,\Vcal)$.

\subsection{Rooted tree decompositions}

Let $(G,o)$ be a simple, locally finite, connected graph rooted at $o \in V(G)$. A \emph{rooted tree decomposition} $\Tcal=(T,\Vcal,r)$ of $(G,o)$ consists of a tree decomposition $(T, \Vcal)$ of $G$ and a fixed vertex $r$ of $T$ such that $o$ is contained in $\Vcal(r)$; note that there can be multiple valid choices for $r$ since $o$ can be contained in more than one part. We call $r$ the \emph{root} of $T$ and $\Vcal(r)$ the \emph{root part} of the decomposition. 

For every $t \in V(T)$ we introduce a graph $\Gcal(t)$ on the vertex set $\Vcal(t)$. Let us start by defining a map $\Ecal: V(T) \rightarrow 2^E(G)$ by $\Ecal(r)=E(G[\Vcal(r)])$ and
\[
\Ecal(t)= E(G[\Vcal(t)])\setminus E(G[\Vcal({t^\ua})]) \quad \text{for } t \neq r,
\]
where $t^\ua$ denotes the parent of $t$ in the rooted tree $T$.
Edges in $\Ecal(t)$ are called \emph{(non-virtual) $t$-edges}.
Property~\ref{itm:td-coveredges} of tree decompositions implies that for every edge $e$ of $G$ there is some $t \in V(T)$ such that $e \in \Ecal(t)$. Fix some edge $e$ of $G$ and let $S$ be the set of all vertices $s$ of $T$ such that $\Vcal(s)$ contains both endpoints of $e$. By property~\ref{itm:td-nocrosssedge} the induced sub-graph $T[S]$ is connected and thus the forefather $t$ of $S$ is contained in $S$. It is easy to see that $t$ is the unique vertex of $T$ with $e \in \Ecal(t)$, so the edge set of $G$ is the disjoint union $E(G)=\biguplus_{t \in V(T)} \Ecal(t)$. 

Additionally we introduce for every edge $e=st$ of $T$ a new set of \emph{virtual $e$-edges} $\Ecal(e)=\Ecal(st)$, such that every pair of vertices of $\Vcal(e)$ is connected by an edge in $\Ecal(e)$. In other words, the \emph{$e$-graph} $\Gcal(e)=(\Vcal(e),\Ecal(e))$ is a complete graph. In order to enhance readability, we usually write $\Ecal(s,t)$ instead of $\Ecal(st)$ and $\Gcal(s,t)$ instead of $\Gcal(st)$.

Finally, we assign to every vertex $t$ of $T$ the \emph{$t$-graph} 
\[
\Gcal(t)=\left(\Vcal(t)\;,\; \Ecal(t) \uplus \biguplus_{e \colon e^-=t} \Ecal(e)\right).
\]
Note that $\Gcal(t)$ generally is not a simple graph since $\Ecal(t)$ and the various sets $\Ecal(s,t)$ potentially contain edges with the same endpoints.

For convenience we extend the definition of $\Gcal$ to subsets $S$ of the vertex set $V(T)$ by taking the union of all graphs $\Gcal(t)$ for $t \in S$ and removing all virtual edges corresponding to edges of $T[S]$, so that only virtual edges corresponding to edges of $T$ with exactly one endpoint in $S$ remain. In terms of sets,
\[
\Gcal(S)=\left(\bigcup_{t \in S} \Vcal(t)\;, \; \bigcup_{t \in S} \Ecal(t) \uplus \biguplus_{e \colon e^- \in S,\, e^+ \notin S} \Ecal({e})\right).
\]

Again, we visualise these concepts using the Cayley graph $G$ from Figure~\ref{fig:caleygraph} and its tree decomposition $(T,\Vcal)$ shown in Figure~\ref{fig:treedecomp}. For a given root $o$ of $G$, denote by $r$ the unique vertex of $T$ such that the part $\Vcal(r)$ has cardinality 3 and contains $o$. Then $(T, \Vcal, r)$ is a rooted tree decomposition of the rooted graph $(G,o)$. Figure~\ref{fig:partgraph} shows a portion of the decomposition tree $T$ and the $t$-graphs for vertices $t$ contained in it. Compare this to Figure~\ref{fig:treedecomp} and note that the $t$-graphs on the parts are generally neither sub-graphs (due to virtual edges) nor super-graphs (due to some missing non-virtual edges) of the induced graphs on the parts.

\begin{figure}[tb]
	\centering
	\begin{tikzpicture}
	\begin{scope}[scale=0.45]
	\def\distt{2.9}
	\def\dista{4.2}
	\def\diamr{1.4}
	\def\diama{1.4}
	\foreach \i [evaluate=\i as \ieva using {int(mod(\i+1,3))},evaluate=\i as \ievb using {int(mod(\i+2,3))}] in {0,1,2} {
		\draw (120*\i+30:\diamr)--(120*\i+150:\diamr);
		\coordinate (A\i) at (120*\i+30:\distt);
		\coordinate (B\i) at (120*\i+30:{\distt+\dista});
		\coordinate (C\i) at (120*\i+30:{2*\distt+\dista});
		\draw[gen\i] (B\i)+(120*\i+210:\diama) -- ($(A\i)+(120*\i+30:\diamr)$);
		\draw[gen\i] (B\i)+(120*\i+90:\diama) -- ($(A\i)+(120*\i+150:\diamr)$);
		\draw[gen\i] (B\i)+(120*\i+330:\diama) -- ($(A\i)+(120*\i+270:\diamr)$);
		\draw[virtual=gen\i]  (120*\i+30:\diamr) to[in={120*\i+30}, out={120*\i+60},looseness=1.5] (120*\i+150:\diamr);
		\draw[virtual=gen\i]  (120*\i+270:\diamr) to[in={120*\i}, out={120*\i+30},looseness=1.5] (120*\i+30:\diamr);
		\draw[virtual=gen\i]  (120*\i+150:\diamr) to[in={120*\i+90}, out={120*\i-30}] (120*\i+270:\diamr);
		\draw[virtual=gen\i]  (A\i)+(120*\i+30:\diamr) to[in={120*\i+30}, out={120*\i+60},looseness=1.5] +(120*\i+150:\diamr);
		\draw[virtual=gen\i]  (A\i)+(120*\i+270:\diamr) to[in={120*\i}, out={120*\i+30},looseness=1.5] +(120*\i+30:\diamr);
		\draw[virtual=gen\i]  (A\i)+(120*\i+150:\diamr) to[in={120*\i+90}, out={120*\i-30}] +(120*\i+270:\diamr);
		\draw[virtual=gen\i]  (B\i)+(120*\i+210:\diamr) to[in={120*\i+210}, out={120*\i+240},looseness=1.5] +(120*\i+330:\diamr);
		\draw[virtual=gen\i]  (B\i)+(120*\i+90:\diamr) to[in={120*\i+180}, out={120*\i+210},looseness=1.5] +(120*\i+210:\diamr);
		\draw[virtual=gen\i]  (B\i)+(120*\i+330:\diamr) to[in={120*\i+270}, out={120*\i+150}] +(120*\i+90:\diamr);
	}	
	\foreach \j in {0,1,2} {
	    \foreach \i in {0,1,2}{
		    \draw (B\i)+(120*\i+120*\j+210:\diama) -- +(120*\i+120*\j+90:\diama);
	    	\draw[virtual=gen\i]  (C\j)+(120*\i+210:\diamr) to[in={120*\i+210}, out={120*\i+240},looseness=1.5] +(120*\i+330:\diamr);
		    \draw[virtual=gen\i]  (C\j)+(120*\i+90:\diamr) to[in={120*\i+180}, out={120*\i+210},looseness=1.5] +(120*\i+210:\diamr);
		    \draw[virtual=gen\i]  (C\j)+(120*\i+330:\diamr) to[in={120*\i+270}, out={120*\i+150}] +(120*\i+90:\diamr);
		}
	}
	\end{scope}
	\begin{scope}[scale=0.5, xshift=-8cm,yshift=-4cm]
	\def\dista{4}
	\def\diamr{0.2}
	\def\diama{0.2}
	\def\diamb{0.2}
	\def\diamc{0.2}
	\node at (0,0.6) {$r$};
    \draw[fill=black] (0,0) circle (\diamr);
	\foreach \i in {0,1,2} {
		\coordinate (A\i) at (120*\i+30:\dista);
		\draw[fill=black] (A\i) circle (\diama);
		\draw (A\i)+(120*\i+210:\diama) --coordinate[midway](m) (120*\i+30:\diamr);
	    \draw[fill=gen\i] (m) circle (\diama);
	    \draw (A\i)++(120*\i+90:\diama) -- +(120*\i+90:\dista/6);
	    \draw (A\i)++(120*\i+330:\diama) -- +(120*\i+330:\dista/6);
	}
	\end{scope}
  	\end{tikzpicture}
  	\caption{Decomposition tree $T$ of the Cayley graph $G$ and its corresponding $t$-graphs. Dashed edges are virtual $e$-edges; if they are shared by different $t$-graphs, they have the same shape in these graphs.}
  	\label{fig:partgraph}
\end{figure}
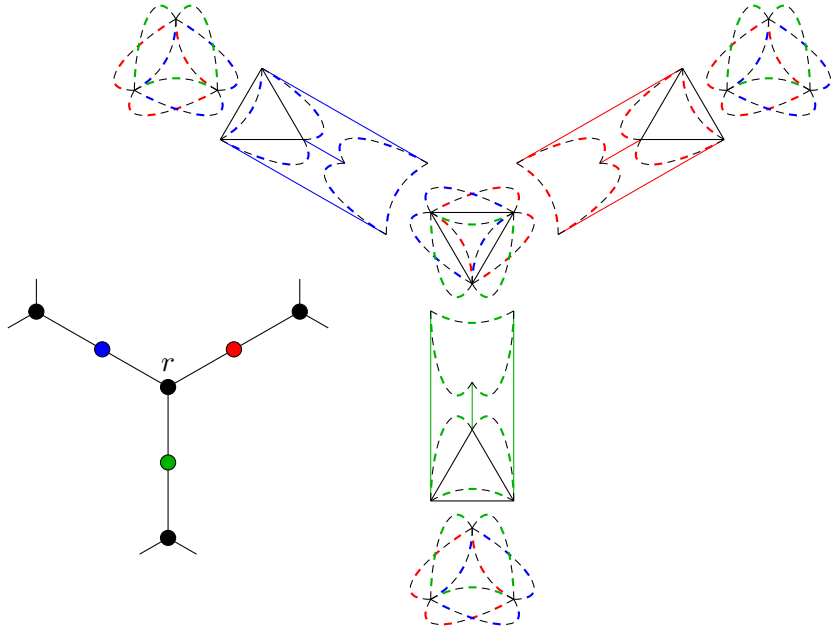

\subsection{Cones and cone types}

Recall that for a rooted tree decomposition $(T,\Vcal,r)$ the cone at a vertex $s \in V(T)$ of $T$ as the set $K_s$ containing all descendants of $s$, that is, all vertices $t$ such that $s$ lies on the $t$--$r$-path in $T$.

Let $\Gamma \subseteq \AUT(G)$ be a group acting on $G$. We say that two vertices $s$ and $t$ of $T$ different from $r$ are \emph{cone-equivalent} and write $s \sim_K t$, if there is a $\gamma \in \Gamma$ mapping $s$ to $t$ and the parent $s^\ua$ of $s$ to the parent $t^\ua$ of $t$. The root $r$ is only cone equivalent to itself. Clearly $\sim_K$ is an equivalence relation and we call the equivalence classes of vertices \emph{cone types} of the rooted tree decomposition $\Tcal$. 

Note that if $\gamma \in \Gamma$ witnesses the cone equivalence of $s$ and $t$, then $\gamma$ maps the cone $K_s$ onto the cone $K_t$; in this case we also call the cones $K_s$ and $K_t$ \emph{equivalent}. The following lemma tells us that the graphs $\Gcal(s)$ and $\Gcal(t)$ are isomorphic whenever $s$ and $t$ are cone equivalent.

\begin{lem}
\label{lem:coneequivalence}
Any $\gamma \in \Gamma$ witnessing the cone equivalence of two vertices $s$ and $t$ of $T$ can be extended to a graph isomorphism between the graphs $\Gcal(s)$ and $\Gcal(t)$. 
\end{lem}
\begin{proof}
Recall that by definition 
\[
\Gcal(s)=\left(\Vcal(s), \Ecal(s) \cup \bigcup_{e \colon e^-=s} \Ecal(e)\right),
\] 
where $\Ecal(s)=E(G[\Vcal(s)])\setminus E(G[\Vcal({s^\ua})])$. Note that $e \in \Ecal(s)$ if and only if $\gamma(e) \in \Ecal(t)$, because $\gamma$ maps $s^\ua$ onto $t^\ua$. We extend $\gamma$ in the natural way to virtual edges in $\Ecal(s,u)$ by mapping $e$ onto the edge $\gamma(e) \in \Ecal({t,\gamma(u)})$ connecting $\gamma(e^-)$ and $\gamma(e^+)$. The result is a graph isomorphism between $\Gcal(s)$ and $\Gcal(t)$.
\end{proof}

\begin{lem}\label{lem:finiteconetypes}
Let $(T, \Vcal, r)$ be a reduced, rooted tree decomposition of a locally finite graph $G$. If there is a subgroup $\Gamma \leq \AUT(G)$  such that $(T, \Vcal)$ is $\Gamma$-invariant and the action of $\Gamma$ on $E(T)$ has finitely many orbits, then the number of cone types of $(T,\Vcal,r)$ is finite.
\end{lem}
\begin{proof}
Let $s \sim_K t$ be two vertices in $V(T)\setminus \{r\}$. Then there is a $\gamma \in \Gamma$ mapping the edge $s^\ua s$ onto $t^\ua t$ while preserving direction. This implies that the number of cone types can be at most two times the number of edge orbits of $\Gamma$ acting on $T$ plus one, where the additional type is the type of the root $r$, the only vertex without a parent. In particular the number of cone types is finite.
\end{proof}

\section{Configurations} \label{sec:configs}

Let $\Tcal=(T,\Vcal,r)$ be a rooted tree decomposition of a simple, locally finite, connected, rooted graph $(G,o)$. A \emph{configuration} on $S \subseteq V(T)$ with respect to $\Tcal$ is a map $C=(P,X)$ assigning to each vertex $s \in S$ a pair $C(s)=(P(s),X(s))$, such that for every $s \in S$ one of the following alternatives holds.
\begin{enumerate}[label=(\alph*)]
    \item $X(s) \in V(T)$ is either $s$ or a neighbour of $s$ in $T$, and $P(s)$ is a self-avoiding walk on $\Gcal(s)$ starting in $\Vcal({s^\ua,s})$, or at $o$ if $s=r$. If $X(s) \neq s$, then $P(s)$ ends in $\Vcal({s,X(s)})$. Moreover, if $X(s)=s$ then the last edge of $P(s)$ is a non-virtual edge.
    \item $X(s)=s^\ua$ and $P(s)=\emptyset$ is the empty walk; this is called the \emph{empty configuration} and can only occur for $s \neq r$.
\end{enumerate}
We call $X(s)$ the \emph{exit direction} of $s$.
A \emph{configuration} on a vertex $s$ of $T$ is an image pair $C(s)=(P(s),X(s))$ of a configuration $C=(P,X)$ on the set $S=\{s\}$. Note that by Lemma~\ref{lem:coneequivalence}, the sets of configurations on two cone equivalent vertices are the same up to isomorphism.

Intuitively, configurations model the behaviour of SAWs on single parts of the tree decomposition in the following way. Let $p$ be a self-avoiding walk on $G$ starting at the root $o$. For any $t \in V(T)$ we define a projection $p_t$ of $p$ onto the graph $\Gcal(t)$. First take all vertices and edges of $p$ contained in $\Gcal(t)$ to obtain the multi-walk $p \cap \Gcal(s,t)$. Every detour of $p$ in some other part $\Vcal(s)$ with $s$ adjacent to $t$ in $T$ corresponds to a virtual edge of $\Ecal(s,t)$ connecting the same endpoints as the detour. By replacing these detours by their ``shortcuts'', we end up with a walk $p_t$ on $\Gcal(t)$. Note that $p_t$ might be the empty walk for many vertices $t$. Let $u$ be the vertex of $T$ such that the final edge of $p$ is contained in $\Gcal(u)$. Let $x_u=u$ and for $t \neq u$ let $x_t$ be the neighbour of $t$ on the unique $t$--$u$-path in $T$. Then the function $C$ defined by $C(t)=(p_t,x_t)$ defines a configuration on $V(T)$ with respect to $\Tcal$. This shows that starting from a SAW, we can give a configuration describing the behaviour of the walk when restricted to single parts.

In order to be able to reverse the above construction, we would like to combine configurations on the single parts into SAWs on $G$. To this end, two more properties are needed. Firstly, since SAWs are finite, only finitely many parts can make non-trivial contributions. Secondly, configurations on the parts that contribute non-trivially must fit together in a certain way. These two properties are implied by the notions of boundedness and consistency of configurations defined below. In what follows, let $C=(P,X)$ be a configuration on $S \subseteq V(T)$.

The \emph{weight} $\norm{C}$ of $C$ is the total number of non-virtual edges contained in all the walks $P(s)$ for $s \in S$, so $\norm{C}=\sum_{s \in S} \norm{C(s)}$, where $\norm{C(s)}$ denotes the number of non-virtual edges in $P(s)$. 
The configuration $C$ is called \emph{boring on} $s \in S \setminus\{r\}$ (we also say that $C(s)$ is \emph{boring}) if $X(s)=s^\ua$ and $P(s)$ contains only edges in $\Ecal({s^\ua,s})$. In particular, the empty configuration is boring and all boring configurations have weight 0. Call a configuration $C$ \emph{bounded}, if $C(s)$ is boring for all but finitely many $s \in S$.

Let $s,t \in S$ be adjacent vertices; without loss of generality assume $s=t^\ua$. The configurations $C(s)$ and $C(t)$ are called \emph{compatible}, if either $P(s) \cap \Vcal(s,t)=\emptyset$ and $C(t)$ is the empty configuration, or if they satisfy the following four conditions.
\begin{enumerate}[label=(C\arabic*)]
    \item \label{itm:compatible-intersection}
    The ordered sequences of vertices obtained by intersecting the walks $P(s)$ and $P(t)$ with $\Vcal(s,t)$ coincide,
    \[
    P(s) \cap \Vcal(s,t)=(v_1,\dots,v_l)= P(t) \cap \Vcal(s,t).
    \]
    \item \label{itm:compatible-alternate}
    For every $i \in \{1,\dots,l-1\}$
    \[
    v_i P(s) v_{i+1} \cap \Ecal(s,t) = \emptyset \iff v_i P(t) v_{i+1} \cap \Ecal(s,t) \neq \emptyset.
    \]
    \item \label{itm:compatible-exitdirection}
    $X(s)=t \iff X(t) \neq s$.
    \item \label{itm:compatible-endvertex}
    If $X(s)=t$, then $P(s)$ ends in $v_l$, otherwise $P(t)$ ends in $v_l$.
\end{enumerate}
The configuration $C$ is called \emph{consistent}, if the configurations $C(t^\ua)$ and $C(t)$ are compatible whenever both $t$ and $t^\ua$ are in $S$.

Configurations on the complete vertex set $V(T)$ of the tree decomposition $\Tcal$ are called configurations on $\Tcal$ and the set of all bounded consistent configurations on $\Tcal$ will be denoted by $\Ccal_{\Tcal}$.

\begin{rmk} \label{rmk:sink}
By \ref{itm:compatible-exitdirection}, a consistent configuration $C=(P,X) \in \Ccal_\Tcal$  induces an orientation of the edges of $T$. Clearly any vertex $s \in V(T)$ can be incident to at most one vertex $t$ with $X(t) \neq s$, namely the vertex $X(s)$ in the case $X(s) \neq s$. It is not hard to see that if there is a vertex $s$ with $X(s)=s$, then for every other vertex $t$ of $T$, $X(t)$ lies on the unique $t$--$s$-path in $T$; in other words, $X(t)$ points towards $s$. In particular there can be at most one such vertex $s$. Also note that in the case where $C$ is bounded there is exactly one vertex $s$ with $X(s)=s$. This vertex $s$ can be found by starting at any vertex of $T$ and following exit directions.
\end{rmk}

Let us go back to the Cayley graph $G$ from Figure~\ref{fig:caleygraph}. Using the decomposition tree $T$ and the $t$-graphs from Figure~\ref{fig:partgraph}, an example of a bounded consistent configuration $(W,X)$ on $T$ is shown in Figure~\ref{fig:configurations}. Note that there are only 3 vertices carrying non-boring configurations and that all exit directions point towards the unique vertex $s$ of $T$ with $X(s)=s$.

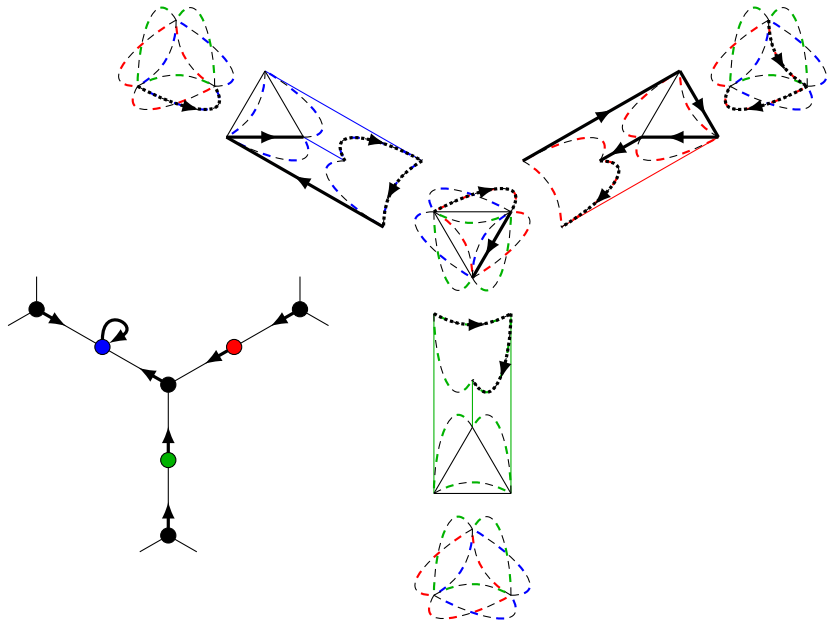
\begin{figure}[bt]
	\centering
	\begin{tikzpicture}
	\begin{scope}[scale=0.45]
	\def\distt{3}
	\def\dista{4}
	\def\diamr{1.3}
	\def\diama{1.3}
	\foreach \i [evaluate=\i as \ieva using {int(mod(\i+1,3))},evaluate=\i as \ievb using {int(mod(\i+2,3))}] in {0,1,2} {
		\draw (120*\i+30:\diamr)--(120*\i+150:\diamr);
		\coordinate (A\i) at (120*\i+30:\distt);
		\coordinate (B\i) at (120*\i+30:{\distt+\dista});
		\coordinate (C\i) at (120*\i+30:{2*\distt+\dista});
		\draw[gen\i] (B\i)+(120*\i+210:\diama) -- ($(A\i)+(120*\i+30:\diamr)$);
		\draw[gen\i] (B\i)+(120*\i+90:\diama) -- ($(A\i)+(120*\i+150:\diamr)$);
		\draw[gen\i] (B\i)+(120*\i+330:\diama) -- ($(A\i)+(120*\i+270:\diamr)$);
		\draw[virtual=gen\i]  (120*\i+30:\diamr) to[in={120*\i+30}, out={120*\i+60},looseness=1.5] (120*\i+150:\diamr);
		\draw[virtual=gen\i]  (120*\i+270:\diamr) to[in={120*\i}, out={120*\i+30},looseness=1.5] (120*\i+30:\diamr);
		\draw[virtual=gen\i]  (120*\i+150:\diamr) to[in={120*\i+90}, out={120*\i-30}] (120*\i+270:\diamr);
		\draw[virtual=gen\i]  (A\i)+(120*\i+30:\diamr) to[in={120*\i+30}, out={120*\i+60},looseness=1.5] +(120*\i+150:\diamr);
		\draw[virtual=gen\i]  (A\i)+(120*\i+270:\diamr) to[in={120*\i}, out={120*\i+30},looseness=1.5] +(120*\i+30:\diamr);
		\draw[virtual=gen\i]  (A\i)+(120*\i+150:\diamr) to[in={120*\i+90}, out={120*\i-30}] +(120*\i+270:\diamr);
		\draw[virtual=gen\i]  (B\i)+(120*\i+210:\diamr) to[in={120*\i+210}, out={120*\i+240},looseness=1.5] +(120*\i+330:\diamr);
		\draw[virtual=gen\i]  (B\i)+(120*\i+90:\diamr) to[in={120*\i+180}, out={120*\i+210},looseness=1.5] +(120*\i+210:\diamr);
		\draw[virtual=gen\i]  (B\i)+(120*\i+330:\diamr) to[in={120*\i+270}, out={120*\i+150}] +(120*\i+90:\diamr);
	}	
	\foreach \j in {0,1,2} {
	    \foreach \i in {0,1,2}{
		    \draw (B\i)+(120*\i+120*\j+210:\diama) -- +(120*\i+120*\j+90:\diama);
	    	\draw[virtual=gen\i]  (C\j)+(120*\i+210:\diamr) to[in={120*\i+210}, out={120*\i+240},looseness=1.5] +(120*\i+330:\diamr);
		    \draw[virtual=gen\i]  (C\j)+(120*\i+90:\diamr) to[in={120*\i+180}, out={120*\i+210},looseness=1.5] +(120*\i+210:\diamr);
		    \draw[virtual=gen\i]  (C\j)+(120*\i+330:\diamr) to[in={120*\i+270}, out={120*\i+150}] +(120*\i+90:\diamr);
		}
	}
    \begin{scope}[walkarrow] 
        \draw[postaction={decorate},densely dotted] (150:\diamr) to[out=30, in=60,looseness=1.5] (30:\diamr);
        \draw[postaction={decorate}] (30:\diamr) -- (270:\diamr);
        \draw[postaction={decorate}] (A0)+(150:\diama) -- ($(B0)+(90:\diama)$);
        \draw[postaction={decorate}] ($(B0)+(90:\diama)$) -- ($(B0)+(-30:\diama)$);
        \draw[postaction={decorate}] ($(B0)+(-30:\diama)$) -- ($(B0)+(210:\diama)$);
        \draw[postaction={decorate}] ($(B0)+(210:\diama)$) -- ($(A0)+(30:\diama)$);
        \draw[postaction={decorate},densely dotted] ($(A0)+(30:\diama)$) to[out=0, in=30,looseness=1.5] ($(A0)+(-90:\diama)$);
        \draw[postaction={decorate},densely dotted] ($(A1)+(150:\diama)$) to[out=120, in=150,looseness=1.5] ($(A1)+(30:\diama)$);
        \draw[postaction={decorate},densely dotted] ($(A1)+(30:\diama)$) to[out=210, in=90] ($(A1)+(-90:\diama)$);
        \draw[postaction={decorate}] ($(A1)+(-90:\diama)$) -- ($(B1)+(210:\diama)$);
        \draw[postaction={decorate}] ($(B1)+(210:\diama)$) -- ($(B1)+(-30:\diama)$);
        \draw[postaction={decorate},densely dotted] ($(A2)+(150:\diama)$) to[out=-30, in=210] ($(A2)+(30:\diama)$);
        \draw[postaction={decorate},densely dotted] ($(A2)+(30:\diama)$) to[out=270, in=-60,looseness=1.5] ($(A2)+(-90:\diama)$);
        \draw[postaction={decorate},densely dotted] ($(C0)+(90:\diama)$) to[out=-90, in=150] ($(C0)+(-30:\diama)$);
        \draw[postaction={decorate},densely dotted] ($(C0)+(-30:\diama)$) to[out=210, in=240,looseness=1.5] ($(C0)+(210:\diama)$);
        \draw[postaction={decorate},densely dotted] ($(C1)+(210:\diama)$) to[out=-30, in=-60,looseness=1.5] ($(C1)+(-30:\diama)$);
    \end{scope}
	\end{scope}
	\begin{scope}[scale=0.5, xshift=-8cm,yshift=-4cm]
	\def\dista{4}
	\def\diamr{0.2}
	\def\diama{0.2}
	\def\diamb{0.2}
	\def\diamc{0.2}
    \draw[fill=black] (0,0) circle (\diamr);
	\foreach \i in {0,1,2} {
		\coordinate (A\i) at (120*\i+30:\dista);
		\draw[fill=black] (A\i) circle (\diama);
		\draw (A\i)+(120*\i+210:\diama) --coordinate[midway](m) (120*\i+30:\diamr);
	    \draw[fill=gen\i] (m) circle (\diama);
		\draw (A\i)++(120*\i+90:\diama) -- +(120*\i+90:\dista/6);
	    \draw (A\i)++(120*\i+330:\diama) -- +(120*\i+330:\dista/6);
	}
	\draw[very thick,-latex] (0,0)++(150:\diama) -- +(150:0.7);
    \draw[very thick,-latex] (30:{\dista/2})++(210:\diama) -- +(210:0.7);
    \draw[very thick,-latex] (-90:{\dista/2})++(90:\diama) -- +(90:0.7);
    \draw[very thick,-latex] (30:{\dista})++(210:\diama) -- +(210:0.7);
    \draw[very thick,-latex] (150:{\dista})++(-30:\diama) -- +(-30:0.7);
    \draw[very thick,-latex] (-90:{\dista})++(90:\diama) -- +(90:0.7);
    \draw[very thick,-latex] (150:{\dista/2})+(90:\diama) to[out=90,in=30,looseness=12] +(30:\diama);
	\end{scope}
  	\end{tikzpicture}
  	\caption{A bounded consistent configuration $(W,X)$ on the tree decomposition $\Tcal$ of the Cayley graph $G$. Edges of the walk $W(t)$ on the $t$-graph $\Gcal(t)$ are drawn bold and decorated with arrows according to their direction. Exit directions of vertices of $T$ are also denoted by arrows pointing from a vertex $t$ to a vertex $s$ if $X(t)=s$.}
  	\label{fig:configurations}
\end{figure}

The following extension lemma can be seen as the reason why boring configurations are indeed not interesting to us. More precisely, it shows that a bounded consistent configuration on $\Tcal$, is uniquely determined by the (finitely many) non-boring configurations. Moreover, it tells us that under certain conditions a consistent configuration on a finite set $S \subseteq V(T)$ can be extended to a bounded consistent configuration on $\Tcal$.

\begin{lem}\label{lem:boringextension}
Let $s,t \in V(T)$ such that $s=t^\ua$ and let $c_s=(p_s,x_s)$ be a configuration on $s$ such that $p_s \cap \Ecal(s,t)= \emptyset$ and $x_s \neq t$. Then there is a unique configuration $c_t$ on $t$ which is compatible with $c_s$, and this configuration $c_t$ is boring.
\end{lem}
\begin{proof}
Suppose $c_t=(p_t,x_t)$ is a configuration on $t$ such that $c_t$ and $c_s$ are compatible. Then $x_t=s$ by \ref{itm:compatible-exitdirection} and $p_t \cap \Vcal(s,t)=(v_1, \dots, v_l)=p_s \cap \Vcal(s,t)$ by \ref{itm:compatible-intersection}. Property~\ref{itm:compatible-alternate} implies that the sub-walks $v_i p_t v_{i+1}$ contain only the virtual edge $v_iv_{i+1}$ in $\Ecal(s,t)$ and by the definition of configurations and \ref{itm:compatible-endvertex} the walk $p_t$ starts at $v_1$ and ends at $v_l$. We conclude that the configuration $c_t$ on $t$ is unique and boring. Moreover, the considerations above can be used to construct such a configuration, and in particular such a configuration exists.
\end{proof}

Our goal in this section is to establish a one-to-one correspondence $\psi_r$ between bounded consistent configurations $\Ccal$ on the rooted tree decomposition $\Tcal=(T,\Vcal,r)$ and self-avoiding walks of length at least $1$ on the underlying graph $G$ starting at its root $o$. The main idea is to contract the sub-tree induced by all vertices of $T$ carrying non-boring configurations to a single vertex. By also contracting the corresponding configurations, only a single non-boring configuration remains; its walk is a walk on $G$ and will be denoted $\psi_r(C)$.

In Section~\ref{sec:treedecomp} we already discussed how to contract a set $F$ of edges of a tree decomposition $\Tcal$ to obtain a contracted tree decomposition $\Tcal/F$. Let us repeat this process for rooted tree decompositions and configurations on those tree decompositions. As we are only interested in contractions of finite sets of edges, we first focus on the special case where a single edge is contracted.

The following definition of the contraction of a rooted tree decomposition coincides with our earlier definition of contractions of tree decompositions; the root part of the contraction is simply the equivalence class of $r$. We still give a detailed definition since we would like to introduce some notation.

Let $\Tcal=(T,\Vcal,r)$ be a rooted tree decomposition of a simple, locally finite, connected, rooted graph $(G,o)$ and let $f \in E(T)$. We may without loss of generality, assume that $f^- = (f^+)^\ua$ (if not, use the reversed edge). Define the contraction $\Tcal/f = (T/f,\Vcal/f,r/f)$ as follows.

The tree $T/f$ is obtained from $T$ by identifying the two endpoints $f^-$ and $f^+$ of $f$ and deleting the edge $f$. More precisely, $T/f$ can be described as follows. The vertex set of $T/f$ is obtained from the vertex set of $T$ by replacing $f^-$ and $f^+$ by a single vertex $t_f$.  Every edge $e \in E(T)\setminus \{f\}$ not incident to $f$ corresponds to an edge in $T/f$ with the same endpoints. Every edge $e = st$  of $T$ where $t$ is an endpoint of $f$ corresponds to an edge connecting $s$ and $t_f$ in $T/f$. We abuse notation and denote the edge corresponding to $e$ in $T/f$ by $e$ as well. The part $\Vcal/f(t_f)$ is defined as $\Vcal(f^-) \cup \Vcal(f^+)$; for every other vertex of $T/f$ we define $\Vcal/f(t) = \Vcal(t)$. Finally, if $r$ is incident to $f$, then let $r/f = t_f$, otherwise let $r/f = r$. 

Denote the parent of $t\in V(T/f)$ by $t^{\ua/f}$. From the assumption $f^- = (f^+)^\ua$ it follows that $(t_f)^{\ua/f} = (f^-)^\ua$, unless $f^- = r$, in this case $t_f = r/f$ has no parent. For every other vertex of $T/f$ we have $t^{\ua/f} = t_f$ if $t^\ua \in \{f^-, f^+\}$, and $t^{\ua/f} = t^\ua$ otherwise. Note that if an edge $e \in E(T) \setminus \{f\}$ connects $t$ to $t^\ua$ (or $f^-$ to $(f^-)^\ua$), then the corresponding edge in $T/f$ that is also denoted by $e$ connects $t$ to $t^{\ua/f}$ (or $t_f$ to $(t_f)^{\ua/f}$).

For $e \in E(T/f)$ let $\Vcal/f(e)$ and $\Ecal/f(e)$ denote the adhesion set corresponding to $e$ and the set of $e$-edges with respect to the tree decomposition $\Tcal/f$, respectively. For $t\in V(T/f)$ let $\Ecal/f(t)$  and $\Gcal/f(t)$ denote the set of $t$-edges and the $t$-graph with respect to the tree decomposition $\Tcal/f$, respectively. Using property~\ref{itm:td-nocrosssedge} of tree decompositions, it is not hard to see that $\Vcal/f(e) = \Vcal(e)$, $\Ecal/f(e) = \Ecal(e)$, $\Ecal/f(t_f) = \Ecal(f^-) \cup \Ecal (f^+)$, and $\Ecal/f(t) = \Ecal (t)$ for $t \neq t_f$.  It follows that 
\[\Gcal/f(t) = 
\begin{cases}
(\Gcal(f^-) \cup \Gcal(f^+)) - \Ecal(f) & \text{if }t = t_f, \\
\Gcal(t) & \text{otherwise.}
\end{cases}
\]

Next we define contractions of configurations. Let $C = (P,X)$ be a bounded consistent configuration on $\Tcal$. For the definition of the contracted configuration $C/f$, assume again without loss of generality that $f^- = (f^+)^\ua$, and let $P(f^-) \cap \Vcal(f) = (v_1, \dots, v_l) = P(f^+) \cap \Vcal(f)$, where the last equality follows from \ref{itm:compatible-intersection}. Let $t_0 = f^-$. For $1 \leq j \leq l-1$, let $t_j \in \{f^-,f^+\}$ be such that $P(t_j) \cap \Ecal(f) = \emptyset$; note that this uniquely defines a vertex by~\ref{itm:compatible-alternate}. If $X(f^-) = f^+$, then let $t_l = f^+$, otherwise let $t_l=f^-$. Define the walk $p_f$ as the concatenation
\[
    P(t_0)v_1P(t_1)v_2\dots v_lP(t_l).
\]
In other words, $p_f$ is obtained from $P(f^-)$ and $P(f^+)$ by deleting all edges in $\Ecal(f)$ and then piecing the walk components of the resulting multi-walks together in a consistent manner, see Figure~\ref{fig:defcontraction}.

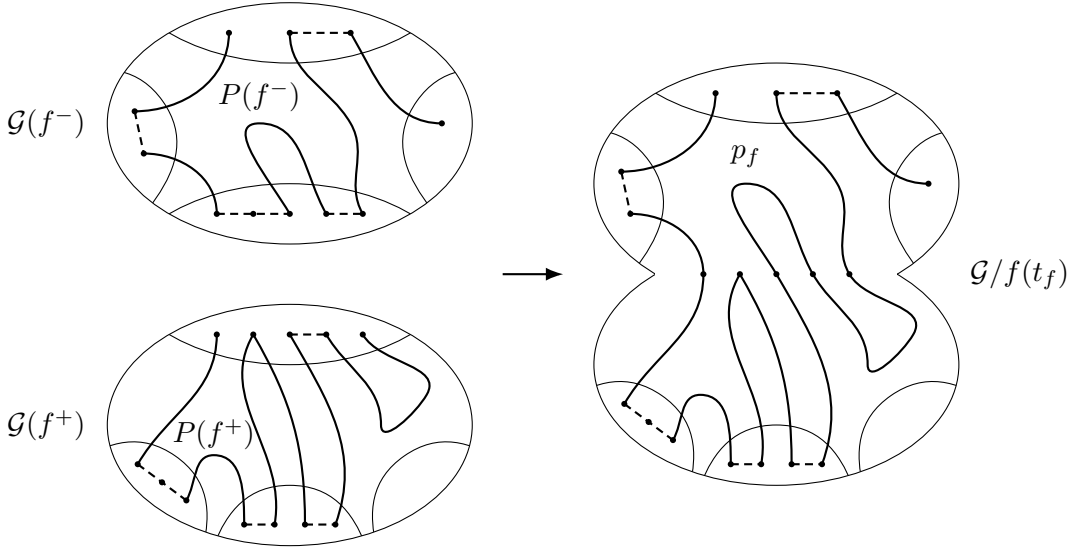
\begin{figure}
	\centering
	\begin{tikzpicture}
	\def\rad{0.1}
	\begin{scope}[scale=0.4]
        \begin{scope}
            \node at (-8,0) {$ \Gcal(f^-)$};
            \node at (-1,0.9) {$P(f^-)$};
            \draw[clip] (0,0) ellipse (6cm and 4cm);
            \draw (0,6) ellipse (6cm and 4cm);
            \draw (0,-6) ellipse (6cm and 4cm);
            \draw[rotate=-55,xshift=-3cm,yshift=-1cm] (-3.5,-8) to[out=75,in=180] (0,-2) to[out=0,in=105] (3.5,-8);
        	\draw[rotate=+55,xshift=+3cm,yshift=-1cm] (-3.5,-8) to[out=75,in=180] (0,-2) to[out=0,in=105] (3.5,-8);
        	\draw[thick] (-2,3) to[out=-90,in=0] (-5.1,0.4) (-4.8,-1) to[out=0,in=90] (-2.4,-3) (0,-3) to[out=120,in=180] (-1,0) to[out=0,in=120] (1.2,-3) (2.4,-3) to[out=120,in=-60] (2,0) to[out=120,in=270] (0,3) (2,3) to[out=-60,in=180] (5,0);
            \draw[thick, dash pattern= on 3pt off 2pt] (-5.1,0.4) -- (-4.8,-1) (0,3) -- (2,3) (-2.4,-3) -- (-1.2,-3) -- (0,-3) (1.2,-3) -- (2.4,-3);
            \fill (-2,3) circle (\rad);
        	\fill (0,3) circle (\rad);
           	\fill (2,3) circle (\rad);
        	\fill (-5.1,0.4) circle (\rad);
        	\fill (-4.8,-1) circle (\rad);
        	\fill (5,0) circle (\rad);
        	\fill (-2.4,-3) circle (\rad);
        	\fill (-1.2,-3) circle (\rad);
           	\fill (0,-3) circle (\rad);
        	\fill (1.2,-3) circle (\rad);
        	\fill (2.4,-3) circle (\rad);
        \end{scope}
        \begin{scope}[yshift=-10cm]
            \node at (-8,0) {$ \Gcal(f^+)$};
            \node at (-2.5,-0.3) {$P(f^+)$};
        	\draw[clip](0,0) ellipse (6cm and 4cm);
        	\draw (0,6) ellipse (6cm and 4cm);
        	\draw[rotate=-35,xshift=-2.5cm,yshift=-1cm] (-3.5,-8) to[out=75,in=180] (0,-2) to[out=0,in=105] (3.5,-8);
        	\draw (-3.5,-8) to[out=75,in=180] (0,-2) to[out=0,in=105] (3.5,-8);
        	\draw[rotate=+35,xshift=+2.5cm,yshift=-1cm] (-3.5,-8) to[out=75,in=180] (0,-2) to[out=0,in=105] (3.5,-8);
        	\draw[thick] (-2.4,3) to[out=-90,in=60] (-5,-1.3) (-3.4,-2.5) to[out=60,in=180] (-2.5,-1) to[out=0,in=90] (-1.5,-3.3) (-0.5,-3.3) to[out=80,in=240] (-1.2,3) to[out=-60,in=90] (0.5,-3.3) (1.5,-3.3) to[out=70,in=-60] (0,3) (1.2,3) to[out=-60,in=90] (3,0) to[out=270,in=240] (4.5,1) to[out=60,in=-60] (2.4,3);
        	\draw[thick, dash pattern= on 3pt off 2pt] (-5,-1.3) -- (-4.2,-1.9) -- (-3.4,-2.5) (-1.5,-3.3) -- (-0.5,-3.3) (0.5,-3.3) -- (1.5,-3.3) (0,3) -- (1.2,3);
        	\fill (-2.4,3) circle (\rad);
        	\fill (-1.2,3) circle (\rad);
           	\fill (0,3) circle (\rad);
        	\fill (1.2,3) circle (\rad);
        	\fill (2.4,3) circle (\rad);
        	\fill (-5,-1.3) circle (\rad);
        	\fill (-4.2,-1.9) circle (\rad);
        	\fill (-3.4,-2.5) circle (\rad);
        	\fill (-1.5,-3.3) circle (\rad);
        	\fill (-0.5,-3.3) circle (\rad);
        	\fill (0.5,-3.3) circle (\rad);
        	\fill (1.5,-3.3) circle (\rad);
        \end{scope}
        \draw[thick, -Latex] (7,-5) -- (9,-5);
	\end{scope}
	\begin{scope}[scale=0.4,xshift=16cm, yshift=-2cm]
	    \node at (8,-3) {$ \Gcal/f(t_f)$};
        \begin{scope}
            \draw[clip] (0,0) ellipse (6cm and 4cm);
            \draw (0,6) ellipse (6cm and 4cm);
            \draw[rotate=-55,xshift=-3cm,yshift=-1cm] (-3.5,-8) to[out=75,in=180] (0,-2) to[out=0,in=105] (3.5,-8);
        	\draw[rotate=+55,xshift=+3cm,yshift=-1cm] (-3.5,-8) to[out=75,in=180] (0,-2) to[out=0,in=105] (3.5,-8);
        \end{scope}
        \begin{scope}[yshift=-6cm]
        	\draw[clip](0,0) ellipse (6cm and 4cm);
        	\draw[rotate=-35,xshift=-2.5cm,yshift=-1cm] (-3.5,-8) to[out=75,in=180] (0,-2) to[out=0,in=105] (3.5,-8);
        	\draw (-3.5,-8) to[out=75,in=180] (0,-2) to[out=0,in=105] (3.5,-8);
        	\draw[rotate=+35,xshift=+2.5cm,yshift=-1cm] (-3.5,-8) to[out=75,in=180] (0,-2) to[out=0,in=105] (3.5,-8);
        \end{scope}
        \begin{scope}
            \node at (-1,0.9) {$p_f$};
            \fill[white] (-{sqrt(7*9)/2},-1.8) rectangle +({sqrt(7*9)},-2.4); 
        	\draw[thick] (-2,3) to[out=-90,in=0] (-5.1,0.4) (-4.8,-1) to[out=0,in=90] (-2.4,-3) (0,-3) to[out=120,in=180] (-1,0) to[out=0,in=120] (1.2,-3) (2.4,-3) to[out=120,in=-60] (2,0) to[out=120,in=270] (0,3) (2,3) to[out=-60,in=180] (5,0);
            \draw[thick, dash pattern= on 3pt off 2pt] (-5.1,0.4) -- (-4.8,-1) (0,3) -- (2,3);
            \fill (-2,3) circle (\rad);
        	\fill (0,3) circle (\rad);
           	\fill (2,3) circle (\rad);
        	\fill (-5.1,0.4) circle (\rad);
        	\fill (-4.8,-1) circle (\rad);
        	\fill (5,0) circle (\rad);
        \end{scope}
        \begin{scope}[yshift=-6cm]
        	\draw[thick] (-2.4,3) to[out=-90,in=60] (-5,-1.3) (-3.4,-2.5) to[out=60,in=180] (-2.5,-1) to[out=0,in=90] (-1.5,-3.3) (-0.5,-3.3) to[out=80,in=240] (-1.2,3) to[out=-60,in=90] (0.5,-3.3) (1.5,-3.3) to[out=70,in=-60] (0,3) (1.2,3) to[out=-60,in=90] (3,0) to[out=270,in=240] (4.5,1) to[out=60,in=-60] (2.4,3);
        	\draw[thick, dash pattern= on 3pt off 2pt] (-5,-1.3) -- (-4.2,-1.9) -- (-3.4,-2.5) (-1.5,-3.3) -- (-0.5,-3.3) (0.5,-3.3) -- (1.5,-3.3);
        	\fill (-2.4,3) circle (\rad);
        	\fill (-1.2,3) circle (\rad);
           	\fill (0,3) circle (\rad);
        	\fill (1.2,3) circle (\rad);
        	\fill (2.4,3) circle (\rad);
        	\fill (-5,-1.3) circle (\rad);
        	\fill (-4.2,-1.9) circle (\rad);
        	\fill (-3.4,-2.5) circle (\rad);
        	\fill (-1.5,-3.3) circle (\rad);
        	\fill (-0.5,-3.3) circle (\rad);
        	\fill (0.5,-3.3) circle (\rad);
        	\fill (1.5,-3.3) circle (\rad);
        \end{scope}
	\end{scope}
  	\end{tikzpicture}
  	\caption{Combining walks $P(f^-)$ and $P(f^+)$ of compatible configurations on the endpoints of $f$ into a walk $p_f$ on $\Gcal/f(t_f)$.}
  	\label{fig:defcontraction}
\end{figure}

The contraction $C/f = (P/f,X/f)$ of the configuration $C$ is defined as follows. For the contracted vertex $t_f$, let
\[
    P/f(t_f) = p_f 
    \qquad \text{and} \qquad 
    X/f(t_f) = 
    \begin{cases}
    X(f^-) & \text{if } X(f^-) \notin \{f^-,f^+\},\\
    X(f^+) & \text{if } X(f^+) \notin \{f^-,f^+\},\\
    t_f & \text{otherwise}.
    \end{cases}
\]
Note that by \ref{itm:compatible-exitdirection}, the conditions in the first two cases in the definition of $X/f$ cannot be satisfied simultaneously, and in the third case $X(f^-) = X(f^+) \in \{f^-,f^+\}$ holds. For $t \neq t_f$ we define
\[
    P/f(t) = P(t)
    \qquad \text{and} \qquad 
    X/f(t) =
    \begin{cases}
        X(t) &\text{if }X(t) \notin \{f^-,f^+\},\\
        t_f &\text{otherwise}.
    \end{cases}
\]
Note the similarity between the definition of $X/f(t)$ and our observations about $t^{\ua/f}$ above; clearly, if $X(t) = t^\ua$, then $X/f(t) = t^{\ua/f}$.

\begin{lem}
\label{lem:walkcontraction}
The walk $p_f$ is a self-avoiding walk on $\Gcal/f(t_f)$ satisfying $p_f \cap \Gcal(f^-) = P(f^-)  - \Ecal(f)$ and $p_f \cap \Gcal(f^+) = P(f^+)  - \Ecal(f)$. In particular, the set of edges contained in $p_f$ consists of the edge sets of $P(f^-)  - \Ecal(f)$ and $P(f^+)  - \Ecal(f)$.
\end{lem}

\begin{proof}
If $P(f^+)$ is the empty walk, then $p_f = P(f^-)$ and all claimed properties are trivially satisfied, so assume that $P(f^+) \neq \emptyset$. Since $C$ is a configuration, $P(f^+)$ must start in $v_1$, that is, $P(f^+)v_1$ is a trivial walk only consisting of $v_1$. By \ref{itm:compatible-alternate}, if $t_j = f^-$, then $P(f^+)$ contains the edge $v_jv_{j+1} \in \Ecal(f)$, and if $t_j = f^+$, then $P(f^-)$ contains the edge $v_jv_{j+1} \in \Ecal(f)$. Properties \ref{itm:compatible-exitdirection} and \ref{itm:compatible-endvertex} imply that if $t_l = f^+$, then $v_lP(t^-)$ is trivial and vice versa. Combining these observations with the fact that $P(f^-)$ can be decomposed as $P(f^-) = P(f^-)v_1P(f^-)v_2\dots v_l P(f^-)$
we conclude that 
\[p_f \cap \Gcal(f^-) = P(f^-)  - \Ecal(f),\] and similarly for $f^+$. This implies that $p_f$ uses no vertex more than once: for vertices in $\Vcal(f)$, this holds by definition, for vertices outside of $\Vcal(f)$, this follows from the fact that $P(f^-)$ and $P(f^+)$ are self-avoiding. Hence $p_f$ is self-avoiding.
\end{proof}

The following lemma shows that $C/f$ as defined above is indeed a bounded consistent configuration on $\Tcal/f$.

\begin{lem}\label{lem:contraction}
Let $\Tcal=(T,\Vcal,r)$ be a rooted tree decomposition of the simple, locally finite, connected, rooted graph $(G,o)$, let $C \in \Ccal_\Tcal$, and let $f\in E(T)$. Then $C/f \in \Ccal_{\Tcal/f}$ and $\norm{C/f}=\norm{C}$.
\end{lem}
\begin{proof}
We start by showing that $C/f=(P/f,X/f)$ is a configuration on $\Tcal/f$. First consider $t \neq t_f$.  If $C(t)$ is empty, then $P/f(t) = P(t) = \emptyset$ and $X(t) = t^\ua$, so by the above observation $X/f(t) = t^{\ua/f}$. This shows that $C/f(t)$ is the empty configuration. If $C(t)$ is non-empty, then $P(t) = P/f(t)$ is a non-empty self-avoiding walk on $\Gcal(t) = \Gcal/f(t)$ starting in $\Vcal(t^\ua,t) = \Vcal/f(t^{\ua/f},t)$, or in $o$, if $t = r$, and ending in $\Vcal(t,X(t)) = \Vcal/f(t,X/f(t))$. In case $X/f(t) = t$, clearly also $X(t) = t$, so in this case $P(t) = P/f(t)$ ends in a non-virtual edge. We conclude that $C/f(t) = (P/f(t),X/f(t))$ is a configuration on $t$.

Now consider the contracted vertex $t_f$. As before, without loss of generality assume that $f^- = (f^+)^\ua$. If $P/f(t_f) = p_f$ is the empty walk, then $P(f^-)$ is the empty walk and thus $C(f^-)$ must be the empty configuration. In particular $X(f^-) = (f^-)^\ua \notin \{f^-,f^+\}$, and thus $X/f(t_f) = X(f^-) = (t_f)^{\ua/f}$, showing that $C/f(t_f)$ is the empty configuration. 

So we may assume that $p_f$ is non-empty. By Lemma~\ref{lem:walkcontraction}, $p_f$ is a self-avoiding walk on $\Gcal/f(t_f)$; it only remains to show that the first and last vertex of $p_f$ lie in the appropriate adhesion sets. 
The first vertex of $p_f$ is the same as the first vertex of $P(f^-)$, consequently it lies in $\Vcal((f^-)^\ua,f^-) = \Vcal/f((t_f)^{\ua/f},t_f)$, or it is equal to $o$ if $t_f = r/f$ and thus $f^- = r$. 
The last vertex of $p_f$ is the last vertex of $P(t_l)$. If $X(f^+) \notin \{f^-,f^+\}$, then $X(f^-) = f^+$, and thus $t_l = f^+$. It follows that the last vertex of $p_f$ lies in $\Vcal(f^+,X(f^+)) = \Vcal/f(t_f, X/f(t_f))$. If $X(f^-) \notin \{f^-,f^+\}$ an analogous argument applies. If both $X(f^-)$ and $X(f^+)$ are contained in $\{f^-,f^+\}$, then $X/f(t_f) = t_f$; in this case $P(t_l)$ ends with a non-virtual edge, and consequently $p_f$ does not end in a virtual edge if $X/f(t_f) = t_f$. We conclude that $C/f(t_f) = (P/f(t_f),X/f(t_f))$ is a configuration on $t_f$.

By construction the number of non-boring parts with respect to $C/f$ is at most the number of non-boring parts with respect to $C$, so $C/f$ is bounded. Moreover $C$ and $C/f$ use the same non-virtual edges, so $\norm{C/f}=\norm{C}$ holds.

It remains to show that $C/f$ is consistent, or in other words, that $C/f(s)$ and $C/f(t)$ are compatible for any edge $st \in E(T/f)$. If $st$ is not incident with $f$, then this follows from the fact that $C$ is consistent, so we may without loss of generality assume that $t=t_f$. We only treat the case where $s$ is a neighbour of $f^-$ in $T$, the case where $s$ and $f^+$ are neighbours is completely analogous.

Note that $\Vcal/f(s,t_f) = \Vcal (s,f^-) \subseteq V(\Gcal(f^-))$. Thus Lemma~\ref{lem:walkcontraction} implies that $p_f \cap \Vcal/f(s,t_f) = P(f^-) \cap \Vcal(s,f^-)$ which in turn implies~\ref{itm:compatible-intersection}. Next, note that $\Ecal/f(s,t_f) = \Ecal (s,f^-) \subseteq E(\Gcal(f^-))$. By Lemma~\ref{lem:walkcontraction} we thus have $u p_f v \cap \Ecal/f(s,t_f) = u P(f^-) v \cap \Ecal(s,f^-)$ for any pair of vertices $u,v$ in $\Vcal/f(s,t_f)$, and~\ref{itm:compatible-alternate} follows. For condition~\ref{itm:compatible-exitdirection} observe that
\[
X/f(t_f) = s \iff X(f^-) = s \iff X(s) \neq f^- \iff X/f(s) \neq t_f.
\]
Finally, note that if $X/f(t_f) = s$, then $X(f^+) = f^-$, and consequently $p_f$ ends in the same vertex as $P(f^-)$, so \ref{itm:compatible-endvertex} is satisfied.
\end{proof}

\begin{lem}\label{lem:contrbij}
Let $\Tcal=(T,\Vcal,r)$ be a rooted tree decomposition of the rooted graph $(G,o)$ and $f \in E(T)$. Then the function $C \mapsto C/f$ bijectively maps $\Ccal_\Tcal$ to $\Ccal_{\Tcal/f}$.
\end{lem}

\begin{proof}
As before, denote by $t_f$ the contracted vertex in $T/f$.
Let $C'=(P',X') \in  \Ccal_{\Tcal/f}$. We show that for any $t \in V(T)$, there is a unique choice $C(t) = (P(t),X(t))$ such that $C$ is a consistent configuration and $C/f = C'$.

First consider $t \notin \{f^-,f^+\}$. Necessarily $P(t) = P'(t)$, otherwise $P/f(t) \neq P'(t)$ by the definition of contraction. Similarly, if $X'(t) \notin \{f^-,f^+\}$, then $X(t) = X'(t)$ as otherwise $X/f(t) \neq X'(t)$. If $X'(t) = t_f$, then $X(t)$ must be either $f^-$ or $f^+$. Moreover, for $C$ to be a configuration, $X(t)$ must be adjacent to $t$, and since $T$ is a tree, $t$ cannot be adjacent to both $f^-$ and $f^+$. So we have shown that $X(t)$ must be the unique neighbour of $t$ in $\{f^-,f^+\}$.

By Lemma~\ref{lem:walkcontraction}, we know that if we want $C/f= C'$, we have to make sure that $P'(t_f) \cap \Gcal(f^-) = P(f^-) - \Ecal(f)$. So $P(f^-)$ can only differ from the multi-walk $P'(t_f) \cap \Gcal(f^-)$ by edges in $\Ecal(f)$. Let $q_1,\dots,q_l$ be the walk components of $P'(t_f) \cap \Gcal(f^-)$. Note that each $q_j$ for $j > 1$ starts in $\Vcal(f)$, and each $q_j$ for $j < l $ ends in $\Vcal(f)$. In particular, it is possible to define a walk $P(f^-) = q_1e_1q_2\dots e_{l-1}q_l$, where $e_j \in \Ecal(f)$ is a virtual edge connecting the last vertex of $q_j$ to the first vertex of $q_{j+1}$. By the above discussion, this is the only choice of $P(f^-)$ for which $P/f(t_f) = P'(t_f)$ can possibly hold. A completely analogous argument applies to $P(f^+)$.

Finally, let us consider the exit directions of $f^-$ and $f^+$. If $X'(t_f) \neq t_f$, then by \ref{itm:compatible-alternate}, there is a unique neighbour $x$ of $t_f$ in $T/f$ such that $X'(x) \neq t_f$ and thus $X'(x) \notin \{f^-,f^+\}$. If $x$ is a neighbour of $f^-$ in $T$, then necessarily $X(f^-) = x$ and $X(f^+) = f^-$, otherwise $C$ is not consistent. Similarly, if $x$ is a neighbour of $f^+$ in $T$, then necessarily $X(f^+) = x$ and $X(f^-) = f^+$. If $X'(t_f) = t_f$, then $X(f^-)=X(f^+) \in \{f^-,f^+\}$, since otherwise either $C/f \neq C'$, or $C$ is not consistent. Note that in this case $P'(t_f)$ ends in a non-virtual edge $e$ because $C'$ is a configuration. If $e \in \Ecal(f^-)$, then $X(f^+) = X(f^-) = f^-$, otherwise $C$ is either not a configuration (if both endpoints of $e$ lie in $\Vcal(f)$), or it is inconsistent due to \ref{itm:compatible-endvertex}. If $e \in \Ecal(f^+)$, then analogously $X(f^+) = X(f^-) = f^+$.

A straight forward check (left to the reader) shows that the above construction indeed gives a  bounded consistent configuration $C = (P,T) \in \Ccal_\Tcal$ with $C/f=C'$.
\end{proof}

Our next goal is to define contraction of finite sets of edges. For this purpose, let $\Tcal =(T,\Vcal,r)$ be a rooted tree-decomposition of a simple, locally finite, connected, rooted graph $(G,o)$ and let $F = \{f_1,\dots,f_k\}$ be a finite subset of $E(T)$. Then we define
\[
    \Tcal/F = \Tcal/f_1/f_2/ \dots /f_k.
\]
We note once again that this definition is consistent with the definition of $\Tcal/F$ given in Section~\ref{sec:treedecomp}. If the set $F$ induces a connected sub-graph of $T$, then there is a unique contracted vertex in $T/F$; we denote it by $t_F$.
Analogously, for a configuration $C$ on $\Tcal$, we define $C/F = (P/F,X/F)$ by
\[
    C/F = C/f_1/f_2/ \dots /f_k.
\]
We would like these definitions to be independent of the order in which the edge contractions are carried out. In order to make sense of this statement, we first need to clarify when we consider two tree decompositions and configurations on them to be the same. Let $\Tcal_1 = (T_1,\Vcal_1,r_1)$ and $\Tcal_2 = (T_2,\Vcal_2,r_2)$ be rooted tree decompositions of the same rooted graph $(G,o)$. We say that $\Tcal_1$ and $\Tcal_2$ are \emph{isomorphic} if there is an isomorphism $\iota\colon T_1 \to T_2$ such that $\iota(r_1) = r_2$ and $\Vcal_1 = \Vcal_2 \circ \iota$. We call two configurations $C_1 = (P_1,X_1)$ on $\Tcal_1$ and $C_2 = (P_2,X_2)$ on $\Tcal_2$ isomorphic, if there is an isomorphism $\iota$ as above additionally satisfying $P_1 = P_2 \circ \iota$ and $\iota \circ X_1 = X_2 \circ \iota$. Since we only care about tree decompositions and configurations up to isomorphism, we write $\Tcal_1 = \Tcal_2$ and $C_1 = C_2$ to denote the fact that the respective tree decompositions and configurations are isomorphic. Inductive application of the following lemma shows that $\Tcal/F$ and $C/F$ (up to isomorphism) indeed do not depend on the order in which edges are contracted.

\begin{lem}
\label{lem:contractionorder}
Let $\Tcal =(T,\Vcal,r)$ be a rooted tree-decomposition of a simple, locally finite, connected, rooted graph $(G,o)$, let $C = (P,X)$ be a configuration on $\Tcal$ let $F=\{f_1, f_2\} \subseteq E(T)$. Then $\Tcal/f_1/f_2 = \Tcal/f_2/f_1$ and $C/f_1/f_2 = C/f_2/f_1$.
\end{lem}

\begin{proof}
Before we get started, we need to discuss a notational issue. Recall that when we defined contractions, we abused notation so that we could refer to vertices and edges of $T$ and $T/f$ by the same names. When considering contractions of different edges, this is a potential source of confusion. For example, if there is an edge $e$ connecting $f_1^-$ to $f_2^-$, then $e$ refers to the edge connecting $t_{f_1}$ to $f_2^-$ in $T/f_1$, as well as to the edge connecting $t_{f_2}$ to $f_1^-$ in $T/f_2$. 

For double contractions as considered in this lemma, however, this abuse of notation works in our favour, that is, the function mapping every vertex of $T/f_1/f_2$ to the vertex of $T/f_2/f_1$ with the same name is an isomorphism (which will play the role of $\iota$). More precisely, there are two cases to consider: if $f_1$ and $f_2$ are not incident, then $T/f_1/f_2$ and $T/f_2/f_1$ both contain two contracted vertices denoted by $t_{f_1}$ and $t_{f_2}$. In this case, any edge incident to $f_1^-$ or $f_1^+$ in $T$ is incident to $t_{f_1}$ in $T/f_1/f_2$ and $T/f_2/f_1$ and  any edge incident to $f_2^-$ or $f_2^+$ in $T$ is incident to $t_{f_2}$ in $T/f_1/f_2$ and $T/f_2/f_1$. If $f_1$ and $f_2$ are incident, then $T/f_1/f_2$ and $T/f_2/f_1$ both contain a unique contracted vertex which we will denote by $t_F$. In this case, any edge incident to $f_1^-$, $f_1^+$, $f_2^-$, or $f_2^+$ in $T$  is incident to $t_F$ in $T/f_1/f_2$ and $T/f_2/f_1$. The endpoints of all other edges are the same in $T$, $T/f_1/f_2$, and $T/f_2/f_1$ thus giving the desired isomorphism. In light of the above discussion, we will from now on treat $T/f_1/f_2$ and $T/f_2/f_1$ as the same tree and denote it by $T/F$.

The claim $\Tcal/f_1/f_2 = \Tcal/f_2/f_1$ now follows directly from the definition of contraction. First note that by definition, if $t \notin \{f_1^-,f_1^+,f_2^-,f_2^+\}$, then  $\Vcal/f_1/f_2 (t) = \Vcal(t) = \Vcal/f_2/f_1 (t)$. If $f_1$ and $f_2$ are not incident, then 
\[
\Vcal/f_1/f_2 (t_{f_1}) = \Vcal/f_1(t_{f_1}) = \Vcal(f_1^-) \cup \Vcal(f_1^+) = \Vcal/f_2(f_1^-) \cup \Vcal/f_2(f_1^+) = \Vcal/f_2/f_1(t_{f_1}),
\]
and analogous arguments show that $\Vcal/f_1/f_2 (t_{f_2}) = \Vcal/f_2/f_1 (t_{f_2})$. In case $f_1$ and $f_2$ are incident, the same line of reasoning leads to 
\[
\Vcal/f_1/f_2 (t_F) = \Vcal/f_2/f_1 (t_F) = \Vcal(f_1^-) \cup \Vcal(f_1^+) \cup \Vcal(f_2^-) \cup \Vcal(f_2^+),
\]
where two of the sets in the union on the right-hand side are the same. It is also easily verified that $r/f_1/f_2 = r/f_2/f_1$, thus showing that indeed $\Tcal/f_1/f_2 = \Tcal/f_2/f_1$.

Our next goal is to show that $P/f_1/f_2 = P/f_2/f_1$. If $t \notin \{f_1^-,f_1^+,f_2^-,f_2^+\}$, then  $P/f_1/f_2 (t) = P(t) = P/f_2/f_1 (t)$ by definition, so it only remains to consider the contracted vertices. 

If $f_1$ and $f_2$ are not incident, then $P/f_1/f_2(t_{f_1}) = P/f_1(t_{f_1})$ and Lemma~\ref{lem:walkcontraction} tells us that this walk contains exactly the edges of $P(f_1^-) - \Ecal(f_1)$ and $P(f_1^+) - \Ecal(f_1)$. On the other hand, $P/f_2/f_1(t_{f_1})$ contains exactly the edges of $P/f_2(f_1^-) - \Ecal/f_2(f_1)$ and $P/f_2(f_1^+) - \Ecal/f_2(f_1)$. Since $P/f_2(f_1^-) = P(f_1^-)$, $P/f_2(f_1^+) = P(f_1^+)$, and $\Ecal/f_2(f_1) = \Ecal(f_1)$, the two edge sets coincide, and using that a self avoiding walk is uniquely determined by its set of edges we conclude that $P/f_1/f_2(t_{f_1}) = P/f_2/f_1(t_{f_1})$. An analogous argument shows that $P/f_1/f_2(t_{f_2}) = P/f_2/f_1(t_{f_2})$.

If $f_1$ and $f_2$ are incident, then we may assume without loss of generality that $f_1^- = f_2^-$, in particular the edge $f_2$ connects $f_2^+$ to $t_{f_1}$ in $T/f_1$.
By Lemma~\ref{lem:walkcontraction}, the edge set of $P/f_1/f_2(t_F)$ consists of the edges of $P/f_1(f_2^+) - \Ecal/f_1(f_2) = P(f_2^+) - \Ecal(f_2)$ and $P/f_1(t_{f_1}) - \Ecal/f_1(f_2)$. Again by Lemma~\ref{lem:walkcontraction}, the edge set of the latter multi-walk consists of the edge sets of $P(f_1^-) - (\Ecal(f_1) \cup \Ecal(f_2))$ and $P(f_1^+) - (\Ecal(f_1) \cup \Ecal(f_2))$. Since $f_1$ is not incident to $f_2^+$, the graph $\Gcal(f_2^+)$ and thus also the walk $P(f_2^+)$ is disjoint from $\Ecal(f_1)$, and we conclude that the edge set of the walk $P/f_1/f_2(t_F)$ consists of the edge sets of $P(f_1^-) - (\Ecal(f_1) \cup \Ecal(f_2))$, $P(f_1^+) - (\Ecal(f_1) \cup \Ecal(f_2))$, and $P(f_2^+) - (\Ecal(f_1) \cup \Ecal(f_2))$. Since $f_1^- = f_2^-$, this is symmetric in $f_1$ and $f_2$, and an analogous argument shows that the edge set of  $P/f_2/f_1(t_F)$ is the same. Thus the two walks coincide.

Finally, we need to show that $X/f_1/f_2 = X/f_2/f_1$. By Lemma~\ref{lem:contraction}, both $C/f_1/f_2$ and $C/f_2/f_1$ are bounded consistent configurations, thus by Remark \ref{rmk:sink} it suffices to show that the unique vertex $t \in T/F$ with $X/f_1/f_2(t) = t$ also satisfies $X/f_2/f_1(t) = t$. This clearly follows from the definition of $X/f$.
\end{proof}

Recall that the goal of this section is relating bounded consistent configurations on $\Tcal$ to self-avoiding walks of length at least $1$ starting at the root $o$ of $G$. In this sense the upcoming Theorem~\ref{thm:bijsawsconfigs} is the main result of this section. In preparation of this theorem, for each vertex $t$, we define a map $\psi_t$ mapping bounded consistent configurations $C$ on $\Tcal$ to SAWs on the graph $\Gcal({K_t})$ corresponding to the cone $K_t$ as follows.

First, recall the definition of $\Gcal(S)$ for $S$ subset of $V(T)$. In particular, when $S = K_s$ is a cone,
\[
    \Gcal(K_s) = \left(\bigcup_{t \in K_s} \Vcal(t)\;, \; \bigcup_{t \in K_s} \Ecal(t) \uplus \Ecal({s,s^{\ua}})\right).
\]
Let $S \subseteq K_t$ consist of the vertex $t$ and all vertices of $K_t$ carrying non-boring configurations. Note that Lemma~\ref{lem:boringextension} implies that $T[S]$ is connected and thus a finite subtree of $K_t$. Let $F=E(T[S])$ be the set of its edges. We define \phantomsection{}\label{def:psit}
\[
\psi_t(C)=P/F(t_F),
\]
where $t_F$ is the unique contracted vertex in $T/F$.
In other words, $\psi_t(C)$ is the self-avoiding walk on the finite graph $\Gcal/F(t_F)$ obtained by contracting all edges of $T[K_t]$ connecting two vertices carrying non-boring configurations. By \ref{itm:compatible-alternate} all its virtual edges must be in $\Ecal/F(t_F^\ua,t_F)=\Ecal(t^\ua,t)$, because all other neighbours of $t_F$ carry boring configurations. In particular, $\psi_t(C)$ is a SAW on $\Gcal(K_t)$ as claimed.

Let us illustrate this definition using the bounded consistent configuration depicted in Figure~\ref{fig:configurations}. In Figure~\ref{fig:contraction} we iteratively contract edges incident to the root vertex $r$ until only a single vertex carrying a non-boring configuration remains. This only takes two steps. Any further contraction, for example the one done in the third step, does not change the walk $\psi_r(C)$ anymore.

\input{contraction}

\begin{thm}\label{thm:bijsawsconfigs}
Let $(G,o)$ be a simple, locally finite, connected graph rooted at $o \in V(G)$, and let $\Tcal=(T, \Vcal,r)$ be a rooted tree decomposition of $(G,o)$. Then $\psi_r$ is a bijection between the set $\Ccal_{\Tcal}$ and the set of self-avoiding walks of length at least 1 on $G$ starting at $o$ and for every $C \in \Ccal_\Tcal$, the weight $\norm{C}$ coincides with the length of $\psi_r(C)$.
\end{thm}
\begin{proof}
Let $C \in \Ccal_{\Tcal}$. As above, let $S$ be the set of all vertices of $T$ carrying non-boring configurations and let $F=E(T[S])$. Note that the root $r$ is contained in $S$, and consequently $t_F$ is the root of $T/F$. By the above discussion, $\psi_r(C)=P/F(t_F)$ is a self-avoiding walk on $G$ starting at the vertex $o$. Furthermore, by inductive application of Lemma~\ref{lem:contraction}, the weight $\norm{C}$ is equal to the weight $\norm{C/F}$, which is the length of the walk $\psi_r(C)$, because $\psi_r(C)$ contains no virtual edges. Finally, note that $X/F(t_F)=t_F$ and thus $P/F(t_F)$ ends with a non-virtual edge, in particular $\psi_r(C)$ has length at least 1.

It remains to show that $\psi_r$ is bijective. We first show that it is injective. For $i=1,2$ let $C_i=(P_i,X_i) \in \Ccal_\Tcal$ such that $\psi_r(C_1)=\psi_r(C_2)$. Let $S$ consist of all vertices $s$ of $T$ such that at least one of $C_1(s)$ and $C_2(s)$ is non-boring and let $F=E(T[S])$. Then $X_1/F(t_F)=X_2/F(t_F)=t_F$ because all neighbours of $t_F$ carry boring configurations. While $S$ could potentially contain some vertices $s$ of $T$ such that $C_i(s)$ is boring, this does not influence the result of $P_i/F(t_F)$. Thus by assumption 
\[
P_1/F(t_F)=\psi_r(C_1)=\psi_r(C_2)=P_2/F(t_F)
\]
and this walk does not contain virtual edges, so Lemma~\ref{lem:boringextension} implies that $C_1/F=C_2/F$. Inductive application of Lemma~\ref{lem:contrbij} yields $C_1=C_2$, so $\psi_r$ is injective.

To prove that $\psi_r$ is surjective, let $p$ be a SAW of length at least $1$ on $G$ starting at $o$. There is a finite subset $S \subseteq V(T)$ such that all edges of $p$ are contained in $\Gcal(S)$ and $T[S]$ is connected. As before, let $F=E(T[S])$. Then $(p,t_F)$ is a configuration on $t_F$ and $p$ does not contain virtual edges, thus Lemma~\ref{lem:boringextension} provides a bounded consistent configuration $C \in \Ccal_{\Tcal/F}$ such that $C(t_F)=(p,t_F)$. Lemma~\ref{lem:contrbij} yields a configuration $C \in \Ccal_{\Tcal}$ with $\psi_r(C)=P$ and therefore $\psi_r$ is surjective.
\end{proof}

\begin{rmk}
While $\psi_t$ maps bounded consistent configurations on the cone $K_t$ to walks on the corresponding graph $\Gcal(K_t)$, it is (in general) only a bijection in the case $t=r$. This is due to the fact that for $t \neq r$ the exit direction $X(t)$ at vertex $t$ is in general not uniquely determined by the walk $\psi_t(C)$; in many cases it can be either $t$ or $t^\ua$.
\end{rmk}

\section{A grammar for bounded consistent configurations} \label{sec:configgrammar}

Throughout this section let $(G,o)$ be a rooted, simple, locally finite, connected graph having only thin ends and let $\Gamma$ be a group acting quasi-transitively on $G$.
Thomassen and Woess~\cite{MR1223698} showed that any transitive graph without thick ends is accessible. Using the fact that any quasi-transitive graph is quasi-isometric to some transitive graph, and additionally that quasi-isometries preserve accessibility, we obtain that the graph $G$ is accessible.

By Corollary~\ref{cor:graphdecomp} there is a reduced, $\Gamma$-invariant, rooted tree decomposition $\Tcal=(T,\Vcal,r)$ of $(G,o)$ such that there are only finitely many $\Gamma$-orbits on $E(T)$. By property~\ref{itm:finiteparts-thinends} in Section~\ref{sec:treedecomp} the parts of such a tree decomposition are finite and by \ref{itm:thinends-locfintree} the tree $T$ is locally finite. 

Our goal in this section is to reveal a recursive structure in the set of bounded consistent configurations, which we later use to define a context-free grammar. To this end we first show that there are only finitely many essentially different configurations on vertices of $T$. The letters in $\Si$ will correspond to these different configurations, and the production rules will reflect the ways in which individual configurations can be combined in a compatible way.

\subsection{Choosing representatives of configurations}

We would like to define a function $\rho$ which assigns one of finitely many representatives to each vertex $t$ of $T$ and each configuration $c = (p,x)$ on $t$. This function $\rho$ is chosen in a way that for neighbouring vertices $s$ and $t$ of $T$ and configurations $c_s$ and $c_t$ on them compatibility only depends on $\rho(c_s)$ and $\rho(c_t)$.

We start by choosing representatives of the cone equivalence classes of vertices of $T$.

\begin{lem}
There is a finite subset $R$ of $V(T)$ containing exactly one vertex of every cone type such that $T[R]$ is connected.
\end{lem}
\begin{proof}
Choose a subset $R$ of $V(T)$ such that $R$ contains at least one vertex of every cone type, $T[R]$ is connected, and $R$ has minimal cardinality among all such sets. Clearly $R$ is finite as there are only finitely many cone types. Assume that $R$ contains two vertices $s$ and $t$ in the same cone type and let $\gamma \in \Gamma$ map $K_{s}$ onto $K_{t}$. Then the set $R'=R \cup \gamma(R \cap K_s) \setminus (R \cap K_s)$ still satisfies the condition that $T[R']$ is connected and has smaller cardinality than $R$, because $t \in R \cap \gamma(R \cap K_s)$.
\end{proof}

Let us now fix a set of representatives $R$ of the cone types of $\Tcal$ as in the previous lemma. The representative of a vertex $t$ of $T$ is denoted by $\rho(t) \in R$.

For every vertex $t$ of $T$ we define an automorphism $\delta_t \in \Gamma$ mapping $K_t$ to $K_{\rho(t)}$. For $t \in R$ let $\delta_t=1_\Gamma$, the neutral element in $\Gamma$, which acts as the identity on $G$. For any vertex $t \in V(T) \setminus R$ with $t^\ua \in R$ fix some arbitrary automorphism $\delta_t \in \Gamma$ mapping $K_t$ to $K_{\rho(t)}$. Finally for all other vertices $t$ of $T$ inductively define
\begin{equation}\label{eq:deltadef}
\delta_t=\delta_{\delta_{t^\ua}(t)} \circ \delta_{t^\ua}.
\end{equation}
This is well defined because $\delta_{t^\ua}(t)$ must be a child of a vertex of $R$ for every $t$; in particular, if $\delta_{t^\ua}(t) \in R$, then $\delta_{\delta_{t^\ua}(t)} = 1_\Gamma$. Note that equation \eqref{eq:deltadef} in fact holds for all vertices $t$ of $T$ besides the root $r$.

We use these maps to extend the map $\rho$ to configurations. Let $c=(p,x)$ be a configuration on a vertex $t$ of $T$. Then $\rho$ maps $c$ onto a configuration on the representative $\rho(t)=\delta_t(t)$ of $t$ by 
\[
\rho(c)=\delta_t(c)=(\delta_t(p),\delta_t(x)).
\]

Additionally for every $t \in T$ let $t_1^\da, \dots, t_{k(t)}^\da$ be the $k(t)$ children of $t$ ordered in a way such that 
\begin{equation}\label{eq:deltacomp}
\delta_t(t_i^\da)=\delta_t(t)_i^\da=\rho(t)_i^\da.
\end{equation}
This can be achieved by fixing any order of the children of vertices in $R$ and then ordering the children of any vertex $t$ accordingly. From the definition \eqref{eq:deltadef} of $\delta_t$ it is immediate that 
\begin{equation}
    \label{eq:deltadefchild}
    \delta_{t_i^\da} = \delta_{\delta_t(t_i^\da)} \circ \delta_t.
\end{equation}

Note that the representative of the $i$-th child of $t$ is exactly the representative of the $i$-th child of $\rho(t)$:
\begin{equation}\label{eq:bijtrees1}
\rho(t_i^\da)=\delta_{t_i^\da}(t_i^\da)=\delta_{\delta_t(t_i^\da)}(\delta_t(t_i^\da))=\delta_{\rho(t)_i^\da}(\rho(t)_i^\da)=\rho(\rho(t)_i^\da).
\end{equation}
In the above equation, the first and last equalities use the definition of $\rho$ on $t_i^\da$ and $\rho(t)_i^\da$ respectively, the other equalities follow from \eqref{eq:deltadefchild} and \eqref{eq:deltacomp}.

Moreover we can use $\delta_t$ to map a consistent configuration $C$ on $t$ and its children to a consistent configuration $\delta_t \circ C \circ \delta_t^{-1}$ on the representative $\rho(t)$ and its children. Note that $\delta_t^{-1}$ is applied to vertices of $T$ whereas $\delta_t$ is applied to configurations on the corresponding parts. Similar to \eqref{eq:bijtrees1}, we get that the configuration assigned by $C$ to the $i$-th child of $t$ and the configuration assigned by $\delta_t \circ C \circ \delta_t^{-1}$ to the $i$-th child of the representative of $t$ have the same image under $\rho$:
\begin{equation}\label{eq:bijtrees2}
\rho(C(t_i^\da))=\delta_{t_i^\da}(C(t_i^\da))=\delta_{\delta_t(t_i^\da)}(\delta_t(C(t_i^\da)))=\rho(\delta_t(C(t_i^\da)))=\rho(\delta_t \circ C \circ \delta_t^{-1}(\rho(t)_i^\da)).
\end{equation}
Similarly to \eqref{eq:bijtrees1}, the first and second equalities in \eqref{eq:bijtrees2} follow from the definition of $\rho$ and equation \eqref{eq:deltadefchild}, respectively. For the third equality, we use the definition of $\rho$ and the fact that $\delta_t (C (t_i^\da))$ is a configuration on $\delta_t (t_i^\da)$. The last equality follows from \eqref{eq:deltacomp}.

Let us illustrate the above definitions using the tree decomposition $(T,\Vcal)$ shown in Figure~\ref{fig:treedecomp}. As before, we choose the central vertex corresponding to a part of cardinality 3 as the root $r$. It is not hard to see that $\Gamma$ acts with six orbits on the set of directed edges of $T$, so we obtain seven cone types and the set $R$ has to contain seven vertices (see also Lemma~\ref{lem:finiteconetypes}). Figure~\ref{fig:conetypes} shows a valid choice of the set $R$. Note that we could also have taken the seven vertices in the first three layers of the rooted tree $T$ as our set $R$. 

Let us sketch the recursive definition of $\delta_t$ using the vertices $t_1$, $t_2$ and $t_3$ in Figure~\ref{fig:conetypes}. First, note that the vertex $t_1$ is contained in $R$, so by definition $\delta_{t_1}$ is the identity map. The parent of $t_2$ is in $R$, so we may choose an arbitrary automorphism $\gamma$ in $\Gamma$ mapping $t_2$ to its representative $t_1$ in $R$ and set $\delta_{t_2}=\gamma$. Regarding $t_3$, note that the representative of $t_3^\ua$ in $R$ is $t_2^\ua$, thus by definition $\delta_{t_3^\ua}$ maps the cone $K_{t_3^\ua}$ onto $K_{t_2^\ua}$ and thus $t_3$ onto $t_2$. The map $\delta_{t_3}$ is obtained by first applying $\delta_{t_3^\ua}$ (which maps $t_3$ to $t_2$), and then applying $\delta_{\delta_{t_3^\ua}} = \delta_{t_2}$ (which maps $t_2$ to $t_1$).

\begin{figure}
	\centering
	\begin{tikzpicture}
	\begin{scope}[scale=0.5]
	\def\dista{5}
	\def\distb{6}
	\def\distc{2}
	\def\distd{3}
	\def\diste{1.2}
	\def\ha{-2}
	\def\hb{-2}
	\def\hc{-2}
	\def\diamr{0.18}
	\def\diamrep{0.3}
	\coordinate (A) at (0,0);
	\draw[fill=black] (A) circle (\diamr);
	\foreach \i [evaluate=\i as \ieva using {int(mod(\i+1,3))}] in {0,1,2} {
        \coordinate (B\i) at (\i*\dista-\dista,\ha);
        \coordinate (A\i) at (\i*\distb-\distb,2*\ha);
 		\draw (A)--(B\i);
 		\draw (B\i)--(A\i);
 		\draw[fill=gen\ieva] (B\i) circle (\diamr);
 		\draw[fill=black] (A\i) circle (\diamr);
 		\foreach \j [evaluate=\j as \jeva using {int(mod(\i+\j+2,3))}] in {0,1} {
 		    \coordinate(B\i\j) at (\i*\distb-\distb+\j*\distc-\distc/2,2*\ha+\hb);
 		    \coordinate(A\i\j) at (\i*\distb-\distb+\j*\distd-\distd/2,2*\ha+2*\hb);
 		    \draw (A\i)--(B\i\j);
 		    \draw (B\i\j)--(A\i\j);
 		    \draw[fill=gen\jeva] (B\i\j) circle (\diamr);
 		    \draw[fill=black] (A\i\j) circle (\diamr);
 		    \foreach \k [evaluate=\k as \keva using {int(mod(\i+\j+\k,3))}] in {0,1} {
     		    \coordinate(B\i\j\k) at (\i*\distb-\distb+\j*\distd-\distd/2+\k*\diste-\diste/2,2*\ha+2*\hb+\hc);
     		    \draw (A\i\j)--(B\i\j\k);
     		    \draw (B\i\j\k) -- +(0,-0.8);
     		    \draw[fill=gen\keva] (B\i\j\k) circle (\diamr);
     		}
 		}
	}
	\draw[fill=black] (A) +(-\diamrep,-\diamrep) rectangle +(\diamrep,\diamrep);
	\draw[fill=black] (A0) +(-\diamrep,-\diamrep) rectangle +(\diamrep,\diamrep);
	\draw[fill=black] (A2) +(-\diamrep,-\diamrep) rectangle +(\diamrep,\diamrep);
	\draw[fill=black] (A00) +(-\diamrep,-\diamrep) rectangle +(\diamrep,\diamrep);
	\draw[fill=gen1] (B0) +(-\diamrep,-\diamrep) rectangle +(\diamrep,\diamrep);
	\draw[fill=gen0] (B2) +(-\diamrep,-\diamrep) rectangle +(\diamrep,\diamrep);
	\draw[fill=gen2] (B00) +(-\diamrep,-\diamrep) rectangle +(\diamrep,\diamrep);
	\draw [line width=0.7pt] (B110)+(-0.7,-1) to[out=90,in=240] ($(B11)+(-0.3,0.7)$) to[out=-20,in=100] ($(B111)+(0.7,-1)$);
	\draw [line width=0.7pt] (B001)+(-0.7,-1) to[out=90,in=240] ($(B001)+(-0.3,0.7)$) to[out=0,in=100] ($(B001)+(0.7,-1)$);
	\draw [line width=0.7pt] (B000)+(-0.7,-1) to[out=90,in=180] ($(B0)+(0.5,0.8)$) to[out=-50,in=90] ($(B011)+(0.7,-1)$);
	\node at ($(A)+(0,1)$) {$r$};
	\node at ($(B11)+(0.3,1)$) {$t_3$};
	\node at ($(B001)+(0.3,1)$) {$t_2$};
	\node at ($(B0)+(0,1.3)$) {$t_1$};
	\end{scope}
  	\end{tikzpicture}
  	\caption{Cone types of the tree decomposition $T$. The square nodes give one possible valid choice for the set $R$. The marked cones rooted at $t_1$, $t_2$ and $t_3$ have the same type.}
  	\label{fig:conetypes}
\end{figure}

\subsection{Construction of the grammar}
\label{sec:constructgrammar}

The goal of this section is to construct a context-free language, whose words are in one-to-one correspondence with bounded consistent configurations on the rooted tree decomposition $\Tcal$. Essentially, this language is obtained by going through the vertices of the decomposition tree carrying non-boring configurations in the order given by depth-first search and noting down the corresponding representatives of configurations (that is, their images under $\rho$).

Instead of constructing a context free grammar for this language, we construct a 1-multiple context free grammar. 
Clearly the two notions are equivalent; recall that by Remark~\ref{rmk:cflanguages} every context-free grammar stands in one-to-one correspondence to a 1-multiple context-free grammar generating the same language.
However, in anticipation of Section~\ref{sec:mcflofsaws}, where a multiple context free grammar for SAWs is constructed in a very similar way, it makes more sense to consider multiple context-free grammars in the present section as well.

The 1-multiple context-free grammar $\Gra=(\Non,\Si,\Pro,A_r)$ generating our desired language is defined as follows.

The alphabet is
\[
\Si=\{a_c \mid c \text{ configuration on some } t \in R\}.
\]
The set of non-terminals is 
\[\Non=\{A_r\} \cup \{A_{t,c} \mid t \in R,\; c \text{ configuration on } t\}.
\] 
For each $t \in R$  and every consistent configuration $C$ on the set $\{t, t_1^\da, \dots, t_{k(t)}^\da\}$, the set $\Pro$ contains a production rule as follows. If $C(t)$ is non-boring, then the production
\[
A_{t,C(t)}(a_{C(t)} x_1 \dots x_{k(t)}) \leftarrow  \left( A_{\rho(t_{i}^\da),\rho(C(t_{i}^\da))}(x_i) \right)_{i \in [k(t)]}
\]
is in $\Pro$. If $C(t)$ is boring, then $\Pro$ contains the terminal rule
\[
A_{t,C(t)}(\epsilon) \leftarrow.
\]
Additionally $\Pro$ contains for every configuration $C(r)$ on $r$ the production
\[
A_r(x) \leftarrow A_{r,C(r)}(x).
\]

Note that any production rule of this grammar is uniquely determined by its head and tail. This means that we do not lose any information by using simplified derivation trees, where every vertex is labelled by the head of its corresponding rule instead of the complete rule. To shorten notation, we henceforth work with simplified labels.

\begin{rmk} \label{rmk:dtlabelswords}
Observe that for any simplified derivation tree $D$ of $\Gra$ whose root $d$ is labelled $A_{t,c}$, the following three conditions hold.
\begin{enumerate}
    \item[(i)] The number of children of $d$ is uniquely given by the pair $(t,c)$; we denote it by $k$.
    \item[(ii)] If $c$ is non-boring, the word corresponding to $D$ is $w(D)=a_c w(D_1) \dots w(D_k)$, where $D_i$ is the sub-tree of $D$ rooted at the $i$-th child of $d$ and $k>0$. Otherwise $c$ is boring, $w(D)=\epsilon$ and $d$ is the only vertex of $D$.
    \item[(iii)] Let $v_1(=d),v_2,\dots, v_n$ be the vertices of $D$ in DFS-order and let $\lambda(v_i)=A_{t_i,c_i}$ be their labels. Then 
    \[
    w(D)= x_1 \dots x_n, \quad \text{where} \quad x_i=
    \begin{cases}
    a_{c_i} \text{ if $c_i$ is non-boring}\\
    \epsilon \text{ otherwise}
    \end{cases}
    \]
\end{enumerate}
Observations (i) and (ii) are direct consequences of the structure of $\Gra$. For (iii) we use induction on the number of vertices $n$ of $D$. If $n=1$, then (ii) implies that $c_1$ is boring and $w(D)=\epsilon$, so (iii) holds. Let now $n>1$ and suppose (iii) holds for every derivation tree with at most $n-1$ vertices. Then $c_1$ is non-boring and the claim follows from (ii) by applying the induction hypothesis on the sub-trees $D_1, \dots, D_k$ rooted at the $k$ children of $v_1$.
\end{rmk}

\begin{lem}
The grammar $\Gra$ is unambiguous 1-multiple context-free.
\end{lem}
\begin{proof}
Let $D$ and $D'$ be different non-trivial (consisting of at least 2 vertices) derivation trees. Then there is a smallest positive number $m$, such that the $m$-th vertices $u$ and $u'$ in the DFS-orders on $D$ and $D'$ either have a different number of children or have a different label. Remark~\ref{rmk:dtlabelswords} (i) implies that in any case $\lambda(u)=A_{t,c} \neq \lambda(u')=A_{t',c'}$, so in particular $c \neq c'$. By minimality of $m$, the parents of $u$ and $u'$ have the same label, so by Lemma~\ref{lem:boringextension} the configurations $c$ and $c'$ must be non-boring. Thus Remark~\ref{rmk:dtlabelswords} (iii) implies that $w(D) \neq w(D')$ and $\Gra$ is unambiguous.
\end{proof}

Bounded consistent configurations on $\Tcal$ are closely related to derivation trees of $\Gra$. 
Let us define a map $\phi$ assigning to any given $C=(P,X) \in \Ccal_\Tcal$ a derivation tree of our grammar $\Gra$ as follows. Let $S \subseteq V(T)$ consist of all vertices $s$ carrying non-boring configurations $C(s)$ and their neighbours. Then $T[S]$ is connected and can be seen as an ordered tree with root $r$, where the order of the children is inherited from the tree $T$. We label every vertex $s$ of $T[S]$ with 
$\lambda(s)=A_{\rho(s),\rho(C(s))}$.

As an example, we provide in Figure~\ref{fig:derivationtree} the derivation tree $\phi(C)$ of the configuration $C$ given in Figure~\ref{fig:configurations}. Note that the vertex $t_2$ is not contained in the set $R$ from Figure~\ref{fig:conetypes}, so we need to apply $\rho$ to the respective configuration. For all other $t_i$, the map $\rho$ is the identity map and can be omitted. The vertex $t_5$ is not contained in $\phi(C)$ because its parent $t_2$ carries a boring configuration.

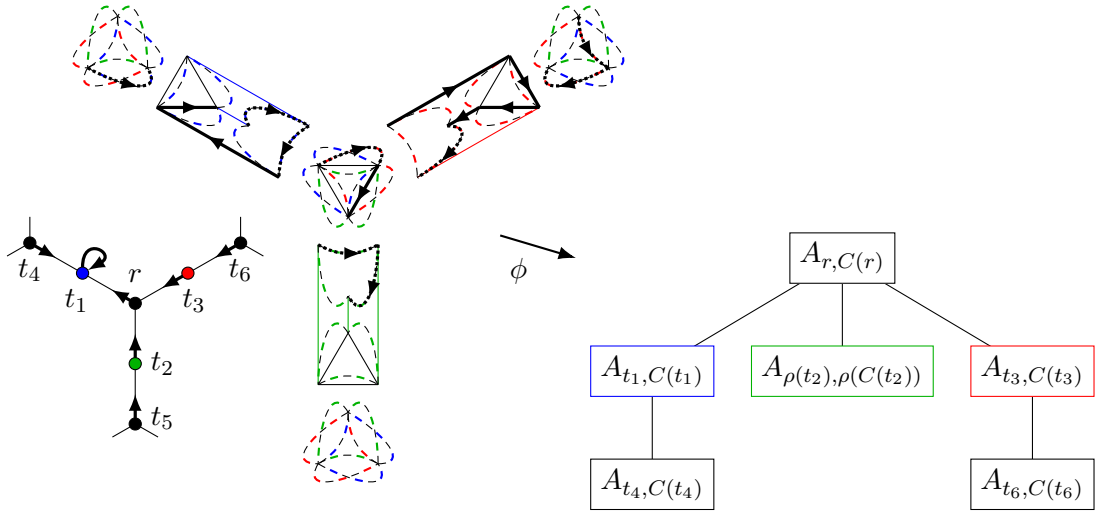
\begin{figure}[tb]
	\centering
	\begin{tikzpicture}
	\begin{scope}[scale=0.35]
	\def\distt{3}
	\def\dista{4}
	\def\diamr{1.3}
	\def\diama{1.3}
	\foreach \i [evaluate=\i as \ieva using {int(mod(\i+1,3))},evaluate=\i as \ievb using {int(mod(\i+2,3))}] in {0,1,2} {
		\draw (120*\i+30:\diamr)--(120*\i+150:\diamr);
		\coordinate (A\i) at (120*\i+30:\distt);
		\coordinate (B\i) at (120*\i+30:{\distt+\dista});
		\coordinate (C\i) at (120*\i+30:{2*\distt+\dista});
		\draw[gen\i] (B\i)+(120*\i+210:\diama) -- ($(A\i)+(120*\i+30:\diamr)$);
		\draw[gen\i] (B\i)+(120*\i+90:\diama) -- ($(A\i)+(120*\i+150:\diamr)$);
		\draw[gen\i] (B\i)+(120*\i+330:\diama) -- ($(A\i)+(120*\i+270:\diamr)$);
		\draw[virtual=gen\i]  (120*\i+30:\diamr) to[in={120*\i+30}, out={120*\i+60},looseness=1.5] (120*\i+150:\diamr);
		\draw[virtual=gen\i]  (120*\i+270:\diamr) to[in={120*\i}, out={120*\i+30},looseness=1.5] (120*\i+30:\diamr);
		\draw[virtual=gen\i]  (120*\i+150:\diamr) to[in={120*\i+90}, out={120*\i-30}] (120*\i+270:\diamr);
		\draw[virtual=gen\i]  (A\i)+(120*\i+30:\diamr) to[in={120*\i+30}, out={120*\i+60},looseness=1.5] +(120*\i+150:\diamr);
		\draw[virtual=gen\i]  (A\i)+(120*\i+270:\diamr) to[in={120*\i}, out={120*\i+30},looseness=1.5] +(120*\i+30:\diamr);
		\draw[virtual=gen\i]  (A\i)+(120*\i+150:\diamr) to[in={120*\i+90}, out={120*\i-30}] +(120*\i+270:\diamr);
		\draw[virtual=gen\i]  (B\i)+(120*\i+210:\diamr) to[in={120*\i+210}, out={120*\i+240},looseness=1.5] +(120*\i+330:\diamr);
		\draw[virtual=gen\i]  (B\i)+(120*\i+90:\diamr) to[in={120*\i+180}, out={120*\i+210},looseness=1.5] +(120*\i+210:\diamr);
		\draw[virtual=gen\i]  (B\i)+(120*\i+330:\diamr) to[in={120*\i+270}, out={120*\i+150}] +(120*\i+90:\diamr);
	}	
	\foreach \j in {0,1,2} {
	    \foreach \i in {0,1,2}{
		    \draw (B\i)+(120*\i+120*\j+210:\diama) -- +(120*\i+120*\j+90:\diama);
	    	\draw[virtual=gen\i]  (C\j)+(120*\i+210:\diamr) to[in={120*\i+210}, out={120*\i+240},looseness=1.5] +(120*\i+330:\diamr);
		    \draw[virtual=gen\i]  (C\j)+(120*\i+90:\diamr) to[in={120*\i+180}, out={120*\i+210},looseness=1.5] +(120*\i+210:\diamr);
		    \draw[virtual=gen\i]  (C\j)+(120*\i+330:\diamr) to[in={120*\i+270}, out={120*\i+150}] +(120*\i+90:\diamr);
		}
	}
    \begin{scope}[walkarrow] 
        \draw[postaction={decorate},densely dotted] (150:\diamr) to[out=30, in=60,looseness=1.5] (30:\diamr);
        \draw[postaction={decorate}] (30:\diamr) -- (270:\diamr);
        \draw[postaction={decorate}] (A0)+(150:\diama) -- ($(B0)+(90:\diama)$);
        \draw[postaction={decorate}] ($(B0)+(90:\diama)$) -- ($(B0)+(-30:\diama)$);
        \draw[postaction={decorate}] ($(B0)+(-30:\diama)$) -- ($(B0)+(210:\diama)$);
        \draw[postaction={decorate}] ($(B0)+(210:\diama)$) -- ($(A0)+(30:\diama)$);
        \draw[postaction={decorate},densely dotted] ($(A0)+(30:\diama)$) to[out=0, in=30,looseness=1.5] ($(A0)+(-90:\diama)$);
        \draw[postaction={decorate},densely dotted] ($(A1)+(150:\diama)$) to[out=120, in=150,looseness=1.5] ($(A1)+(30:\diama)$);
        \draw[postaction={decorate},densely dotted] ($(A1)+(30:\diama)$) to[out=210, in=90] ($(A1)+(-90:\diama)$);
        \draw[postaction={decorate}] ($(A1)+(-90:\diama)$) -- ($(B1)+(210:\diama)$);
        \draw[postaction={decorate}] ($(B1)+(210:\diama)$) -- ($(B1)+(-30:\diama)$);
        \draw[postaction={decorate},densely dotted] ($(A2)+(150:\diama)$) to[out=-30, in=210] ($(A2)+(30:\diama)$);
        \draw[postaction={decorate},densely dotted] ($(A2)+(30:\diama)$) to[out=270, in=-60,looseness=1.5] ($(A2)+(-90:\diama)$);
        \draw[postaction={decorate},densely dotted] ($(C0)+(90:\diama)$) to[out=-90, in=150] ($(C0)+(-30:\diama)$);
        \draw[postaction={decorate},densely dotted] ($(C0)+(-30:\diama)$) to[out=210, in=240,looseness=1.5] ($(C0)+(210:\diama)$);
        \draw[postaction={decorate},densely dotted] ($(C1)+(210:\diama)$) to[out=-30, in=-60,looseness=1.5] ($(C1)+(-30:\diama)$);
    \end{scope}
	\end{scope}
	\begin{scope}[scale=0.4, xshift=-7cm,yshift=-4cm]
	\def\dista{4}
	\def\diamr{0.2}
	\def\diama{0.2}
	\def\diamb{0.2}
	\def\diamc{0.2}
    \draw[fill=black] (0,0) circle (\diamr);
	\foreach \i in {0,1,2} {
		\coordinate (A\i) at (120*\i+30:\dista);
		\draw[fill=black] (A\i) circle (\diama);
		\draw (A\i)+(120*\i+210:\diama) --coordinate[midway](m) (120*\i+30:\diamr);
	    \draw[fill=gen\i] (m) circle (\diama);
		\draw (A\i)++(120*\i+90:\diama) -- +(120*\i+90:\dista/6);
	    \draw (A\i)++(120*\i+330:\diama) -- +(120*\i+330:\dista/6);
	}
	\node at (0,0.9) {$r$};
	\node at ($(30:{\dista/2})+(0.2,-0.9)$) {$t_3$};
	\node at ($(270:{\dista/2})+(0.9,0)$) {$t_2$};
	\node at ($(150:{\dista/2})+(-0.2,-0.9)$) {$t_1$};
	\node at ($(30:{\dista})+(0,-0.9)$) {$t_6$};
	\node at ($(270:{\dista})+(0.9,0.2)$) {$t_5$};
	\node at ($(150:{\dista})+(0,-0.9)$) {$t_4$};
	\draw[very thick,-latex] (0,0)++(150:\diama) -- +(150:0.7);
    \draw[very thick,-latex] (30:{\dista/2})++(210:\diama) -- +(210:0.8);
    \draw[very thick,-latex] (-90:{\dista/2})++(90:\diama) -- +(90:0.8);
    \draw[very thick,-latex] (30:{\dista})++(210:\diama) -- +(210:0.8);
    \draw[very thick,-latex] (150:{\dista})++(-30:\diama) -- +(-30:0.8);
    \draw[very thick,-latex] (-90:{\dista})++(90:\diama) -- +(90:0.8);
    \draw[very thick,-latex] (150:{\dista/2})+(90:\diama) to[out=90,in=30,looseness=14] +(30:\diama);
	\end{scope}
	\draw[-Latex, thick] (2,-0.7) -- node[below left] {$\phi$} (3,-1);
	\begin{scope}[scale=0.5,xshift=13cm,yshift=-2cm]
	\def\dista{5}
	\def\distb{5}
	\def\ha{-3}
	\def\hb{-3}
	\def\diamr{0.2}
	\coordinate (A) at (0,0);
	\draw[fill=black] (A) circle (\diamr);
	\foreach \i [evaluate=\i as \ieva using {int(mod(\i+1,3))}] in {0,1,2} {
        \coordinate (B\i) at (\i*\dista-\dista,\ha);
        \coordinate (A\i) at (\i*\distb-\distb,2*\ha);
 		\draw (A)--(B\i);
 		\draw[fill=gen\ieva] (B\i) circle (\diamr);
	}
 	\draw (B0)+(0,-\diamr)--(A0);
 	\draw (B2)+(0,-\diamr)--(A2);
 	\draw[fill=black] (A0) circle (\diamr);
 	\draw[fill=black] (A2) circle (\diamr);
 	\node[fill=white,draw=black] at (A) {$A_{r,C(r)}$}; 
 	\node[fill=white,draw=gen1] at (B0) {$A_{t_1,C(t_1)}$}; 
 	\node[fill=white,draw=gen2] at (B1) {$A_{\rho(t_2),\rho(C(t_2))}$}; 
 	\node[fill=white,draw=gen0] at (B2) {$A_{t_3,C(t_3)}$}; 
 	\node[fill=white,draw=black] at (A0) {$A_{t_4,C(t_4)}$};
 	\node[fill=white,draw=black] at (A2) {$A_{t_6,C(t_6)}$}; 
	\end{scope}
  	\end{tikzpicture}
  	\caption{The map $\phi$ transforms the configuration $C$ on the decomposition tree $T$ into the derivation tree $\phi(C)$ over the grammar $\Gra$.
      The generated word is $w(\phi(C))=a_{C(r)} a_{C(t_1)} a_{C(t_4)} a_{\rho(C(t_2))} a_{C(t_3)} a_{C(t_6)}$.
}
  	\label{fig:derivationtree}
\end{figure}

\begin{lem}\label{lem:bijconder}
The map $\phi$ is a bijection between the set $\Ccal_\Tcal$ of bounded consistent configurations on $\Tcal$ and the set of derivation trees whose root is labelled by $A_{r,c}$ for some configuration $c$ on $r$.
\end{lem}
\begin{proof}
Observe that an ordered tree labelled with non-terminals of $\Non$ is a simplified derivation tree of $\Gra$ if and only if for every vertex $t$ and its children $t_1, \dots, t_k$ there is a rule in $\Pro$ with head $\lambda(t)$ and tail $(\lambda(t_1), \dots, \lambda(t_k))$.

Let $S$ be as above and $s$ be a vertex of $T[S]$. If $C(s)$ is boring, then $s$ is a leaf in $T[S]$ and $A_{\rho(s),\rho(C(s))}(\epsilon) \leftarrow$  is a rule in $\Pro$. 

Otherwise $C(s)$ is non-boring and we consider the children $s_1^\da, \dots, s_{k(s)}^\da$ of $s$ in $T$, which are also the children of $s$ in $T[S]$. Then the production
\[
A_{\rho(s),\rho(C(s))}(a_{\rho(C(s))} x_1 \dots x_{k(s)}) \leftarrow  \left(A_{\rho(s_i^\da),\rho(C(s_i^\da))} (x_i)\right)_{i \in [k(s)]}
\]
is in $\Pro$ because $\delta_s \circ C \circ \delta_s^{-1}$ is a consistent configuration on $\{\delta_s(s), \delta_s(s_i^\da), \dots, \delta_s(s_{k(t)}^\da)\}$ and \eqref{eq:bijtrees1} and \eqref{eq:bijtrees2} hold.
We conclude that $\phi(C)$ is a derivation tree of $\Gra$.

Our next step is to show that $\phi$ is surjective. Let $D$ be a derivation tree of $\Gra$ whose root $d$ is labelled $A_{r,c}$. We recursively construct an embedding $\iota$ of $D$ into $T$ and a bounded consistent configuration $C$ on $T$ such that every vertex $u$ of $D$ has the label 
\begin{equation} \label{eq:labelderiv}
    \lambda(u)=A_{\rho(\iota(u)),\rho(C(\iota(u)))}.
\end{equation} 

We start the top-down construction by setting $\iota(d)=r$ and $C(r)=c$. Then clearly $d$ satisfies \eqref{eq:labelderiv}. Suppose now $\iota(u)=t$ is already defined for a vertex $u$ of $D$ and that $u$ satisfies \eqref{eq:labelderiv}. Let $u_1, \dots, u_k$ be the $k>0$ children of $u$ in $D$ and $\lambda(u_i)=A_{s_i,c_i}$ their labels. Then 
\begin{equation} \label{eq:phibijrule}
A_{\rho(t),\rho(C(t))}(a_{\rho(C(t))} x_1 \dots x_k) \leftarrow (A_{s_i,c_i}(x_i))_{i \in [k]}
\end{equation} 
is a rule in $\Pro$. This implies that $t$ has precisely $k$ children $t_1^\da, \dots, t_k^\da$ in $T$ and moreover that $\rho(t_i^\da)=s_i$ holds for every $i$. We define $\iota(u_i)=t_i^\da$ and $C(t_i^\da)=\delta_{t_i}^{-1}(c_i)$. Then $C(t_i^\da)$ is compatible with $C(t)$ for every $i \in [k]$; the production rule \eqref{eq:phibijrule} is in $\Pro$ if and only if $\rho(C(t))=\delta_t(C(t))$ and $\delta_{\rho(t)_i^\da}^{-1}(c_i)= \delta_t(\delta_{t_i}^{-1}(c_i))$ are compatible. In this way we have constructed a consistent configuration $C$ on the sub-tree $\iota(D)$ of $T$. Note that by definition of the set $\Pro$, configurations on leaves of $\iota(D)$ are boring. By Lemma~\ref{lem:boringextension} the configuration $C$ can be (uniquely) extended to a bounded consistent configuration on $T$. Moreover it follows directly from \eqref{eq:labelderiv} that $\phi(C)=D$, so $\phi$ is surjective.

Finally, it is not hard to see that $\phi$ is injective. Any two different configurations $C_1 \neq C_2$ have to differ on some vertex $t$ of $T$; we pick such a $t$ with minimal distance to the root $r$. Then $C_1(t^\ua)=C_2(t^\ua)$, so Lemma~\ref{lem:boringextension} yields that $C_1(t^\ua)$ is non-boring. Thus $t$ is a vertex of $\phi(C_1)$ and $\phi(C_2)$. For $i \in \{1,2\}$ the label of $t$ in $\phi(C_i)$ is $A_{\rho(t),\rho(C_i(t))}$, so in particular $\phi(C_1) \neq \phi(C_2)$ and we conclude that $\phi$ is injective.
\end{proof}

It is clear that $\phi$ also describes a bijection between the set $\Ccal_{\Tcal}$ of bounded consistent configurations on $\Tcal$ and derivation trees of $\Gra$ whose roots are labelled $A_r$. Moreover, the number of occurrences of a given letter $a_c$ in the word $w(\phi(C)) \in L(\Gra)$ corresponding to $\phi(C)$ is equal to the number of vertices $t$ of $T$ with $\rho(C(t))=c$.

Combining the previous results, the composition of the bijection $\phi$ mapping configurations onto derivation trees and the natural bijection $w$ between derivation trees and their corresponding words of an unambiguous 1-multiple context-free grammar $\Gra$ is a bijection between $\Ccal_{\Tcal}$ and words in $L(\Gra)$, as stated in the following corollary.

\begin{cor}\label{cor:bijconfigswords}
There is a bijection $\theta$ between the set $\Ccal_\Tcal$ of bounded consistent configurations on $\Tcal$ and the words in the unambiguous 1-multiple context-free language $L(\Gra)$ such that the number of occurrences of a given letter $a_c$ in $\theta(C)$ coincides with the number of vertices $t$ of $T$ with $\rho(C(t))=c$.
\end{cor}

Further combining Corollary~\ref{cor:bijconfigswords} with the connection between bounded consistent configurations and SAWs established in Theorem~\ref{thm:bijsawsconfigs}, we obtain the proof of our first main result, Theorem~\ref{thm:main-1}.  Before we turn to the proof, let us recall the statement of the theorem.

\begingroup
\def\thethm{\ref{thm:main-1}}
\begin{thm}
Let $G$ be a locally finite, connected, quasi-transitive graph having only thin ends and let $o \in V(G)$. Then $F_{\SAW,o}$ is algebraic over $\mathbb{Q}$. In particular the connective constant $\mu(G)$ is an algebraic number.
\end{thm}
\addtocounter{thm}{-1}
\endgroup

\begin{proof}
Let $\Gamma$ be a group acting quasi-transitively on $G$ and $\Tcal=(T,\Vcal,r)$ be a reduced, $\Gamma$-invariant, rooted tree decomposition of $(G,o)$ such that there are only finitely many $\Gamma$-orbits on $E(T)$. Then by Theorem~\ref{thm:bijsawsconfigs} the generating function of self-avoiding walks coincides with the generating function of the set $\Ccal_{\Tcal}$,
\[
F_{\SAW,o}(z)=\sum_{C \in \Ccal_\Tcal} z^{\norm{C}}.
\]
Let $\Gra=(\Non,\Si,\Pro,S)$ be the unambiguous 1-MCFG over the alphabet $\Si=\{a_{c_1}, \dots, a_{c_m}\}$ as defined at the start of Section~\ref{sec:constructgrammar}, where the $c_i$ are configurations on vertices in $R$, and let $F_{L(\Gra)}(a_{c_1}, \dots, a_{c_m})$ be the commutative language generating function of $L(\Gra)$. Chomsky and Schützenberger showed in \cite{MR0152391} that commutative language generating functions of unambiguous context-free languages are algebraic over $\mathbb{Q}$, that is, there is an irreducible polynomial $P$ in $m+1$ variables with coefficients in $\mathbb{Q}$ such that
\[
P(F_{L(\Gra)}(a_{c_1}, \dots, a_{c_m}),a_{c_1}, \dots, a_{c_m})=0.
\] 
Corollary~\ref{cor:bijconfigswords} yields that the generating function of $\Ccal_\Tcal$ coincides with the generating function obtained by substituting every variable $a_{c_i}$ in $F_{L(\Gra)}$ by the monomial $z^{\norm{c_i}}$,
\[
F_{\SAW,o}(z)=F_{L(\Gra)}\left(z^{\norm{c_1}}, \dots, z^{\norm{c_m}}\right)
\]
In particular $F_{\SAW,o}(z)$ is algebraic over $\mathbb{Q}$; it satisfies the equation 
\[
Q\left(F_{\SAW,o}(z),z\right)=0,
\]
where $Q(y,z)=P\left(y, z^{\norm{c_1}}, \dots, z^{\norm{c_m}}\right)$ is a polynomial with coefficients in $\mathbb{Q}$. The connective constant $\mu(G)$ is the reciprocal of the radius of convergences of the algebraic function $F_{\SAW,o}(z)$ and thus an algebraic number.
\end{proof}

\section{The multiple context-free language of self-avoiding walks}
\label{sec:mcflofsaws}

In this section, we prove the second main result of this paper, which we briefly recall for convenience.

\begingroup
\def\thethm{\ref{thm:main-2}}
\begin{thm}
Let $G$ be a simple, locally finite, connected, quasi-transitive deterministically edge-labelled graph and let $o \in V(G)$. Then $L_{\SAW,o}$ is $k$-multiple context-free if and only if every end of $G$ has size at most $2k$.
\end{thm}
\addtocounter{thm}{-1}
\endgroup

The proofs of the two implications are quite different from one another and will be discussed separately in the two subsections of this section. In order to show that bounded end size implies that $L_{\SAW,o}$ is a MCFL, we adopt a similar approach as in the previous section and construct a MCFG which is very closely related to the grammar $\Gra$ defined in Section~\ref{sec:configgrammar}. For the converse implication we essentially follow the approach of Lindorfer and Woess~\cite{Lindorfer2020}; we note that this final part does not depend on the other results of this paper and can be read independently.

\subsection{Bounded end size implies multiple context-freeness}

We use again the assumptions, notation and definition from the previous section. In particular $\Tcal = (T,\Vcal,r)$ again denotes a rooted tree decomposition of a rooted graph $(G,o)$ which will be fixed throughout this section; we also fix a map $\rho$ as constructed in Section~\ref{sec:configgrammar}. Our aim is proving the following theorem, a stronger version of one of the two implications of Theorem~\ref{thm:main-2}.

\begin{thm}\label{thm:sawlangmcf}
Let $G$ be a simple, locally finite, connected, quasi-transitive edge-labelled graph having only ends of size at most $k$ and let $o$ be a given vertex of $G$. Then the language of self-avoiding walks $L_{\SAW,o}$ is $\lceil k/2 \rceil$-multiple context-free.

If the edge-labelling is deterministic, then $L_{\SAW,o}$ is unambiguous $\lceil k/2 \rceil$-multiple con\-text-free.
\end{thm}

To prove the above theorem, we give an MCFG $\Gra=(\Non,\Si,\Pro,A_r)$ and show that it generates the language $L_{\SAW,o}$. As mentioned above, $\Gra$ is a refinement of the 1-MCFG from Section~\ref{sec:configgrammar}. 

Obviously, the alphabet $\Si$ has to consist of all edge-labels. Note that this is a finite set since $G$ is locally finite and the group of label preserving automorphisms is assumed to act quasi-transitively.

Like in Section~\ref{sec:configgrammar}, the set of non-terminals is 
\[\Non=\{A_r\} \cup \{A_{t,c} \mid t \in R,\; c \text{ configuration on $t$}\}.\]
However, since we are constructing a MCFG, we need to assign a rank to each non-terminal; to this end some additional definitions are necessary.
For a vertex $t$ of $T$ and a configuration $c=(p,x)$ on $t$ let $\mu(c)$ denote the number of walk components of $p-\Ecal({t^\ua,t})$ containing at least two vertices of $\Vcal({t^\ua,t})$. Furthermore let $r(c)=\mu(c)+1$ if $x \neq t^\ua$ and the final walk component of $p-\Ecal({t^\ua,t})$ contains only a single vertex of $\Vcal({t^\ua,t})$ and let $r(c)=\mu(c)$ otherwise. Define the rank of the non-terminal $A_{t,c}$ to be $r(c)$. Note that $r(c)=0$ if and only if $c$ is boring. For a configuration $c = (p,x)$ on the root part, we define $r(c) = 1$; note that this is consistent with the above definition in the sense that the root has no parent and there is exactly one walk component of $p$.

Next we turn to the set $\Pro$ of production rules. For every boring configuration $c$ on a vertex $t \in R$, $\Pro$ contains the rule
\[
A_{t,c}(\emptyset) \leftarrow.
\]


For non-boring configurations, we need more involved production rules that require some preliminary definitions. Let $t \in R$ and let $C=(P,X)$ be a consistent configuration on $\{t, t_1^\da, \dots, t_{k(t)}^\da\}$ such that $C(t)$ is non-boring. Let $P_{1}, \dots, P_{\mu(C(t))}$ be the walk components of $P(t)-\Ecal({t^\ua,t})$ containing at least 2 vertices of $\Vcal({t^\ua,t})$, and if $r(C(t))>\mu(C(t))$ let $P_{r(C(t))}$ be the (possibly trivial) final walk component of $P(t)-\Ecal({t^\ua,t})$. Moreover, for each $i \in [k(t)]$ let $U_i^{1}, \dots, U_i^{\mu_i}$ be the non-trivial walk components of $P(t) \cap \Gcal({t,t_i^\da})$, that is, the walk components that contain more than one vertex. Every $P_h$ admits a unique decomposition into an odd number of sub-walks $P_h= P_h^1 P_h^2 \dots P_h^{2m+1}$ such that $P_h^{l}$ is a (possibly trivial) non-virtual walk if $l$ is odd and equal to some $U_i^j$ if $l$ is even. Observe that $P_h^l = U_i^j$ means that the $l$-th part in the decomposition of $P_h$ is the $j$-th walk component of $P(t) \cap \Gcal(t,t_i^\da)$, that is, the notation $P_h^l$ indicates which walk component of $P(t)$ the walk lies in, whereas the notation $U_i^j$ tells us which adhesion set the virtual edges belong to, see Figure~\ref{fig:puwalks} for an example. 

\begin{figure}[tb]
	\centering
	\begin{tikzpicture}
	\def\rad{0.08}
	\begin{scope}[scale=0.8]
	    \coordinate (E1) at (-2,3);
    	\coordinate (E2) at (0,3);
    	\coordinate (E3) at (2,3);
    	\coordinate (E4) at (-5.3,-0.5);
    	\coordinate (E5) at (-4.65,-1.2);
    	\coordinate (E6) at (-4,-1.9);
    	\coordinate (E7) at (-3.35,-2.6);
    	\coordinate (E8) at (-1.5,-3.1);
    	\coordinate (E9) at (-0.5,-3.1);
    	\coordinate (E10) at (0.5,-3.1);
    	\coordinate (E11) at (1.5,-3.1);
    	\coordinate (E12) at (7.35/2,-4.5/2);
    	\coordinate (E13) at (8.65/2,-3.1/2);
    	\coordinate (E14) at (9.95/2,-1.7/2);
	\foreach \i in {1,...,14}{
	    \fill (E\i) circle (\rad);
	}
    	\draw[thick] (E3) to[out=-90,in=150] (E14) (E12) to[out=100,in=60,looseness=1.1] (E5) (E4) to[out=70,in=270] (E1) (E2) to[out=-70,in=90] (E11) (E10) to[out=80,in=70,looseness=1.7] (E6) (E7) to[out=60,in=120] (E8) (E9) to[out=90,in=-10] (-2,-1.5);
    	\draw[thick, dash pattern= on 3pt off 2pt] (E14) -- (E13) -- (E12) (E5) -- (E4) (E1) -- (E2) (E11) -- (E10) (E6) -- (E7) (E8) -- (E9);
    	\node at (3.4,0.8) {$P_1^1$};
    	\node at (4.05,-1.2) {$P_1^2$};
    	\node at (-0.5,1.2) {$P_1^3$};
    	\node at (-4.7,-0.6) {$P_1^4$};
    	\node at (-4,1.3) {$P_1^5$};
    	\node at (1.7,-1) {$P_2^1$};
    	\node at (1,-2.75) {$P_2^2$};
    	\node at (-0.5,-0.2) {$P_2^3$};
    	\node at (-3.4,-2) {$P_2^4$};
    	\node at (-2.2,-2.1) {$P_2^5$};
    	\node at (-1.1,-2.75) {$P_2^6$};
    	\node at (-1.5,-1.3) {$P_2^7$};
    	\node at (-5.3,-1.1) {$U_1^1$};
    	\node at (-4,-2.5) {$U_1^2$};
    	\node at (-1,-3.5) {$U_2^2$};
    	\node at (1,-3.5) {$U_2^1$};
    	\node at (4.75,-1.75) {$U_3^1$};
        \node at (0,4.5) {$\Gcal(t^\ua,t)$};
        \node at (-5.7,-2.5) {$\Gcal(t,t_1^\da)$};
        \node at (0,-4.5) {$\Gcal(t,t_2^\da)$};
        \node at (5.7,-2.5) {$\Gcal(t,t_3^\da)$};
        \draw[clip] (0,0) ellipse (6cm and 4cm);
    	\draw (0,6) ellipse (6cm and 4cm);
    	\draw[rotate=-35,xshift=-3cm,yshift=-0.7cm] (-4,-8) to[out=75,in=180] (0,-2) to[out=0,in=105] (4,-8);
        \draw (-4,-8) to[out=75,in=180] (0,-2) to[out=0,in=105] (4,-8);
        \draw[rotate=+35,xshift=+3cm,yshift=-0.7cm] (-4,-8) to[out=75,in=180] (0,-2) to[out=0,in=105] (4,-8);
	\end{scope}

  	\end{tikzpicture}
  	\caption[test]{Decomposition of $P(t)$ into subwalks $P_h^l$. For even $l$, the walk $P_h^l$ is equal to some $U_i^j$; for instance $P_2^4=U_1^2$ means that the fourth subwalk in the second walk component of $P(t) - \Ecal(t^\ua,t)$ coincides with the second walk component of $P(t) \cap \Gcal(t,t_1^\protect\da)$.}
  	\label{fig:puwalks}
\end{figure}
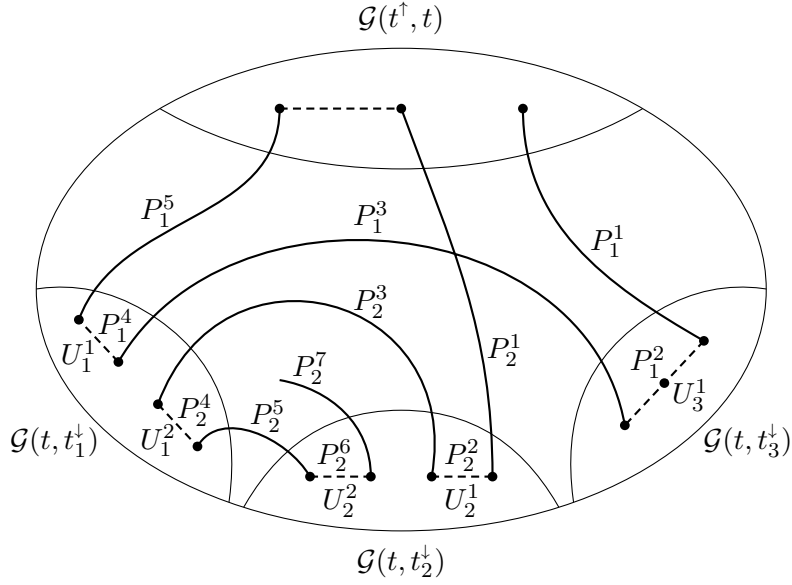

For $h <  r(C(t))$, the string $\alpha_h$ corresponding to $P_h$ is 
\[
\alpha_h=\ell(P_h^1) z(P_h^2) \ell(P_h^3) z(P_h^4) \dots \ell(P_h^{2m+1}),
\]
where $z(P_h^l)=z_{i,j}$ with $i,j$ chosen such that $P_h^l=U_i^j$; in other words, $\alpha_h$ is obtained from $P_h$ by concatenating labels of non-virtual walks $P_h^l$, and variables $z_{i,j}$ for virtual walks $P_h^l$. For $h = r(C(t))$, we define
\[
\alpha_h=\ell(P_h^1) z(P_h^2) \ell(P_h^3) z(P_h^4) \dots \ell(P_h^{2m+1})\beta,
\]
where $z(P_h^l)=z_{i,j}$ as above, and $\beta = \epsilon$ unless there is some $i \in [k(t)]$ such that $X(t)=t_i^\da$ and $P(t)$ does not end with an edge in $\Ecal({t,t_i^\da})$, in which case $\beta$ is equal to the single variable $z_{i,\mu_i+1}$.

For every $t \in R$ and each configuration $C$ on $t$ and its children such that $C(t)$ is non-boring, $\Pro$ includes the production rule
\begin{equation}\label{eq:mcfprod}
A_{t,C(t)}(\alpha_1, \dots, \alpha_{r(C(t))}) \leftarrow \left(A_{\rho(t_i^\da),\rho(C(t_i^\da))}(z_{i,1}, \dots, z_{i,r(C(t_i^\da))})\right)_{i\in [k(t)]}
\end{equation}
where the strings $\alpha_1, \dots, \alpha_{r(C(t))}$ are defined as above.

Finally, for every configuration $c$ on the root $r$ the production rule
\[
A_r(z) \leftarrow A_{r,c}(z)
\]
is in $\Pro$; this is a well formed production rule since $A_{r,c}$ has rank 1.

Before showing that this grammar indeed generates the language $L_{\SAW,o}$, let us discuss why this intuitively should be true. 

Let us extend the label function $\ell$ to walks $p$ on the graph $\Gcal({K_t})$ corresponding to the cone $K_t$ in the natural way by mapping $p$ onto the tuple of labels $\ell(p)= (\ell(p_1), \dots, \ell(p_m))$ of the non-trivial walk components $p_1, \dots, p_m$ of $p-\Ecal({t^\ua,t})$. If $p$ is the empty walk, then $\ell(p)=\emptyset$. 
For a bounded consistent configuration $C \in \Ccal_\Tcal$ with $C(t) = c$, we want the term  $A_{t,c}(\ell(\psi_t(C)))$ to be derivable in $\Gra$, where $\psi_t(C)$ is the walk on $\Gcal(K_t)$ corresponding to the configuration $C$ as defined on page \pageref{def:psit}. This is inductively taken care of by production rules of type \eqref{eq:mcfprod}: sub-walks of $P_h$ consisting of virtual edges correspond to variables $z_{i,j}$ in the string $\alpha_h$ which are subsequently replaced by strings corresponding to walk components of $\psi_s(C)-\Ecal({t,s})$ where $s$ is some child of $t$. To see that this intuitively makes sense, recall that by definition $\psi_t(C)$ is obtained by contracting all edges vertices carrying non-boring configurations and then taking the walk $P(t_F)$, where $t_F$ is the unique contracted vertex. By Lemma~\ref{lem:contractionorder}, the order of edge contractions does not matter, hence we can first contract all edges not incident to $t$; in particular, the walk $\psi_t(C)$ can be obtained by replacing sub-walks consisting of virtual edges in $\Ecal(t,s)$ by appropriate sub-walks of $\psi_s(C)$, see Figure~\ref{fig:mcfgcontraction}. This replacement procedure is essentially captured by rules of type \eqref{eq:mcfprod}.

\begin{figure}[tb]
	\centering
	\begin{tikzpicture}
	\def\rad{0.15}
	\begin{scope}[scale=0.35, xscale=0.9]
    	\begin{scope}[xshift=-0.5cm, yshift=-2.5cm] 
        	\draw[rotate=-30,xshift=-3cm,yshift=-1cm,clip] (-3.5,-8) to[out=75,in=180] (0,-2) to[out=0,in=105] (3.5,-8) (-3.5,-8);
        	\draw[thick] (-5,-1.3) to[out=240,in=90] (-8.8,-4.8) (-8,-5.5) to[out=60,in=210] (-4.2,-1.9) to[out=270,in=150] (-5,-4) to[out=-30,in=240] (-3.4,-2.5);
        	\draw (0,0) ellipse (6cm and 4cm);
            \fill (-5,-1.3) circle (\rad);
        	\fill (-4.2,-1.9) circle (\rad);
        	\fill (-3.4,-2.5) circle (\rad);
    	\end{scope}
    	\begin{scope}[yshift=-3cm] 
        	\draw[clip] (-3.5,-8) to[out=75,in=180] (0,-2) to[out=0,in=105] (3.5,-8) (-3.5,-8);
        	\draw[thick] (-1.5,-3.3) to[out=-90,in=150] (-1.5,-5.5) to[out=-30,in=260] (-0.5,-3.3) (0.5,-3.3) to[out=-90,in=80] (0.5,-5) to[out=260,in=60] (-2,-7) to[out=240,in=150] (1.5,-7) to[out=-30,in=260,] (1.5,-3.3);
        	\draw[thick, dash pattern= on 3pt off 2pt] (-0.5,-3.3) -- (0.5,-3.3);
        	\draw (0,0) ellipse (6cm and 4cm);
            \fill (-1.5,-3.3) circle (\rad);
    	    \fill (-0.5,-3.3) circle (\rad);
    	    \fill (0.5,-3.3) circle (\rad);
    	    \fill (1.5,-3.3) circle (\rad);
    	\end{scope}
    	\begin{scope}[xshift=0.5cm, yshift=-2.5cm] 
        	\draw[rotate=+30,xshift=+3cm,yshift=-1cm,clip] (-3.5,-8) to[out=75,in=180] (0,-2) to[out=0,in=105] (3.5,-8) (-3.5,-8);
        	\draw[thick] (5,-1.3) to[out=280,in=90] (7,-4) to[out=-90,in=-90] (4.2,-1.9) (3.4,-2.5) to[out=280,in=90] (6,-7);
        	\draw[thick, dash pattern= on 3pt off 2pt] (4.2,-1.9) -- (3.4,-2.5);
        	\draw (0,0) ellipse (6cm and 4cm);
            \fill (5,-1.3) circle (\rad);
        	\fill (4.2,-1.9) circle (\rad);
        	\fill (3.4,-2.5) circle (\rad);
    	\end{scope}
    	\draw[thick] (-2,3) to[out=-90,in=60] (-5,-1.3) (-3.4,-2.5) to[out=60,in=180] (-2,0) to[out=0,in=90] (-1.5,-3.3) (-0.5,-3.3) to[out=80,in=-90] (0,3) (2,3) to[out=-90,in=90] (0.5,-3.3) (1.5,-3.3) to[out=70,in=240] (3,1.5) to[out=60,in=100] (5,-1.3)  (4.2,-1.9) to[out=90,in=30] (3.5,-0.5) to[out=210,in=100] (3.4,-2.5);
    	\draw[thick, dash pattern= on 3pt off 2pt] (-5,-1.3) -- (-4.2,-1.9) -- (-3.4,-2.5) (-1.5,-3.3) -- (-0.5,-3.3) (0.5,-3.3) -- (1.5,-3.3) (5,-1.3) -- (4.2,-1.9);
        \draw[thick, dash pattern= on 3pt off 2pt] (0,3) -- (2,3);
    	\fill (-2,3) circle (\rad);
    	\fill (-5,-1.3) circle (\rad);
    	\fill (-4.2,-1.9) circle (\rad);
    	\fill (-3.4,-2.5) circle (\rad);
    	\fill (-1.5,-3.3) circle (\rad);
    	\fill (-0.5,-3.3) circle (\rad);
    	\fill (0,3) circle (\rad);
    	\fill (2,3) circle (\rad);
    	\fill (0.5,-3.3) circle (\rad);
    	\fill (1.5,-3.3) circle (\rad);
    	\fill (5,-1.3) circle (\rad);
    	\fill (4.2,-1.9) circle (\rad);
    	\fill (3.4,-2.5) circle (\rad);
    	\draw[-Latex, thick] (10.5,-1) -- (13.5,-1);
        \draw[clip] (0,0) ellipse (6cm and 4cm);
    	\draw (0,6) ellipse (6cm and 4cm);
    	\draw[rotate=-30,xshift=-3cm,yshift=-1cm] (-3.5,-8) to[out=75,in=180] (0,-2) to[out=0,in=105] (3.5,-8);
        \draw (-3.5,-8) to[out=75,in=180] (0,-2) to[out=0,in=105] (3.5,-8);
        \draw[rotate=+30,xshift=+3cm,yshift=-1cm] (-3.5,-8) to[out=75,in=180] (0,-2) to[out=0,in=105] (3.5,-8);
	\end{scope}
	\begin{scope}[scale=0.35,xshift=22cm, xscale=0.9]
	    \path[name path=ell] (0,0) ellipse (6cm and 4cm);
	    \begin{scope}
    	    \draw[clip] (0,0) ellipse (6cm and 4cm);
        	\draw (0,6) ellipse (6cm and 4cm);
	    \end{scope}
    	\path[name path=conea, rotate=-30,xshift=-3cm,yshift=-1cm,fill=white] (-3.5,-8) to[out=75,in=180] (0,-2) to[out=0,in=105] (3.5,-8) (-3.5,-8);
    	\path[name path=coneb, fill=white] (-3.5,-8) to[out=75,in=180] (0,-2) to[out=0,in=105] (3.5,-8) (-3.5,-8);
    	\path[name path=conec, rotate=+30,xshift=+3cm,yshift=-1cm,fill=white] (-3.5,-8) to[out=75,in=180] (0,-2) to[out=0,in=105] (3.5,-8) (-3.5,-8);
        \draw[intersection segments={of=conea and ell,sequence={L1 L3}}];
        \draw[intersection segments={of=coneb and ell,sequence={L1 L3}}];
        \draw[intersection segments={of=conec and ell,sequence={L1 L3}}];
    	\begin{scope}
    	    \draw[thick] (-5,-1.3) to[out=240,in=90] (-8.8,-4.8) (-8,-5.5) to[out=60,in=210] (-4.2,-1.9) to[out=270,in=150] (-5,-4) to[out=-30,in=240] (-3.4,-2.5);
    	    \draw[thick] (-1.5,-3.3) to[out=-90,in=150] (-1.5,-5.5) to[out=-30,in=260] (-0.5,-3.3) (0.5,-3.3) to[out=-90,in=80] (0.5,-5) to[out=260,in=60] (-2,-7) to[out=240,in=150] (1.5,-7) to[out=-30,in=260,] (1.5,-3.3);
        	\draw[thick] (5,-1.3) to[out=280,in=90] (7,-4) to[out=-90,in=-90] (4.2,-1.9) (3.4,-2.5) to[out=280,in=90] (6,-7);
    	\end{scope}
    	\draw[thick] (-2,3) to[out=-90,in=60] (-5,-1.3) (-3.4,-2.5) to[out=60,in=180] (-2,0) to[out=0,in=90] (-1.5,-3.3) (-0.5,-3.3) to[out=80,in=-90] (0,3) (2,3) to[out=-90,in=90] (0.5,-3.3) (1.5,-3.3) to[out=70,in=240] (3,1.5) to[out=60,in=100] (5,-1.3)  (4.2,-1.9) to[out=90,in=30] (3.5,-0.5) to[out=210,in=100] (3.4,-2.5);
    	\draw[thick, dash pattern= on 3pt off 2pt] (0,3) -- (2,3);
    	\fill (-2,3) circle (\rad);
    	\fill (-5,-1.3) circle (\rad);
    	\fill (-4.2,-1.9) circle (\rad);
    	\fill (-3.4,-2.5) circle (\rad);
    	\fill (-1.5,-3.3) circle (\rad);
    	\fill (-0.5,-3.3) circle (\rad);
    	\fill (0,3) circle (\rad);
    	\fill (2,3) circle (\rad);
    	\fill (0.5,-3.3) circle (\rad);
    	\fill (1.5,-3.3) circle (\rad);
    	\fill (5,-1.3) circle (\rad);
    	\fill (4.2,-1.9) circle (\rad);
    	\fill (3.4,-2.5) circle (\rad);
	\end{scope}
  	\end{tikzpicture}
  	\caption{Iterated contraction shows that we can obtain $\psi_t(C)$ by combining walks $P(t)$ and $\psi_s(C)$ for all children $s$ of $t$.}
  	\label{fig:mcfgcontraction}
\end{figure}
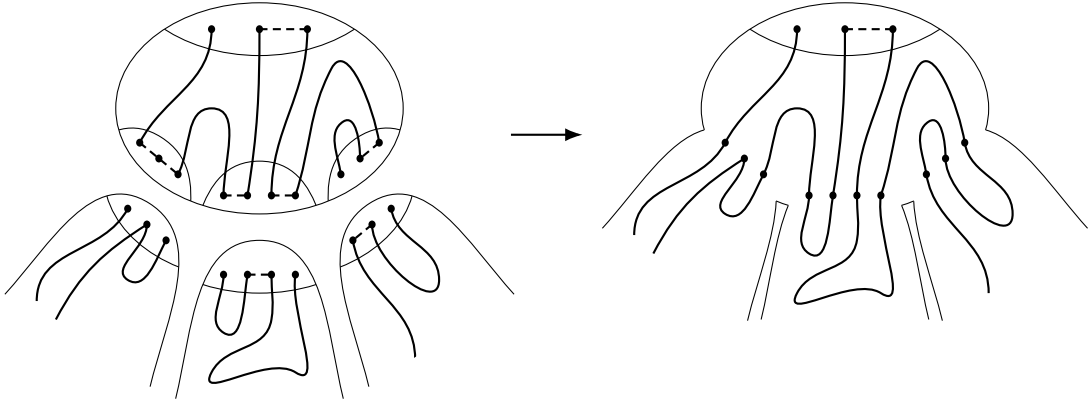

When making the above intuition precise, there are some technical issues that need to be addressed, leading to the fairly involved definition of $r(c)$ and to the subtle difference between $\alpha_h$ for $h < r(C(t))$ and $\alpha_{r(C(t))}$. These are due to the fact that walk components of $\psi_t(C) - \Ecal(t^\ua,t)$ come in two different flavours. Recall that the walk $\psi_t(C)$ starts in $\Vcal(t^\ua,t)$, and so does every walk component of $\psi_t(C)-\Ecal(t^\ua,t)$. Call a walk component a \emph{U-walk} if it returns to $\Vcal(t^\ua,t)$, and an \emph{I-walk} if it doesn't. We now sketch how these two types of walks relate to the definitions of $r(c)$ and $\alpha_h$.

Firstly, recall that the rank $r(C(t))$ of $A_{t,C(t)}$ should equal the number of walk components of $\psi_t(C) - \Ecal(t^\ua,t)$.  It is not hard to see that the projection of any U-walk is a non-trivial walk component of $P(t) - \Ecal(t^\ua,t)$ containing at least two vertices in $\Vcal(t^\ua,t)$. I-walks on the other hand may or may not use edges of $\Gcal(t)$. Note however that if there is an I-walk, then it is necessarily the last walk component of $\psi_t(C)-\Ecal(t^\ua,t)$, and it is not hard to see that in this case $X(t) \neq t^\ua$. In particular, we have $r(C(t)) = \mu(C(t)) + 1$ if and only if there is an I-walk. Since $\mu(C(t))$ clearly counts the number of U-walks, we conclude that $r(C(t))$ is indeed the number of walk components.

I-walks are also the reason why we need to include $\beta$ in the definition of $\alpha_{r(C(t))}$. If $t$ has a child $t_i^\da$ such that $\psi_{t_i^\da}(C)-\Ecal(t,t_i^\da)$ contains an I-walk, then $\psi_t(C)$ ends in this I-walk. Note that this happens if and only if $X(t) = t_i^\da$ and the last edge of $P(t)$ is not contained in $\Ecal(t,t_i^\da)$. We would like the  production rules to reflect this possibility, but the I-walk only intersects $\Vcal(t)$ in its starting point and thus it does not correspond to any non-trivial walk component of $P(t) \cap \Gcal(t,t_i^\da)$. Adding $z_{i,\mu_i+1}$ to the end of the string $\alpha_{r(C(t))}$ allows us to replace the trivial walk consisting of the last vertex of $P(t)$ by such an I-walk. We once again point out that an I-walk always sits at the end of $\psi_t(C)$, so we only have to consider this in the definition of $\alpha_h$ for $h = r(C(t))$.

With the above intuition and the resulting subtleties in mind, let us start by proving some basic results about the grammar $\Gra$.

\begin{lem}\label{lem:propgra}
The grammar $\Gra$ is $\lceil k/2 \rceil$-multiple context-free.
\end{lem}
\begin{proof}
As mentioned in the previous section, $R$ is a finite set and the number of valid configurations on a given part is finite, so $\Non$ and $\Pro$ are finite sets. For the proof of multiple context-freeness of $\Gra$, it only remains to verify that every expression \eqref{eq:mcfprod} is a well-formed production rule over $(\Non, \Si)$. Compatibility of $C(t)$ and $C(t_i^\da)$ implies that $\mu_i = \mu(C(t_i^\da)))$ for every $i$. Additionally, if $X(t) =t_i^\da$ for some $i \in [k(t)]$ and $P(t)$ does not end with an edge in $\Ecal(t,t_i^\da)$, then $X(t_i^\da) \neq t$ and the final walk component of $P(t_i^\da)-\Ecal({t,t_i^\da})$ contains only a single vertex of $\Vcal({t,t_i^\da})$, so $\mu_i+1 = \mu(C(t_i^\da))+1= r(C(t_i^\da))$. We conclude that every variable $z_{i,j}$ with $j\leq r(C(t_i^\da))$ occurs in $\alpha_1 \dots \alpha_{r(C(t))}$ and it follows directly from the construction that none of them occurs more than once. As a consequence, $\Gra$ is multiple-context-free.

Let $P_1, \dots, P_{r(C(t))}$ be the walk components of $P(t) - \Ecal(t^\ua,t)$. Then $P_i$ contains at least two vertices of $\Vcal(t^\ua,t)$ for $i<r(C(t))$ and at least one vertex of $\Vcal(t^\ua,t)$ for $i=r(C(t))$.

The size of $\Vcal({t^\ua,t})$ is at most $k$, so $2r(c)-1 \leq k$ holds, which for an integer $k$ is equivalent to $r(c) \leq \lceil k/2 \rceil$. We conclude that $\Gra$ is $\lceil k/2 \rceil$-multiple-context-free.
\end{proof}

While the grammar $\Gra$ may appear more complicated than the 1-multiple context-free grammar of configurations introduced in Section~\ref{sec:configgrammar}, the two grammars share many structural similarities. In particular, production rules are again uniquely determined by their heads and tails, so we can again work with simplified derivation trees, where every vertex is labelled with the head of its production rule. In fact we will even show that simplified derivation trees of the two grammars are the same.

To this end, let us again define a map $\phi$ mapping bounded consistent configurations $C$ on $T$ to simplified derivation trees of $\Gra$. Let the set $S$ consist of all vertices $s \in V(T)$ carrying non-boring configurations $C(s)$ and the neighbours of such vertices. Then $T[S]$ is an ordered tree with root $r$, where the order on the children of a vertex $s$ is $s_1^\da, \dots, s_{k(t)}^\da$. By labelling every vertex $s$ of $T[S]$ with $A_{\rho(s), \rho(C(s))}$, we obtain an ordered tree labelled with elements of $\Non$.

The following lemma is analogous to Lemma~\ref{lem:bijconder}, the proof is exactly the same and is thus omitted.
\begin{lem}\label{lem:bijconder2}
The map $\phi$ is a bijection between the set $\Ccal_{\Tcal}$ of bounded consistent configurations on $\Tcal$ and the set of derivation trees whose root is labelled by $A_{r,c}$ for some configuration $c$ on $r$.
\end{lem}

It remains to show that  for any configuration $C \in \Ccal_\Tcal$, the word $\ell(\psi_r(C))$ given by the SAW $\psi_r(C)$ on $G$ coincides with the word corresponding to the derivation tree $\phi(C)$.

\begin{lem}\label{lem:pathlabelequalsword}
Let $C \in \Ccal_\Tcal$ be a bounded consistent configuration on $\Tcal$. Then
\[
\ell(\psi_r(C))=w(\phi(C)).
\]
\end{lem}
\begin{proof}
Let $C \in \Ccal_\Tcal$. During the proof we denote for $t \in V(T)$ by $\phi_t(C)$ the cone of $\phi(C)$ rooted at $t$. We prove that whenever $C(t^\ua)$ is non-boring it holds that
\begin{equation}\label{eq:pathlabelequalsword}
\ell(\psi_t(C))=w(\phi_t(C)).
\end{equation}
We proceed by induction on the number of vertices $s \in K_t$ carrying non-boring configurations $C(s)$. 

Let $t \in V(T)$ be such that $C(s)$ is boring for every $s \in K_t$. Then $\psi_t(C)$ is the empty walk, so $\ell(\psi_t(C))=\emptyset$. Furthermore $\phi_t(C)$ consists only of the vertex $t$ labelled $A_{\rho(t),\rho(C(t))}$ and $\rho(C(t))$ is boring, so $w(\phi_t(C))=\emptyset$.

For the induction step we first set up some notation. Let $r$ be the number of walk components of $\psi_t(C)-\Ecal(t^\ua,t)$. For $h \in [r]$, let $Q_h$ denote the $h$-th walk component of $\psi_t(C)-\Ecal(t^\ua,t)$, and let $w_h = \ell(Q_h)$. 
By definition $\ell(\psi_t(C))=(w_1,\dots,w_r)$.
Analogously, for each child $t_i^\da$ of $t$, we define $r_i$ as the number of walk components of $\psi_{t_i^\da}(C)-\Ecal(t,t_i^\da)$. For $j \in [r_i]$, let $Q_i^j$ be the $j$-th walk component of $\psi_{t_i^\da}(C) - \Ecal(t, t_i^\da)$ and let $w_i^j$ be the label of $Q_i^j$. The definition of $\ell$ together with the induction hypothesis imply that
\[
w(\phi_{t_i^\da}(C)) = \ell(\psi_{t_i^\da}(C))= (w_i^1,\dots,w_i^{r_i}),
\]
in particular $r_i = r(C(t_i^\da))$. We examine the left-hand side and right-hand side of \eqref{eq:pathlabelequalsword} independently, and show that they yield the same tuple of words.

For the left-hand side first recall that for any vertex $s$ of $T$ we have that $\psi_{s}(C) = P/F(t_F)$ where $F$ is the set of edges in $K_{s}$ incident to vertices with non-boring configurations. By Lemma~\ref{lem:contractionorder} the order of edge contractions does not play a role, in particular $\psi_t(C)$ can be obtained by first contracting all such edges not incident to $t$, and then contracting the edges connecting $t$ to $t_i^\ua$ one by one.

This means that we can construct $\psi_t(C)$ from $P(t)$ by performing the following modifications for each $i \in [k(t)]$. For every virtual edge $uv \in P(t) \cap\Ecal(t,t_i^\da)$, the walk $\psi_{t_i^\da}(C)$ contains a sub-walk with the same endpoints consisting entirely of non-virtual edges. Replace every such virtual edge in $P(t)$ by the corresponding walk in $\psi_{t_i^\da}(C)$. If $X(t) = t_i^\da$, then append the sub-walk of $\psi_{t_i^\da}(C)$ after the last vertex in $\Vcal(t,t_i^\da)$ to the resulting walk. Note that equivalently, we can let $U_i^1, \dots, U_i^{\mu_i}$ be the sequence of non-trivial walk components of $P(t) \cap \Gcal(t,t_i^\da)$, replace every $U_i^j$ by the respective $Q_i^j$, and append $Q_i^{r_i}$ in case $X(t) = t_i^\da$ and $P(t)$ does not end in a virtual edge in $\Ecal(t,t_i^\da)$.

Let $P_1, \dots, P_{r(C(t))}$ be the walk components of $P(t)-\Ecal({t^\ua,t})$ and for every $h \in [r(C(t))]$ let $P_h=P_h^1 P_h^2 P_h^2 \dots P_h^{2m+1}$ be the unique decomposition into sub-walks such that $P_h^{l}$ is a possibly trivial non-virtual walk if $l$ is odd and equal to some $U_i^j$ if $l$ is even. 
By the above discussion, for $h < r(C(t))$ we thus have 
\[
\ell(Q_h)=\ell(P_h^1) \tilde{\ell}(P_h^2) \ell(P_h^3) \tilde{\ell}(P_h^4) \dots \ell(P_h^{2m+1}),
\]
where $\tilde{\ell}(P_h^l)=\ell(Q_i^j) = w_i^j$ for the unique indices $i,j$ satisfying $P_h^l=U_i^j$. For the final walk component, that is, $h = r(C(t))$, we analogously obtain 
\[
\ell(Q_h)=\ell(P_h^1) \tilde{\ell}(P_h^2) \ell(P_h^3) \tilde{\ell}(P_h^4) \dots \ell(P_h^{2m+1}) \beta,
\]
where $\tilde{\ell}(P_h^l)= w_i^j$ as above and $\beta = \epsilon$, unless $X(t)=t_i^\da$ and $P(t)$ does not end with an edge in $\Ecal(t,t_i^\da)$, in which case $\beta = \ell(Q_i^{\mu_i+1}) = w_i^{\mu_i+1}$.

We now turn to the right-hand side of \eqref{eq:pathlabelequalsword}. For every $i \in [k(t)]$, the induction hypothesis implies that $\phi_{t_i^\da}(C)$ is a derivation tree of the term 
\[
\tau_i:= A_{\rho(t_i^\da),\rho(C(t_i^\da))}\left(w_i^1,\dots,w_i^{r_i}\right).
\]
Moreover the root $t$ of $\phi_t(C)$ has label $A_{\rho(t),\rho(C(t))}$, so $\phi_t(C)$ is a derivation tree of the term obtained by application of the rule 
\[
A_{\rho(t),\rho(C(t))}(\alpha_1, \dots, \alpha_{r(C(t))}) \leftarrow \left(A_{\rho(t_i^\da),\rho(C(t_i^\da))}(z_{i,1}, \dots, z_{i,r(C(t_i^\da))})\right)_{i\in [r(C(t))]}
\]
to $(\tau_i)_{i \in [k(t)]}$. By definition of $\alpha_h$, the $h$-th entry of this term is obtained from the $w_i^j$ in the exact same way as $\ell(Q_h)$ and we conclude that $\ell(\psi_t(C))=w(\phi_t(C))$ as claimed.
\end{proof}

We are now able to prove the main result of this section by combining the previous results.
\begin{proof}[Proof of Theorem~\ref{thm:sawlangmcf}] 
Theorem~\ref{thm:bijsawsconfigs} yields that the language of self-avoiding walks of the graph 
$G$ satisfies
\[
L_{\SAW,o}= \{\ell(\psi_r(C)) : C \in \Ccal_\Tcal\}.
\]
Furthermore Lemma~\ref{lem:bijconder2} implies that the language generated by the grammar $\Gra$ is given by
\[
L_{\Gra}=\{w(\phi(C)) : C \in \Ccal_\Tcal\}.
\]
But by Lemma~\ref{lem:pathlabelequalsword} these two sets are equal and $\Gra$ is a $\lceil k/2 \rceil$-multiple-context-free grammar generating $L_{\SAW,o}$.

Finally, if the edge-labelling of $G$ is deterministic, then $\ell$ is a bijection between the set of self-avoiding walks on $G$ and $L_{\SAW,o}$. Lemma~\ref{lem:pathlabelequalsword} provides equality of the maps $w \circ \phi= \ell \circ \psi_r$, so in particular $w \circ \phi$ is also a bijection. We conclude that $w$ bijectively maps derivation trees with respect to $\Gra$ onto words in $L(\Gra)$, so $\Gra$ is unambiguous. 
\end{proof}

\subsection{Multiple context-freeness implies bounded end size} \label{sec:endsmcfls}

In this section we prove the second part of our main result, namely
\begin{thm}\label{thm:sawlangnotmcf}
Let $G$ be a simple, locally finite, connected, quasi-transitive deterministically edge-labelled graph such that $L_{\SAW,o}$ is $k$-multiple context-free for some $o \in V(G)$. Then every end of $G$ has size at most $2k$.
\end{thm}
As mentioned before, the proof of this statement will mostly follow the approach of Lindorfer and Woess~\cite{Lindorfer2020}.

Recall that any graph automorphism is either elliptic, parabolic or hyperbolic, depending on whether it fixes a finite subset of vertices, a unique end or a unique pair of ends. In what follows, elliptic automorphisms are useless, so as a first step we will eliminate the possibility that all label-preserving automorphisms are elliptic.

We remark that there are numerous examples of infinite graphs admitting a transitive group action by only elliptic automorphisms. To see this, note that any non-elliptic automorphism must have infinite order because it cannot fix a finite set of vertices. Therefore some interesting examples arise from the study of the famous Burnside Problem from 1902, asking whether every finitely generated torsion group, that is a group in which every element has finite order, must be finite. While this question remained unsolved for more than 60 years, nowadays various examples of infinite torsion groups are known. Any such group acts transitively on its Cayley graph by only elliptic automorphisms.

However, if $L_{\SAW,o}(G)$ is multiple context-free, then there are always non-elliptic automorphisms. The following lemma extends \cite[Lemma 4.3]{Lindorfer2020} to multiple context-free languages; the proof is essentially the same.

\begin{lem} \label{lem:existnonell}
Let $G$ be a connected, infinite, locally finite and deterministically edge-labelled graph and let $\Gamma \leq \AUT(G,\ell)$ act quasi-transitively on $G$. If $L_{\SAW,o}$ is multiple context-free for some vertex $o$ of $G$, then $\Gamma$ contains a non-elliptic element.
\end{lem}
\begin{proof}
The graph $G$ is infinite and connected, so $L_{\SAW,o}$ is an infinite language. Thus by Lemma~\ref{lem:pumplem} the $k$-multiple context-free language $L_{\SAW,o}$ contains some word $w=x_1y_1x_2 \dots y_{2k}x_{2k+1}$ such that at least one of $y_1, \dots, y_{2k}$ is not the empty word $\epsilon$, and $x_1 y_1^n x_2 \dots y_{2k}^n x_{2k+1} \in L_{\SAW,o}$ for every $n \in \N_0$. Let $m=\min\{i \in [2k] : y_i \neq \epsilon\}$. Then for every $n \in \N_0$, the word $x_1 \dots x_m y_m^n$ is a prefix of some word in $L_{\SAW,o}$ and thus itself contained in $L_{\SAW,o}$. Let $v_0$ be the end-vertex of the unique walk $p$ on $G$ starting at $o$ and having label $\ell(p)=x_1 \dots x_m$. Then for every $n\geq 0$ there is a unique self-avoiding walk $p_n$ of length $n \abs{y_m}$ starting at $v_0$ and having label $y_m^n$. We denote by $v_n$ the endpoint of the walk $p_n$. Using the fact that $\Gamma$ acts quasi-transitively on $G$, there must be some $\tau \in \Gamma$ and some $0 \leq i < j \leq n$ such that $\tau v_i = v_j$. Since $\tau$ is label preserving, $\tau^l v_i= v_{j+(l-1)(j-i)} \neq v_i$ for every $l>0$ and~\cite[Proposition 12]{MR335368} yields that $\tau$ is non-elliptic. 
\end{proof}

A locally finite, connected graph $S$ is called a \emph{strip} if it is quasi-transitive and has precisely two ends. Note that both ends of a strip have the same finite size $k$ which we call the size of $S$. It is known that any strip $S$ of size $k$ has an automorphism $\tau$ such that the cyclic group $\langle \tau \rangle$ generated by $\tau$ acts quasi-transitively on $S$. By \cite[Theorem 9]{MR335368}, there is $n \in \N$ such that $S$ contains $k$ disjoint $\tau^n$-invariant double rays. Let us call $S$ a \emph{$\tau$-strip} for $\tau \in \AUT(S)$, if the last statement holds with $n=1$. We use the same notation if $S$ is a sub-graph of a graph $G$ invariant under $\tau \in \AUT(G)$.

The following lemma is a combination of~\cite[Lemmas 3.3 and 3.4]{Lindorfer2020} and provides the existence of $\tau$-strips in certain types of graphs. In particular, it implies that  $\tau$-strips exist whenever there is a non-elliptic automorphism.
\begin{lem} \label{lem:existtaustrips}
Let $G$ be a locally finite connected graph and let $\Gamma$ act quasi-transitively on $G$.
\begin{enumerate}
    \item If $G$ has a thin end of size $k$, then it contains a $\tau$-strip of size $k$ for some $\tau \in \Gamma$.
    \item If $\Gamma$ contains a parabolic element, then for every $k \geq 1$, the graph $G$ contains a $\tau$-strip of size $k$ for some $\tau \in \Gamma$.
\end{enumerate}
\end{lem}

With these existence results for $\tau$-strips in a graph $G$ in mind, we turn to the relation between strips in a graph $G$ and its language of SAWs. A combination of the previous lemma and the upcoming lemma is already sufficient to treat graphs $G$ without thick ends.

\begin{lem} \label{lem:snailwalks}
Let $G$ be a connected, infinite, deterministically edge-labelled graph on which $\Gamma=\AUT(G,\ell)$ acts quasi-transitively, let $o$ be a vertex of $G$ and let $k \in \N$. If $G$ contains a $\tau$-strip of size $2k+1$ for some $\tau \in \Gamma$, then $L_{\SAW,o}$ is not $k$-multiple context-free. 
\end{lem}
\begin{proof}
The proof can be outlined as follows. We start by defining an infinite set $\Pcal$ of walks such that firstly, the language $\ell(\Pcal)$ is regular, and secondly, the language $\ell(\Pcal_{\SAW})$ of the subset $\Pcal_{\SAW}$ consisting of all self-avoiding walks in $\Pcal$ is not $k$-multiple context-free. 
It then follows from the closure properties of $k$-multiple context free languages that $L_{\SAW,o}$ is not $k$-multiple context-free.

The set $\Pcal$ essentially consists of spiral-shaped walks on the strip $S$, see Figure~\ref{fig:snailwalks}. For a concise definition, first recall that the strip $S$ contains $2k+1$ $\tau$-invariant rays $R_1, \dots, R_{2k+1}$ by the definition of a $\tau$-strip, and that the subgroup $\langle \tau \rangle$ of $\Gamma$ generated by $\tau$ acts quasi-transitively on $S$. Let $K$ be a set of orbit representatives of the action of $\langle \tau \rangle$ on $S$ such that the induced sub-graphs $S[K]$ and $R_i[K]$ (for every $i \in [2k+1]$) are connected. 

Let $T_K$ be a spanning tree of $S[K]$ containing all edges of the paths $R_i[K]$. Such a tree exists because $S[K]$ is connected and the rays $R_i$ are disjoint and acyclic. For $j \in [2k]$ let $T_K(j)$ be the smallest sub-tree of $T_K$ containing the paths $R_j[K], \dots, R_{2k+1}[K]$. We call a ray $R_i$ pendant in $T_K(j)$ if $T_K(j)$ contains precisely one edge connecting a vertex of $R_i$ to a vertex not in $R_i$. Clearly $T_K(j)$ contains at least two pendant rays and we may relabel the rays in a way such that for every $i \in [2k]$, the rays $R_i$ and $R_{i+1}$ are pendant in the tree $T_K(i)$.

For $i,j \in [2k+1]$ with $i \neq j$ let $W_{i,j}$ be the path connecting $R_i$ to $R_j$ in $T_K$. Furthermore let $W_{0,2}$ consist of a shortest walk connecting $o$ to some $v_0 \in \tau^n(T_K)$, followed by a walk connecting $v_0$ and $R_2$ on $\tau^n(T_K)$.  

It will be convenient to slightly abuse notation and define concatenations of walks whose endpoint and starting point do not coincide, but for which the endpoint of the first walk can be mapped onto onto the starting point by a suitable power of $\tau$. More precisely, let $P$ and $Q$ be two walks on $S$, let $u$ be the endpoint of $P$ and let $v$ be the starting point of $Q$. If $v=\tau^i(u)$, we write $PQ$ for the concatenation of $P$ and $\tau^i Q$. 

Using this notation, for each $i \in [2k]$ let us define a walk  
\[
X_i=W_{i-1,i+1} Q W_{i+1,i},
\]
where $Q$ is the path connecting the endpoint of $W_{i-1,i+1}$ to the starting point of $\tau(W_{i+1,i})$ on the ray $R_{i+1}$ if $i$ is odd and the endpoint of $\tau(W_{i-1,i+1})$ to the starting point of $W_{i+1,i}$ if $i$ is even. Note that we apply the notation defined above, so $X_i$ consists of the paths $W_{i-1,i+1}$, $Q$ and $\tau(W_{i+1,i})$. Moreover, let $X_{2k+1}=W_{2k,2k+1} Q$, where $Q$ connects the endpoint of $W_{2k,2k+1}$ with its image under $\tau$ on the ray $R_{2k+1}$. Note that $X_i$ is self-avoiding because $W_{i-1,i+1}$ and $W_{i+1,i}$ are fully contained in $S[K]$. Furthermore observe that $X_i$ is contained in $T_K(i-1)$ and meets $R_{i-1}$ and $R_i$ only in its endpoints.

Next, for every $i \in [2k+1]$ let $r_i$ be the terminal vertex of $X_i$. Note that $r_i$ is a vertex of $R_i$. Moreover for $i \leq 2k$ it lies in the same orbit as the initial vertex of $X_{i+1}$ because the rays $R_i$ and $R_{i+1}$ are pendant in the tree $T_K(i)$.
Finally let $Y_i$ be the sub-path of $R_i$ connecting $r_i$ with $\tau^2(r_i)$ if $i$ is odd and let $Y_i$ be the sub-path of $R_i$ connecting $\tau^2(r_i)$ with $r_i$ if $i$ is even.

Let $\Pcal$ be the infinite set of walks of the form
\begin{equation} \label{eq:spiralwalk}
X_1 Y_1^{n_1} X_2 Y_2^{n_2} \dots X_{2k+1} Y_{2k+1}^{n_{2k+1}},
\end{equation}
where $n_1, \dots, n_{2k+1} \in \N$. See Figure~\ref{fig:snailwalks} for an illustration of an element of $\Pcal$.

The language $L(\Pcal)$ has the form
\begin{equation}\label{eq:spiralwalkwords}
L(\Pcal)=\{x_1 y_1^{n_1} x_2 y_2^{n_2} \dots x_{2k+1} y_{2k+1}^{n_{2k+1}}: n_1, \dots, n_{2k+1} \in \N\}
\end{equation}
where the words $x_i$ and $y_i$ are the labels of the walks $X_i$ and $Y_i$, respectively. Clearly, $L(\Pcal)$ is a regular language. 

We claim that a walk of type \eqref{eq:spiralwalk} is self-avoiding if and only if $n_{i+1} < n_i$ for every $i \in [2k]$. Fix a walk $W=X_1 Y_1^{n_1} X_2 Y_2^{n_2} \dots X_{2k+1} Y_{2k+1}^{n_{2k+1}} \in \Pcal$ and denote by $\tilde{X}_i$ the sub-walk of $W$ corresponding to $X_i$ and by $\tilde{Y}_i$ the sub-walk of $W$ corresponding to the concatenated walk $Y_i^{n_i}$. In the example shown in Figure~\ref{fig:snailwalks} we have $n_2=n_1-1$. Observe that $n_2 \geq n_1$ would yield a self-intersection on $R_2$.

We say that a vertex $v \in S$ lies on level $l \in \Z$, if $v \in \tau^{2l}(K) \cup \tau^{2l+1}(K)$.  
First note that the walk $\tilde{X}_1$ does not contain vertices on level $l \geq 1$. Moreover, it follows inductively that $\tilde{X}_i$ contains only vertices on level $l_i=\sum_{j=1}^{i-1} (-1)^{j-1} n_j$ and that $\tilde{Y}_i$ starts on level $l_i$ and ends on level $l_{i+1}$.

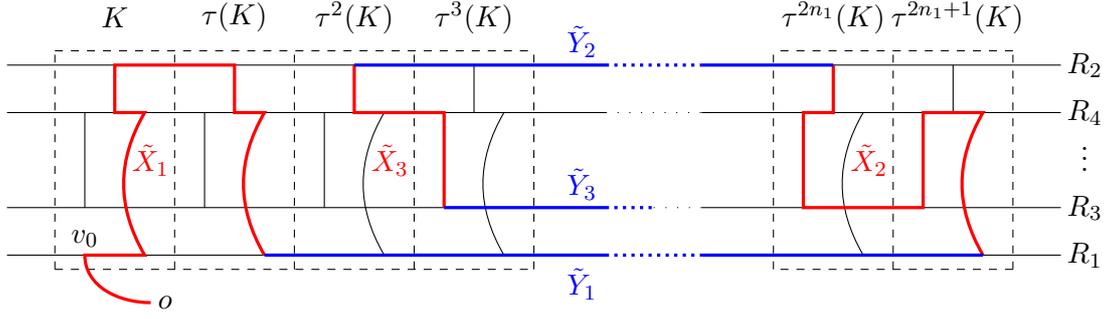
\begin{figure}
	\centering
	\begin{tikzpicture}
	\def\length{20}
	\def\width{4}
	\def\kw{2.5}
	\begin{scope}[scale=0.63]
	    \foreach \i in {0,1,\width-1,\width} {
            \draw (-1,\i) -- (4*\kw+1.5,\i);
            \draw (\length-2*\kw-1.5,\i) -- (\length+1,\i);
        }
        \node at (2.3,-1) {$o$};
        \node at (0.25*\kw,0.35) {$v_0$};
        \node[gen0] at (0.8*\kw,\width/2) {$\tilde X_1$};
        \node[gen0] at (\length-1.2*\kw,\width/2) {$\tilde X_2$};
        \node[gen0] at (2.8*\kw,\width/2) {$\tilde X_3$};
        \node[gen1] at (4*\kw+1,-0.6) {$\tilde Y_1$};
        \node[gen1] at (4*\kw+1,\width+0.5) {$\tilde Y_2$};
        \node[gen1] at (4*\kw+1,1.5) {$\tilde Y_3$};
        \node at (\length+1.5,0) {$R_1$};
        \node at (\length+1.5,\width) {$R_2$};
        \node at (\length+1.5,1) {$R_3$};
        \node at (\length+1.5,\width-1) {$R_4$};
        \node at (\length+1.5,\width/2+0.2) {$\vdots$};
        \node at (0.5*\kw,\width+1) {$K$};
        \node at (1.5*\kw,\width+1) {$\tau(K)$};
        \node at (2.5*\kw,\width+1) {$\tau^2(K)$};
        \node at (3.5*\kw,\width+1) {$\tau^3(K)$};
        \node at (\length-1.5*\kw,\width+1) {$\tau^{2n_1}(K)$};
        \node at (\length-0.5*\kw,\width+1) {$\;\tau^{2n_1+1}(K)$};
        \foreach \i in {0,...,3}{
            \draw[dashed] (\i*\kw,-0.3) -- (\i*\kw,\width+0.3);
            \draw (\i*\kw+0.5*\kw,\width) -- (\i*\kw+0.5*\kw,\width-1);
            \draw (\i*\kw+0.25*\kw,\width-1) -- (\i*\kw+0.25*\kw,1);
            \draw (\i*\kw+0.75*\kw,\width-1) to[bend right] (\i*\kw+0.75*\kw,0);
        }
        \draw[dashed] (0,-0.3)--(\kw*4,-0.3);
        \draw[dashed] (0,\width+0.3)--(\kw*4,\width+0.3);
        \draw[dashed] (\kw*4,-0.3)--(\kw*4,\width+0.3);
        \draw[dashed] (\length,-0.3)--(\length-2*\kw,-0.3);
        \draw[dashed] (\length,\width+0.3)--(\length-2*\kw,\width+0.3);
        \draw[dashed] (\length,-0.3)--(\length,\width+0.3);
        \draw[dashed] (\length-\kw,-0.3)--(\length-\kw,\width+0.3);
        \draw[dashed] (\length-2*\kw,-0.3)--(\length-2*\kw,\width+0.3);
        \draw (\length-0.5*\kw,\width) -- (\length-0.5*\kw,\width-1);
        \draw (\length-0.75*\kw,\width-1) -- (\length-0.75*\kw,1);
        \draw (\length-0.25*\kw,\width-1) to[bend right] (\length-0.25*\kw,0);
        \draw (\length-1.5*\kw,\width) -- (\length-1.5*\kw,\width-1);
        \draw (\length-1.75*\kw,\width-1) -- (\length-1.75*\kw,1);
        \draw (\length-1.25*\kw,\width-1) to[bend right] (\length-1.25*\kw,0);
        \draw[gen0,very thick] (2,-1) to[out=180,in=-90] (0.25*\kw,0) -- (0.75*\kw,0) to[bend left] (0.75*\kw,\width-1) -- (0.5*\kw,\width-1) -- (0.5*\kw,\width) -- (1.5*\kw,\width) -- (1.5*\kw,\width-1) -- (1.75*\kw,\width-1) to[bend right] (1.75*\kw,0)  (\length-0.25*\kw,0) to[bend left] (\length-0.25*\kw,\width-1) -- (\length-0.75*\kw,\width-1) -- (\length-0.75*\kw,1) -- (\length-1.75*\kw,1) -- (\length-1.75*\kw,\width-1) -- (\length-1.5*\kw,\width-1) -- (\length-1.5*\kw,\width) 
        (2.5*\kw,\width) -- (2.5*\kw,\width-1) -- (3.25*\kw,\width-1) -- (3.25*\kw,1);
        \draw[gen1,very thick] (1.75*\kw,0) -- (4*\kw+1.5,0)  (\length-2*\kw-1.5,0) -- (\length-0.25*\kw,0);
        \draw[gen1,very thick,dotted] (4*\kw+1.5,0) -- (\length-2*\kw-1.5,0); 
        \draw[gen1,very thick] (\length-1.5*\kw,\width) -- (\length-2*\kw-1.5,\width) 
        (4*\kw+1.5,\width) -- (2.5*\kw,\width);
        \draw[gen1,very thick,dotted] (4*\kw+1.5,\width) -- (\length-2*\kw-1.5,\width); 
        \draw[gen1,very thick] (3.25*\kw,1) --  (4*\kw+1.5,1);
        \draw[gen1,very thick, dotted] (4*\kw+1.5,1) -- (4*\kw+2.5,1);
        \draw[loosely dotted] (4*\kw+1.5,\width-1) -- (\length-2*\kw-1.5,\width-1);
        \draw[loosely dotted] (4*\kw+2.5,1) -- (\length-2*\kw-1.5,1);
	\end{scope}
  	\end{tikzpicture}
  	\caption[Spiral shaped walks on a strip]{Spiral shaped walks in $\Pcal$. The dashed rectangles contain the set $K$ and its respective translates under $\tau^i$.}
  	\label{fig:snailwalks}
\end{figure}

Assume that the walk $W$ is not self-avoiding. Then there is some index $i \in [2k+1]$ such that $\tilde{X}_i$ intersects either $\tilde{X}_j$ for $j > i$ or contains an interior point of $\tilde{Y}_j$ for some $j \in [2k+1]$. For $j < i-1$, the walk $\tilde{X}_i$ does not intersect $\tilde{Y}_j$ because $R_j$ does not intersect $T_K(i-1)$. For $j \in \{i-1,i\}$, the walk $\tilde{X}_i$ contains only a single vertex of $R_j$, which is an endpoint of $\tilde{Y}_j$. Therefore $j > i$, and in particular $\tilde{Y}_j$ contains a vertex on level $l_i$. Without loss of generality assume that $j$ is odd, the other case is symmetric. Since $\tilde{Y}_j$ connects levels $l_j$ to $l_{j+1} >l_j$ we conclude that $l_{j+1}\geq l_i \geq l_j$. If $i$ is odd, then
\[
0 \geq l_j-l_i= \sum_{l=i}^{j-1} (-1)^{l-1}n_l = (n_i-n_{i+1})+\dots +(n_{j-2}-n_{j-1}),
\]
so there is some index $l$ such that $n_{l-1} \leq n_l$.
Otherwise $i$ is even and an analogous argument using $0 \leq l_{j+1}-l_i$ leads to the same conclusion.

For the converse implication assume that there is $i \in [2k]$ such that $n_{i+1} \geq n_i$. We claim that the sub-walk $\tilde{X}_i \tilde{Y}_i \tilde{X}_{i+1} \tilde{Y}_{i+1}$ of $W$ is not self-avoiding. Assume without loss of generality that $i$ is odd, the other case is symmetric. Since $n_{i}-n_{i+1} \leq 0$ we know that $l_{i+2} \leq l_i < l_{i+1}$. In particular, both $\tilde X_i$ and $\tilde Y_{i+1}$ contain vertices on level $l_i$. Moreover, by definition of $X_i$ the walk $\tilde{X}_i$ contains a vertex $v$ of $R_{i+1} \cap \tau^{2l_i+1}(K)$. Finally the sub-path $\tilde{Y}_{i+1}$ of $R_{i+1}$ starts at a vertex in $\tau^{2l_{i+1}}(K)$ and ends at a vertex in $\tau^{2l_{i+2}}(K)$ and thus must contain the vertex $v$. We conclude that $W$ is not self-avoiding.

Let us now assume that the language $L_{\SAW,o}$ is $k$-multiple context-free. Recall that the class of $k$-multiple context-free languages is closed under homomorphisms, inverse homomorphisms, and intersection with regular languages. Using these properties and Theorem~\ref{thm:mcfgpump}, we derive a contradiction. First, note that the language 
\[
L(\Pcal_{\SAW})=\{x_1 y_1^{n_1} x_2 y_2^{n_2} \dots x_{2k+1} y_{2k+1}^{n_{2k+1}}: n_1 > n_2 >\dots > n_{2k+1} > 0\}
\]
is the intersection of the regular language $L(\Pcal)$ and the $k$-multiple context-free language $L_{\SAW,o}$ and thus must be $k$-multiple context-free. We define a language homomorphism 
\[
\phi: \{a_1, b_1, \dots, a_{2k+1}, b_{2k+1}\}^* \rightarrow \Si^* 
\]
by setting $\phi(a_i)=x_i y_i^{2k+2-i}$ and $\phi(b_i)= y_i$ for every $i \in [2k+1]$; we point out that $a_i$ and $b_i$ are single letters, whereas $x_i$ and $y_i$ are labels of walks and thus may consist of multiple letters. The language
\[
L_1= \{a_1 b_1^{n_1} \dots a_{2k+1} b_{2k+1}^{n_{2k+1}} : n_1 \geq n_2 \geq \dots \geq n_{2k+1} \geq 0 \} 
\]
is $k$-multiple context-free because 
\[
L_1=\phi^{-1}(L(\Pcal_{\SAW})) \cap \{a_1 b_1^{n_1} \dots a_{2k+1} b_{2k+1}^{n_{2k+1}} : n_1, \dots, n_{2k+1} \in \N_0 \}.
\]
Note that this statement strongly relies on the fact that the edge-labelling $\ell$ is deterministic: the image $\phi(w)$ of any word $w \in \{a_1 b_1^{n_1} \dots a_{2k+1} b_{2k+1}^{n_{2k+1}} : n_1, \dots, n_{2k+1} \in \N_0 \}$ is the label of a unique walk in $\Pcal$ and thus has a unique representation of the form \eqref{eq:spiralwalkwords}, which lies in $L(\Pcal_{\SAW})$ if and only if $n_1 \geq n_2 \geq \dots \geq n_{2k+1} \geq 0$.

Finally, the language $L_2=\{ c_1^{n_1} \dots c_{2k+1}^{n_{2k+1}} : n_1 \geq n_2 \geq \dots \geq n_{2k+1} \geq 0\}$ is the image of $L_1$ under the obvious homomorphism mapping $a_i$ to $\epsilon$ and $b_i$ to $c_i$ and thus must be $k$-multiple context-free, a contradiction to Theorem~\ref{thm:mcfgpump}.
\end{proof}

The property of having a $k$-multiple context-free language of SAWs is closed under taking certain sub-graphs. The following lemma extends~\cite[Lemma 4.2]{Lindorfer2020} to MCFLs; the proof works exactly the same and is thus omitted.
\begin{lem}\label{lem:subgraphmcf}
Let $H$ be a sub-graph of $G$ which is invariant under a subgroup $\Gamma$ of $\AUT(G,\ell)$ acting quasi-transitively on $H$. If $L_{\SAW,o}(G)$ is $k$-multiple context-free, then there is a vertex $o' \in V(H)$ such that $L_{\SAW,o'}(H)$ is also $k$-multiple context-free.
\end{lem}

Knowing that $k$-multiple context-freeness of the language of self-avoiding walks forbids $\tau$-strips of size at least $2k+1$, we are able to prove the main result of this section. Note that the case of thin ends is already taken care of, so we mainly need to deal with thick ends in the proof.

\begin{proof}[Proof of Theorem~\ref{thm:sawlangnotmcf}]
By Lemma~\ref{lem:snailwalks} there is no $\tau \in \Gamma=\AUT(G,\ell)$ such that the graph $G$ contains a $\tau$-strip of size $2k+1$. Lemma~\ref{lem:existtaustrips} yields that all thin ends of $G$ have size at most $2k$, thus $G$ is accessible.

Assume for a contradiction that $G$ contains a thick end. By Theorem~\ref{thm:graphdecomp} there is a tree decomposition $(T,\Vcal)$ of $G$ efficiently distinguishing all ends of $G$ such that there are only finitely many $\Gamma$-orbits on $E(T)$. We have seen in Section~\ref{sec:treedecomp} that $G$ having a thick end implies that there is a vertex $t$ of $T$ such that the part $\Vcal(t)$ is infinite. By Lemma~\ref{lem:partsqtrans} the set-wise stabiliser $\Gamma_{\Vcal(t)} \leq \Gamma$ of $\Vcal(t)$ acts quasi-transitively on this part. 

Let $H$ be the sub-graph of $G$ obtained from the induced sub-graph $G[\Vcal(t)]$ in the following way. For every edge $e$ of $T$ incident to $t$ add for every pair of vertices in the adhesion set $\Vcal(e)$ all shortest walks connecting these vertices. Then $H$ is connected and $\Gamma_{\Vcal(t)}$ acts quasi-transitively on $H$ because it acts with finitely many orbits on the edges of $T$ and thus on the adhesion sets contained in $\Vcal(t)$. By Lemma~\ref{lem:subgraphmcf} there exists a vertex $o'$ of $H$ such that the language $L_{\SAW,o'}(H)$ is $k$-multiple context-free and by Lemma~\ref{lem:existnonell} the subgroup $\Gamma_{\Vcal(t)}$ contains a non-elliptic graph automorphism $\gamma$. As $H$ has only a single (thick) end, $\gamma$ fixes this end and is parabolic. By Lemma~\ref{lem:existtaustrips}, the graph $H$ contains a $\tau$-strip $S$ of size $2k+1$ for some $\tau \in \Gamma_{\Vcal(t)}$. But $S$ is also a $\tau$-strip in the original graph $G$, contradicting Lemma~\ref{lem:snailwalks}. We conclude that all ends of $G$ are thin.
\end{proof}

\bibliographystyle{abbrv}
\bibliography{latex}

\end{document}